\documentclass{report}
\usepackage{amsmath, amssymb, amsthm, relsize, tikz-cd, graphicx, hyperref, url, csquotes}

\theoremstyle{plain}
\newtheorem{theorem}{Theorem}[section]
\newtheorem{lemma}[theorem]{Lemma}

\theoremstyle{definition}
\newtheorem{definition}[theorem]{Definition}
\newtheorem{proposition}[theorem]{Proposition}
\newtheorem{corollary}[theorem]{Corollary}
\newtheorem{remark}[theorem]{Remark}

\DeclareMathOperator*{\OmSum}{\mathlarger{\mathlarger{\Omega}}}
\DeclareMathOperator*{\MhSum}{\mathlarger{\mathlarger{\mho}}}

\newcommand{\tet}{\text{tet}}

\newcommand{\up}{\uparrow}

\begin{document}

\begin{titlepage}
    \begin{center}
        \vspace*{1cm}
            
        \Huge
        \textbf{Asymptotic Solutions of the Tetration Equation}
            
        \vspace{0.5cm}
        \LARGE
        \textit{An analysis in the spirit of James Stirling}
            
        \vspace{1.5cm}
            
        \textbf{James David Nixon}
        
        \Large
        February 2022
            
        \vfill
            
        \vspace{0.8cm}
            
        \includegraphics[scale = 0.6]{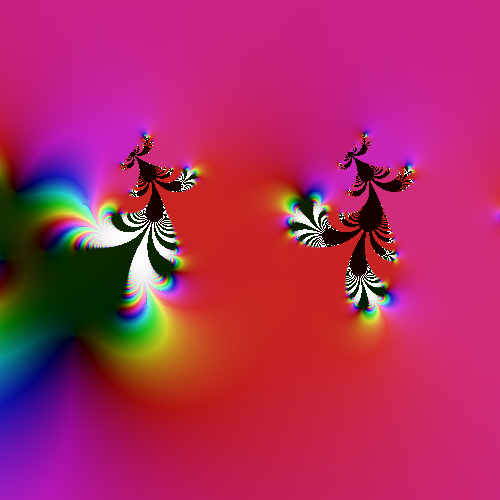}
            
    \end{center}
\end{titlepage}

\begin{abstract}
In this report we construct a family of holomorphic functions $\beta_{\lambda,\mu} (s)$ which behave asymptotically like iterated exponentials as $|s| \to \infty$ in the right half plane. Each $\beta_{\lambda,\mu}$ satisfies a convenient functional relationship with nested exponentials; and has a series expansion that converges in a half-plane. They provide a nearness to the dynamics of the map $e^{\mu z} : \mathbb{C}\to\mathbb{C}$ and behave asymptotically as a fractional iteration would behave. 

These objects are used to describe the various orbits of the exponential function. We describe where Abel equations are feasibly constructed from $\beta$. Where there exists wildly holomorphic functions with period $2 \pi i / \lambda$ that are holomorphic Abel functions of the form $t(s+1) = e^{\mu t(s)}$.

These functions help us see the behaviour of the iterated exponential base $b = e^{\mu}$--the inverse Abel function of $e^{\mu z}$. And are intended as a blue print of how iterated exponentials will behave as we observe them asymptotically. Where at worst, the approximation reduces to an asymptotic series.

For the Shell-Thron region we provide alternative tetrations to the regular iteration method. Init, we allow for a large domain of possible periodic solutions to the tetration equation. This shows many types of holomorphic iterations of the exponential functions that were previously non-extant; though believed to exist.

We close by comparing these iterations to the traditional path of using regular iterations. By such, we explore the rich field of $\theta$ mappings. Where we describe with brief notes how the multiple inverse Abel functions compare to Kneser; and Kneser's counterparts with complex $b$.\\


\end{abstract}

\tableofcontents

\section{Introduction}\label{sec1}
\setcounter{equation}{0}

The goal of this paper is to sequentially approximate tetration--or, more accurately the inverse Abel equation of $z \mapsto e^{\mu z}$ for $\mu \in \mathbb{C}$. The limit of the sequence may or may not be analytic; but if it isn't analytic, it serves as a fantastic asymptotic approximation. We won't concern ourselves excessively with the limit; except describe various properties the limit will have necessarily. But, before we can progress--it's helpful to remind ourselves what tetration is; and what it's good for; why we care.\\

To begin, let $b = e^{\mu}$ and look at what is commonly called the power tower:

\begin{align*}
    & 1\\
    & b\\
    & b^b\\
    & b^{\displaystyle b^b}\\
    & b^{\displaystyle b^{\displaystyle b^b}}\\
    &\vdots\\
    & b \up \up n\\
\end{align*}

Where the functional relationship should be apparent. The central idea and problem of tetration is to analytically continue this function in the variable $n$. So that we have some analytic function $F(s)$ such that $F\Big{|}_\mathbb{N} = b \up \up n$ and $b^{F(s)} = F(s+1)$. There are a couple of problems with this formulation that we should address before continuing.

First of all, we are calling on the function $b^z$; and as we should know, these have countably many solutions for $b \in \mathbb{C}$. Thus, to make this clearer, we are going to make the change of variables $b = e^{\mu}$; and the function we care about is $e^{\mu F(s)} = F(s+1)$. Now if we add $\mu + 2\pi i$, we still have the same natural values of the exponential, but a different branch of $b^z$. The behaviour in the complex plane will be vastly different. This isn't too important, but it's a distinction to recall when iterating exponentials.

Second of all, we should describe the domains of these solutions, and what a potential solution looks like. A tetration function $F$ must be holomorphic almost everywhere on a translation invariant domain $\mathcal{P}$ for it to be of interest of us. This means that for all $k \in \mathbb{Z}$, if $s \in \mathcal{P}$, then so is $s+k \in \mathcal{P}$.

Where additionally, if $\mathcal{E} \subset \mathcal{P}$ and $F$ is holomorphic on $\mathcal{P}/ \mathcal{E}$, then we only care about solutions in which:

$$
\int_{\mathcal{E}}\,dA = 0\\
$$

Where $dA$ is the standard Lebesgue area measure in $\mathbb{R}^2$. This can equivalently be said that there exists no open set within $\mathcal{E}$, and that it is nowhere dense in $\mathcal{P}$. This condition is to avoid purely real solutions; who may be analytic only on the real line; or strange domains. We prefer tetrations which are holomorphic on large domains, so to speak. And we require these domains to be translation invariant up to a measure zero set.

Last but not least, when $\mu \in \mathbb{R}$ and when $s \in \mathbb{R}$ with $\Re(s) > -2$, we would like for $F(s) \in \mathbb{R}$. So, the tetrations we care about must be real-valued. This is the significantly most difficult requirement of the construction. It's relatively simple to construct a tetration function otherwise.

There is an additional distinction to be made. We do not necessitate that a tetration function satisfies $F(0) = 1$. Typically this is a strict requirement on tetration.  As we are primarily interested in solving the equation $e^{\mu F(s)} = F(s+1)$ for arbitrary $\mu,s\in\mathbb{C}$; renormalizing each $F$ so $F(0) = 1$ is a nuissance we don't need. If a tetration function does satisfy $F(0) = 1$; we will call it a normalized tetration.\\

For example, take the fixed point $L \approx 0.31813 + 1.33723i$ of $e^z$, construct its inverse Schr\"{o}der function $\Psi$ (which is entire), then:

\begin{equation}\label{eq:1}
F(s) = \Psi(e^{L(s-s_0)})
\end{equation}

Where $\Psi$ is the unique entire function which satisfies:

$$
\Psi(0) = L\,\,\,\,\,\Psi'(0) = 1\,\,\,\,\,e^{\Psi(\xi)} = \Psi(e^L\xi)\\
$$

This will be a normalized tetration function for an appropriate $s_0$. This is an elementary exercise in superfunction theory/complex dynamics. Although this is a tetration function; it isn't real valued; and suffers from a plethora of non-uniqueness properties. There are countably infinite $s_0$ which provide a normalized tetration; and there's no obvious choice which is more natural.

For details on the construction of this function, we refer to \cite{milnor_2000,devaney_2021}. The construction of the Schr\"{o}der function is done in different manners in each; there are many such methods of construction.\\

From here we can note some fantastic results on tetration, which date very far back. The first being Kneser's observation that a Riemann mapping can make the above tetration in \eqref{eq:1} viable. Kneser's paper \cite{kneser_1950} is in German and there exists no English translation; but we cite his work for historical purposes. Kneser's construction is a very delicate procedure, but boils into solving for a $1$-periodic function $\theta$, such that,

$$
\theta(s) = \sum_{k=0}^\infty c_k e^{2 \pi i k s}\\
$$

And,

\begin{equation}\label{eq:Kne}
\tet_K(s) = \Psi(e^{L(s + \theta(s)-s_0)})\,\,\text{for}\,\,\Im(s) > 0\\
\end{equation}

This is apocryphally known as Kneser's tetration (Kneser would formulate this definition considerably different, but end up with the same result). Which is the first and foremost tetration that exists. It satisfies the tetration functional equation $F(s+1) = e^{F(s)}$, and additionally the conjugate property:

$$
\overline{\tet_K(s)} = \tet_K(\overline{s})\\
$$

Which is constructed by doing a similar procedure with $\overline{L}$ and its Schr\"{o}der function. This tetration will be a bijection from $(-2,\infty) \to \mathbb{R}$. It's set of discontinuities will be $\mathcal{E}_K = (-\infty,-2]$, which certainly satisfies our domain conditions--it is measure zero under the Lebesgue area measure in the translation invariant domain $\mathbb{C}$.

From here, thanks to the work of two people, we have a generalization of these results. Paulsen and Cowgill showed through meticulous reshaping of Kneser's construction (and considerable simplification), that this result can be extended to $b > e^{1/e}$ \cite{paulsen_cowgill_2017}; and also for a large data set of complex values, this time due solely to Paulsen \cite{paulsen_2018}. Much of their work comes off as a slicker, more sleet, version of Kneser's construction. I highly suggest the pair of papers. This is a result that's been in the air for a while, as for $b > e^{1/e}$ Kneser's construction works just as well.

From here, it is worthy to mention the work of Dmitrii Kouznetsov. He is the first person to produce an actual textbook on the theories of superfunctions \cite{kouznetsov_2020}. And, fitting neatly in this, is tetration. Where, it represents a very hard problem within a hard problem.

Kouznetsov chose a very different path in his construction of tetration. We will only sketch it briefly, and we'll focus on the case $\mu =1$ ($b = e$). Firstly, he develops a series:

$$
P_m(s) = L + \sum_{j=1}^m a_j e^{jLs}\\
$$

And takes the limit,

\begin{equation}\label{eq:Kou}
\tet_{K}(s) = \lim_{n\to\infty} \exp^{\circ n} P_m(s-n)\,\,\text{for}\,\,\Im(s) > 0\\
\end{equation}

And, per Paulsen and Cowgill, and a uniqueness condition they derive; this and Kneser's tetration are one and the same. But the miracle is that this construction converges so well. 

All three of these authors choose to use the race-track method of calculating tetration. This is a contour integration trick, which allows for fast and accurate calculations, that is particular to tetration. I highly suggest reading Kouznetsov's book, and his various works; where he is the inventor of the race-track method.\\

The author is not particular to this computation method. For the most versatile tetration calculator; designed after Kneser's method; the author turns to Sheldon Levenstein. Although not a published result in a scientific journal, Levenstein's Pari-GP calculator for tetration is full of a plethora of novel results; and a far better design system; allowing for the various Kneser tetrations to be graphed and calculated to arbitrary precision.

Levenstein's work is partly in collaboration with the many users of the tetration forum \cite{tetration_forum}; but the heavy lifting is largely handled by him. If you are curious as to what tetration is, and what it looks like, the author definitely suggests Levenstein's fatou.gp program for Pari-GP \cite{levenstein_2011}.\\

At this point, it is worth noting that tetration isn't always hard. Sometimes it's, frankly, rather simple. If I take $0 < \mu < 1/e$ and $b \in (1, e^{1/e})$--then the exponential $e^{\mu z}$ has a real attracting fixed point $1 < \omega < e$ with multiplier $0<\mu \omega<1$. Implying it has a Schr\"{o}der function $\varphi$ in a neighborhood of the fixed point. And we get a function:

$$
\tet_\mu(s) = \varphi(e^{\log(\omega \mu)(s-s_0)})\\
$$

Where the domain of discontinuity is,

$$
\mathcal{E}_\mu = \{s \in \mathbb{C}\,|\, \Re(s) < -2,\,\Im(s) = 2 \pi ik/\log(\omega \mu),\,k \in \mathbb{Z}\}
$$

We cite Henryk Trappman and Kouznetsov's paper \cite{kouznetsov_trappmann_2010} on the various types of super functions that $\sqrt{2}$ induces; where among them is the above tetration function. Though, this is a well known result, and is not difficult to prove (relative to the Kneser construction). In this paper they focus on Kouznetsov's idea of a regular iteration. In this report then, we will construct many \textit{irregular} iterations.\\ 

With all these results on tetration, it's important to ask what's new to be shown--what's really left. There are many things new to be shown in tetration; it is still a blossoming field, even if the golden apples are already plucked. To draw one from the hat, we'd point out that tetration in the Shell-Thron region is little understood. The Shell-Thron region is defined as:

$$
\mathfrak{S} = \{b \in \mathbb{C}\,|\,\lim_{n\to\infty}\exp_b^{\circ n}(1)\,\,\text{converges}\}\\
$$

This produces the region in Figure \ref{fig:Sh-Th}.

\begin{figure}
    \centering
    \includegraphics[scale=0.5]{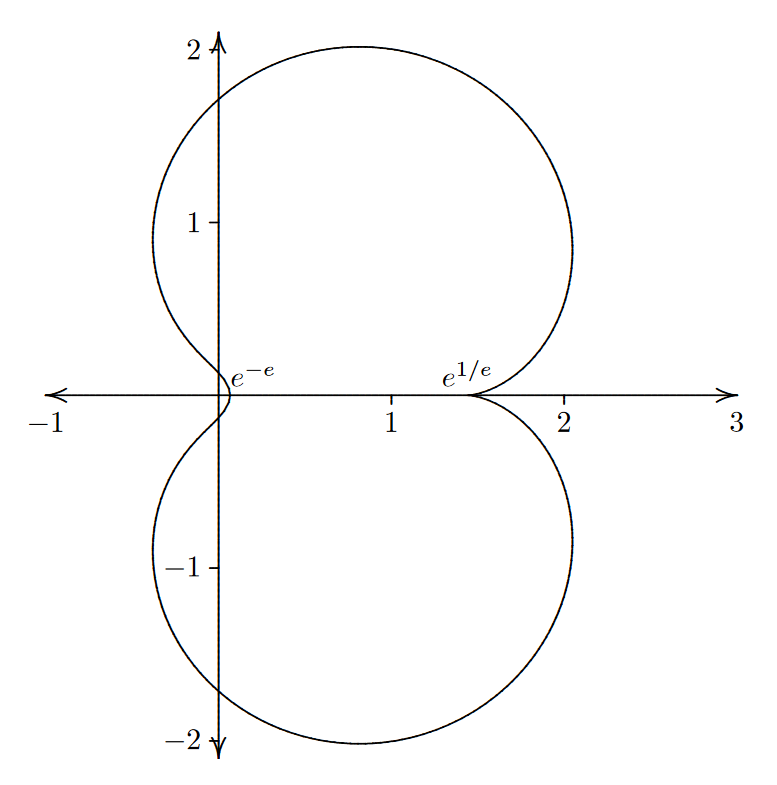}
    \caption{A picture of the Shell-Thron region, courtesy of William Paulsen \cite{paulsen_2018}.}
    \label{fig:Sh-Th}
\end{figure}

Constructing tetration in this area should be easier; but it's very anomalous, especially if you try to use the techniques Kneser (and his genealogy) has provided. In this way, it's less than desirable. A similar problem arises for extreme values: how can we take the tetration of $\mu = -1000$; or when $b$ is really close to zero? Or, how do we take tetration when $\mu = 1000 + \pi i$, when $b$ is way out in the left half plane? There's little literature on such cases. And as fantastic as Levenstein's program is; you can see the gears grinding together for these extremes.

Now, we have not come here to solve these cases. But instead, create asymptotic expansions, which gives the shape, structure, and look and feel of tetration; while being slightly more manageable. But nonetheless; we will still be able to produce tetration functions; but they may be less than what we desire. Though they certainly submit themselves to our tetration criteria. But ultimately we will be constructing asymptotic results which are occasionally holomorphic.\\

This paper is done in the spirit of James Stirling's asymptotic expansion of the Gamma function \cite{remmert_1998}. Where it stands as one of the first recorded asymptotic series. Where, upon first inspection, Stirling was confident the expression converged and was a new expression for the Gamma function. But, instead, it was a very accurate asymptotic series.

We are attempting to do something similar. Where as: Stirling's asymptotic expansion is derived for the recursive equation $\Gamma(s+1) = s \Gamma(s)$; ours is derived from the recursive equation $F(s+1) = e^{\mu F(s)}$. We won't be using any of his machinery; but the author encountered the same conundrum as Stirling. The author was equally convinced of holomorphy--where it is not always the case; and in fact we have only developed successive approximations. The author owes Sheldon Levenstein many thanks for providing counter examples; and helping the author pull out the actual truth.

And here, is the perfect moment to branch off into what we'll be studying in this paper. Our little corner of the mathematical world known as infinite compositions.

\chapter{A Crash Course In Infinite Compositions}

\section{A Crash Course In Infinite Compositions}\label{sec2}
\setcounter{equation}{0}

What is an infinite composition? The author finds himself constantly re-explaining these concepts to every mathematician he meets. Either the definition doesn't stick, or the importance is shrugged off. Sometimes it's understood, but it's goals are misinterpreted. And sometimes, people just flat out do not comprehend the normality theorems which allow for infinite compositions. So, to begin, we shall go very slow. Much of the author's work is unpublished; and exists solely on arXiv. We shall not reference anything from these works, though we will summarize a couple of core concepts. For more details on this work, see the author's arXiv page \cite{nixon_2022}.

Consider a sequence of functions $q_j(z) : \mathcal{G} \to \mathcal{G}$ for $\mathcal{G} \subseteq \mathbb{C}$ a domain in $\mathbb{C}$. The pivotal question of infinite compositions, is where do the following nested compositions converge?

\begin{align*}
    &\lim_{n\to\infty}q_1(q_2(...q_n(z)))\\
    &\lim_{n\to\infty} q_n(q_{n-1}(...q_1(z)))\\
\end{align*}

The first case is known as Inner Compositions, and the second as Outer Compositions--as to whether we add terms on the inside or the outside. Now, luckily, we can throw away the second case--we don't care about outer compositions. The outer case has its interesting properties, but they're irrelevant to tetration. Unfortunately, the inner case, the first case, despite being the more interesting case, is the more difficult case.

Now, notationally it can be very frustrating to consistently write this limit notation. And for that, the author chose to write this in an Euler style summation notation. And for that, we restrict the symbols $\Omega$ and $\mho$ to represent inner or outer. This is to say that:

\begin{align*}
    \OmSum_{j=n}^m q_j(z) \bullet z &= q_n(q_{n+1}(...q_m(z)))\\
    \MhSum_{j=n}^m q_j(z) \bullet z &= q_m(q_{m-1}(...q_n(z)))\\
\end{align*}

And now, you may be wondering, why the $\bullet z$? Well, an important feature of infinite compositions is that we nest across a certain variable. The $\bullet z$ is essentially telling the reader to compose the functions $q_j$ across the variable $z$. What if $q_j$ depended on two variables, say $\lambda , z$? Without the $\bullet z$, look what happens when I write:

$$
\OmSum_{j=n}^m q_j(\lambda , z)\\
$$

Do I nest the compositions through $\lambda$, or do I nest the compositions through $z$? The reader has the potential to misread and think this function is actually:

$$
\OmSum_{j=n}^m q_j(\lambda , z) = q_n(q_{n+1}(...q_m(\lambda,z),...,z),z)\\
$$

Whereas, if I write with the bullet, we know that:

\begin{align*}
\OmSum_{j=n}^m q_j(\lambda, z) \bullet z &= q_n(\lambda,q_{n+1}(\lambda,...q_m(\lambda,z)))\\
\OmSum_{j=n}^m q_j(\lambda, z) \bullet \lambda &= q_n(q_{n+1}(...q_m(\lambda,z),...,z),z)\\
\end{align*}

This works little differently than Leibniz's notation for integration. When we call upon $\OmSum$, we have to make sure we declare the variable we are composing across. This is to say, if I write $\int e^{sy}$ are we integrating across $s$, or across $y$? The reader has to guess, whether from context or sheer ignorance.

Now, we will reserve the variable $z$ to the $\bullet z$; and we'll stay as consistent as possible. This is why we've chosen to refer to the exponential as $e^{\mu z}$ with $z$ rather than $s$; because we will be taking infinite compositions across this variable; and a modified exponential will take the role of $q$. By which, $\bullet z$ is the domain of the iterate, and $s$ is the domain of the iteration.

It is hereupon, where the author must recite much of his work in previous papers. These are just some of the normality theorems of infinite compositions the author has gathered. We won't prove these heuristics excepting the one we actually use for tetration. But much of the heuristics fit together; so it's important to mention them.

To begin, lets define the function:

$$
Q(z) = \OmSum_{j=1}^\infty q_j(z) \bullet z\\
$$

And ask when it is holomorphic in $z$. The first result; if for all compact disks $\mathcal{K}\subset \mathcal{G}$, we get:

$$
\sum_{j=1}^\infty \left|\left|q_j(z) - z\right|\right|_{\mathcal{K}} = \sum_{j=1}^\infty \sup_{z \in \mathcal{K}} \left| q_j(z) - z \right| < \infty\\ 
$$

Then $Q(z)$ is holomorphic for $z \in \mathcal{G}$. This can be thought of, a little fluidly, as: if $q_j(z) \to z$ as $j\to\infty$, and does so in a normally summable manner, then their infinite composition is a holomorphic function in $z$. Which is, if $q_j(z) - z$ is compactly normally summable, then $Q$ is holomorphic.

We don't need this result. But it begins to paint a picture of checking a sum to ensure an infinite composition is holomorphic. The next case is much more important to us. Let's assume that $q_j(z) \to A$ where $A \in \mathcal{G}$ is constant. Where and when does $Q(z)$ converge?

Here, we have what I like to call the ``degenerate case." This is when the final function $Q(z)$ will be holomorphic, but it will be constant. There's a similar theorem which accounts for these cases, only slightly changed. Taking $\mathcal{K}$ to be a compact disk in $\mathcal{G}$ again, if for all $\mathcal{K}$:

$$
\sum_{j=1}^\infty \left| \left| q_j(z) - A \right| \right|_{\mathcal{K}} < \infty\\
$$

Then $Q(z) = Q$ is a constant function for $z \in \mathcal{G}$. And here is where we have to enter a more difficult discussion that we ignited earlier. What happens when $q$ depends on another variable other than $z$? Now, this requires a turn of the key, because we still need this to be comparable to a sum to show convergence, but how to do this efficiently?

We're going to consider $q_j(s,z)$ for $s \in \mathcal{S} \subset \mathbb{C}$. And, for convenience we will assume the set $\mathcal{S}$ is compact. And additionally, we will say a function $f(s)$ is holomorphic on $\mathcal{S}$ if it is holomorphic on its interior. So holomorphy on a compact set means it's holomorphic on its interior.

The author spent a very long time detailing that:

$$
Q(s,z) = \OmSum_{j=1}^\infty q_j(s,z)\,\bullet z\\
$$

Has different modes of convergence, depending on how $q_j$ behaves depending on $s$ or $z$. If this portion of the report interests the reader we suggest our arXiv repository \cite{nixon_2022}. We only need to detail one mode of convergence in this paper. So let us call the summability criterion that we care about:

$$
\sum_{j=1}^\infty \left| \left|q_j(s,z) - A \right| \right|_{\mathcal{S},\mathcal{K}} =  \sum_{j=1}^\infty \sup_{s \in \mathcal{S},z \in \mathcal{K}}\left|q_j(s,z) - A \right|< \infty\\
$$

Which will gives us a holomorphic function $Q(s,z)$ which is holomorphic in $s$ and constant in $z$. The summation of the next section is to prove this theorem rigorously.

\section{Our Infinite Composition Theorem}\label{sec3}
\setcounter{equation}{0}

Before beginning this theorem, we would like the reader to understand that this is soup from concentrate. For a more organic, well thought out, long explanation, we refer to our previous work; where we spent more time detailing how and why this happens. For the moment, all that's needed is the following theorem. As per that, the proof can seem blunt and uninspiring.

I chose not to break it into little lemmas and the such, because it's flow is as it should be. And the theorem and its proof are clear as it stands. We are trying to relate a summation condition to an infinite composition. And if you can accept this theorem, most of the rest of this paper will be expected consequences of this $\OmSum \to \sum$ comparison.

\begin{theorem}[Our Infinite Composition Theorem]\label{thmINF}
Let $\{q_j(s,z)\}_{j=1}^\infty$ be a sequence of holomorphic functions such that $q_j(s,z) : \mathcal{S} \times \mathcal{G} \to \mathcal{G}$ where $\mathcal{S} \subset \mathbb{C}$ is compact, and $\mathcal{G}$ is a domain in $\mathbb{C}$. Suppose there exists some $A \in \mathcal{G}$, such for all compact sets $\mathcal{K}\subset\mathcal{G}$, the following sum converges,

$$
\sum_{j=1}^\infty ||q_j(s,z) - A||_{z \in \mathcal{K},s \in \mathcal{S}} = \sum_{j=1}^\infty \sup_{z \in \mathcal{K},s \in \mathcal{S}}|q_j(s,z) - A| < \infty
$$

Then the expression,

$$
Q(s,z) = \lim_{n\to\infty}\OmSum_{j=1}^n q_j(s,z)\bullet z = \lim_{n\to\infty} q_1(s,q_2(s,...q_n(s,z)))\\
$$

Converges uniformly for $s \in \mathcal{S}$ and $z \in \mathcal{K}$ as $n\to\infty$ to $Q$, a holomorphic function in $s\in\mathcal{S}$, constant in $z$.
\end{theorem}

\begin{proof}

The first thing we show is for all $\epsilon > 0$, there exists some $N$, such when $m \ge n  > N$,

$$
|\OmSum_{j=n}^{m} q_j(s,z)\bullet z - A| < \epsilon
$$

For $z$ in $\mathcal{K}\subset \mathcal{G}$ (where $A$ is in the open component of $\mathcal{K}$), and $s\in\mathcal{S}$. This then implies as we let $m\to\infty$, the tail of the infinite composition stays bounded. Forthwith, the infinite composition becomes a normal family, and proving convergence becomes simpler. We provide a quick proof of this inequality.\\

Set $||q_j (s, z)-A||_{\mathcal{S}, \mathcal{K}} = \rho_j$. Pick $\epsilon > 0$, and choose $N$ large enough so when $n > N$,

$$
\rho_n < \epsilon
$$

Such that the domain $\{|z-A| < \epsilon\} \subset \mathcal{K}$. Denote: $\phi_{nm}(s, z) =\OmSum_{j=n}^m q_j (s, z) \bullet z = q_n(s, q_{n+1}(s, ...q_m(s, z)))$. We go by induction on the difference $m-n = k$. When $k=0$ then,

$$
||\phi_{nn}(s,z) - A||_{\mathcal{S},\mathcal{K}} = ||q_n(s,z)-A||_{\mathcal{S},\mathcal{K}}= \rho_n < \epsilon
$$

Assume the result holds for $m-n < k$, we show it holds for $m-n = k$. Observe,

\begin{eqnarray*}
||\phi_{nm}(s,z)-A||_{\mathcal{S},\mathcal{K}} &=& ||q_n(s,\phi_{(n+1)m}(s,z)) - A||_{\mathcal{S},\mathcal{K}}\\
&\le& ||q_n(s,z)-A||_{\mathcal{S},\mathcal{K}}\\
&=& \rho_n < \epsilon\\
\end{eqnarray*}

Which follows by the induction hypothesis because $\phi_{(n+1)m}(s,z) \subset \mathcal{K}$--it's in a $\epsilon$-neighborhood of $A$ which is in $\mathcal{K}$. That is $m-n-1 < k$; that's part of the inductive proof. Therefore, for all $m \ge n > N$:

$$
||\phi_{nm}(s,z) - A||_{\mathcal{S},\mathcal{K}} < \epsilon
$$\\

The next step is to observe that $\OmSum_{j=1}^m q_j(s,z)\bullet z$ is a normal family as $m\to\infty$, for $z \in \mathcal{K}$ and $s \in \mathcal{S}$. This follows because the tail of this composition is bounded. We can say:

$$
\OmSum_{j=1}^m q_j(s,z) \, \bullet z = \phi_{1m}(s,z) = \phi_m(s,z) = \phi_N(s,A+\epsilon)\\
$$

Specifically:

$$
\left|\left|\OmSum_{j=1}^m q_j(s,z) \, \bullet z\right|\right|_{\mathcal{S},\mathcal{K}}  = \left|\left|\OmSum_{j=1}^N q_j(s,z) \, \bullet z\right|\right|_{\mathcal{S},|z-A|<\epsilon} \le M\\
$$

Where $\epsilon \to 0$ as $N\to\infty$. Since $\phi_m(s,z) = \OmSum_{j=1}^m q_j(s,z)\bullet z$ are a normal family for all compact sets $\mathcal{K}\subset \mathcal{G}$ (they're bounded with the supremum norm); there is some constant $M \in \mathbb{R}^+$ and $L \in \mathbb{R}^+$ such,

\[
\left|\left|\frac{d^k}{dz^k} \phi_m(s,z) \right|\right|_{\mathcal{S},\mathcal{K}} \le M \cdot k! \cdot L^k
\]\\

To see this, take $|z-A| < 2\delta$ and observe,

\[
\frac{d^k}{dz^k} \phi_m(s,z) = \frac{k!}{2\pi i}\int_{|\xi - A| = 2\delta} \frac{\phi_m(s,\xi)}{(\xi - z)^{k+1}}\,d\xi\\
\]

Taking the supremum norm across $|z-A| \le \delta$ and $s \in \mathcal{S}$,

\begin{eqnarray*}
||\frac{d^k}{dz^k} \phi_m(s,z)||_{\mathcal{S},|z-A| \le \delta} &\le& \frac{k!}{2\pi} \int_{|\xi-A| = 2\delta} \frac{||\phi_m(s,\xi)||_{\mathcal{S}}}{|\xi-z|_{|z-A| \le \delta}^{k+1}}\,d\xi\\
&\le& \frac{k!}{2\pi} \int_{|\xi-A| = 2\delta} \frac{M}{\delta^{k+1}}\,d\xi\\
&\le& \frac{2 M k!}{\delta^{k}}\\  
\end{eqnarray*}

Where we've used the bound $|\xi - z| \ge \delta$ when $|\xi - A| = 2\delta$ and $|z-A| \le \delta$. This bound can be derived regardless of $\mathcal{K}$ for varying $M$ and $L$.\\

Secondly, using Taylor's theorem,

\begin{align*}
\phi_{m+1}(s,z) &- \phi_m(s,z) = \phi_m(s,q_{m+1}(s,z)) - \phi_m(s,z)\\
&= \sum_{k=1}^\infty \frac{d^k}{dz^k} \phi_m(s,z) \frac{(q_{m+1}(s,z) - z)^k}{k!}\\
&= (q_{m+1}(s,z) - z) \sum_{k=1}^\infty \frac{d^k}{dz^k} \phi_m(s,z) \frac{(q_{m+1}(s,z) - z)^{k-1}}{k!}\\
\end{align*}

So that, setting $z=A$,

\begin{align*}
||\phi_{m+1}(s,A) &- \phi_m(s,A)||_{s \in\mathcal{S}} \le ||q_{m+1}(s,A) - A||_{s\in\mathcal{S}} \sum_{k=1}^\infty M L^k ||q_{m+1}(s,A) - A||^{k-1}\\
&\le ||q_{m+1}(s,A) - A||_{\mathcal{S}} \frac{ML}{1-p}\\
\end{align*}

For $L||q_{m+1}(s,A) - A||_{\mathcal{S}} \le p <1 $, which is true for large enough $m>N$. Setting $C = \frac{ML}{1-p}$. Applying from here,

\[
||\phi_{m+1}(s,A) - \phi_m(s,A)||_{s \in \mathcal{S}} \le C ||q_{m+1}(s,A) - A||_{s \in \mathcal{S}}\\
\]

This is a convergent series per our assumption. Choose $N$ large enough, so that when $m,n>N$,

\[
\sum_{j=n}^{m-1}||q_{j+1}(s,A) - A||_{s \in \mathcal{S}} < \frac{\epsilon}{C}\\
\]

Then,

\begin{eqnarray*}
||\phi_{m}(s,A) - \phi_n(s,A)||_{s \in \mathcal{S}} &\le& \sum_{j=n}^{m-1} ||\phi_{j+1}(s,A) - \phi_j(s,A)||_{s \in \mathcal{S}}\\
&\le& C\sum_{j=n}^{m-1}||q_{j+1}(s,A) - A||_{s \in \mathcal{S}}\\
&<& \epsilon
\end{eqnarray*}

So we can see $\phi_m(s)$ must be uniformly convergent for $s \in \mathcal{S}$, and therefore defines a holomorphic function $Q(s)$ as $m\to\infty$.

This tells us,

\[
Q(s) = \OmSum_{j=1}^\infty q_j(s,z)\bullet z \Big{|}_{z=A}\\
\]

Converges and is holomorphic. To show this function equals,

\[
\OmSum_{j=1}^\infty q_j(s,z)\bullet z
\]

For all $z \in \mathcal{G}$; simply notice that,

\[
\OmSum_{j=m}^\infty q_j(s,z)\bullet z \to A
\]

It does so uniformly. Since this is arbitrarily close to $A$ as we let $m$ grow (which was shown at the beginning of this proof). Then,

\begin{eqnarray*}
\OmSum_{j=1}^\infty q_j(s,z)\bullet z &=& \OmSum_{j=1}^{m-1} q_j(s,z)\bullet \OmSum_{j=m}^\infty q_j(s,z)\bullet z\\
&=& \lim_{m\to\infty} \OmSum_{j=1}^{m-1} q_j(s,z)\bullet \lim_{m\to\infty} \OmSum_{j=m}^\infty q_j(s,z)\bullet z\\
&=& \OmSum_{j=1}^\infty q_j(s,z)\bullet z\Big{|}_{z=A}\\
\end{eqnarray*}
\end{proof}

I have written this theorem as effectively as I can. It is not an overly difficult theorem. But it does make one scratch their head. And, to repeat myself: this is the only infinite composition theorem we need. Much of the heuristics are helpful; but in the dirt and grit, this is all we need.

To summarize in a common manner, we return to the discussion of sums. If we have a sequence of functions $q_j(s,z)$ such that $q_j(s,z) \to A$ as $j \to \infty$; but additionally, this is a summable limit; which means:

$$
\sum_{j=1}^\infty q_j(s,z) - A \,\,\text{converges}\\
$$

And additionally--this sum converges compactly normally--then the inner infinite composition is a holomorphic function in $s$. And this is truly all we need the reader to understand about infinite compositions. Though they are a deep subject, we need not go too deep to get a nugget of gold.

\section{The First Construction of $\beta$}\label{sec4}
\setcounter{equation}{0}

As with any good function in mathematics, there should be multiple avenues of its construction. And as with the function we are about to produce, there are two manners of construction that are only slightly different. At least, two obvious manners. We will focus specifically on one manner first, and the next in the next section. This method, although a tad taxing, is mathematically simpler. The next method is computationally quicker.

So, let us begin by fiddling with a sequence of functions $q_j(s,z)$. Let us state that:

$$
q_{j+1}(s,z) = q_j(s-1,z)\\
$$

Let's additionally assume that:

$$
\sum_{j=1}^\infty \left| \left | q_j(s,z) \right| \right|_{\mathcal{S},\mathcal{K}} < \infty\\
$$

For all $\mathcal{K} \subset \mathbb{C}$ a compact disk, and for all $\mathcal{S} \subset \mathbb{C}$ compact. Now, let's look at $Q$; it'll be a little generous with these additions. 

\begin{align*}
    Q(s+1) &= \OmSum_{j=1}^\infty q_j(s+1,z) \bullet z &\\
    &= \OmSum_{j=1}^\infty q_{j-1}(s,z)\bullet z\,\,&\text{--}q\text{'s function identity}\\
    &= \OmSum_{j=0}^\infty q_{j}(s,z)\bullet z\,\,&\text{--we re-indexed the composition, here}\\
    &= q_0(s,\OmSum_{j=1}^\infty q_j(s,z) \bullet z)\,\,&\text{--associative property}\\
    &= q_0(s,Q(s)) &\\
\end{align*}

So we can see that $Q$ has a rather nice functional equation. From here, we choose a function $q_0(s,z)$ which looks like the exponential function $e^{\mu z}$ as $s \to \infty$; and let $q_j(s,z) = q_0(s-j,z)$. We can choose a vast amount of functions for this; but the one we will choose, is inspired by the logistic function.

$$
q_0(s,z) = \frac{e^{\mu z}}{1+e^{-\lambda s}}\\
$$

Where we restrict $\Re \lambda >0$ so that the summation condition converges. Now you may be wondering why we introduce the variable $\lambda$, and the answer is simple. It produces a different mode of convergence. The variable $\lambda$ determines how fast we converge. It also induces a period in $s$ of $2 \pi i/\lambda$; we'd like to be able move the period around later in this paper. This function also has singularities at $(2k+1)\pi i/\lambda$ for $k \in \mathbb{Z}$; and we'd like to be able to move these around too. 

From here, we can now introduce the $\beta$ function:

\begin{equation}\label{eq:beta}
    \beta_{\lambda, \mu} (s) = \OmSum_{j=1}^\infty \frac{e^{\mu z}}{1+e^{\lambda(j-s)}}\bullet z\\
\end{equation}

Which satisfies the functional equation:

$$
\beta_{\lambda,\mu} (s+1) = \frac{e^{\mu \beta_{\lambda , \mu}(s)}}{1 + e^{-\lambda s}}\\
$$

We write this in a quick theorem:

\begin{theorem}[The $\beta$ Theorem]
The function:

$$
\beta_{\lambda, \mu} (s) = \OmSum_{j=1}^\infty \frac{e^{\mu z}}{1+e^{\lambda(j-s)}}\bullet z\\
$$

Is holomorphic for $ \mu \in \mathbb{C}$ and $(s,\lambda) \in \mathbb{L} = \{(s,\lambda) \in \mathbb{C}^2\,|\,\Re\lambda > 0,\, \lambda(j-s) \neq (2k+1)\pi i,\, j,k \in \mathbb{Z},\, j\ge 1\}$. And this function satisfies the identity:

$$
\beta_{\lambda,\mu} (s+1) = \frac{e^{\mu \beta_{\lambda , \mu}(s)}}{1 + e^{-\lambda s}}\\
$$
\end{theorem}

\begin{proof}
The sum:

$$
\sum_{j=1}^\infty \frac{e^{\mu z}}{1+e^{\lambda(j-s)}}\\
$$

Converges compactly normally for $z,\mu \in \mathbb{C}$ and $(s,\lambda) \in \mathbb{L}$. Therefore, by Our Infinite Composition Theorem \ref{thmINF}, the infinite composition must converge uniformly on these domains. The functional equation speaks for itself.
\end{proof}

You can see quite clearly how we can make cake work of a lot of infinite compositions by comparing them to sums. We can now state the first theorem which is really home to this paper. It is precisely the first form of our asymptotic theorems. 

\begin{theorem}[The Asymptotic Theorem of the first kind]\label{thmASYM1}
The function $\beta_{\lambda,\mu}(s)$ satisfies the asymptotic equation:

$$
\ln \beta_{\lambda, \mu} (s+1) / \mu = \beta_{\lambda,\mu}(s) + \mathcal{O}(e^{-\lambda s})\,\,\text{as}\,\,|s| \to \infty\,\,\text{while}\,\, |\arg(\lambda s)| < \pi /2\\
$$
\end{theorem}

\begin{proof}
Since,

$$
\ln \beta_{\lambda , \mu}(s+1) / \mu = \beta_{\lambda , \mu}(s) - \ln(1+e^{-\lambda s})/\mu\\
$$

And since $\ln(1+e^{-\lambda s}) = \mathcal{O}(e^{-\lambda s})$ on the specified domain, the result.
\end{proof}

Now, this is very important to us for a couple of reasons. The most prominent being, this function gets arbitrarily close to satisfying the tetration functional equation:

$$
\ln F(s+1)/\mu = F(s)\\
$$

So the further to the right we move, the more we unlock a closeness to tetration. We've attached here a couple of graphs. It is important to note that these functions have volatile behaviour where they get arbitrarily large, and then get arbitrarily close to zero. For this sake, in the hue plots, I have set overly large values to default to the value $0$ as to not cause overflow errors. It will still uncover the shape of tetration. So, just know that a pitch black is either zero or an overflow in this hue mapping. Unfortunately, there's not much to be done in these cases without a supercomputer.

\begin{figure}
    \centering
    \includegraphics[scale=0.45]{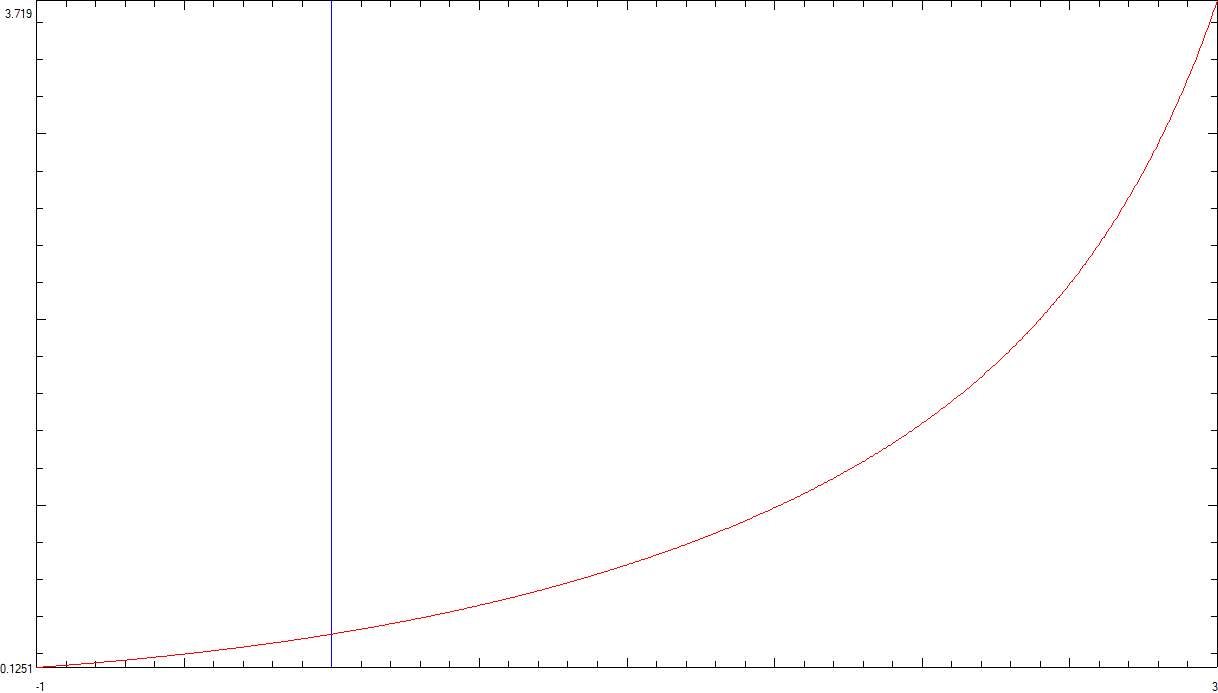}
    \caption{This is the function $\beta_{1,1}(x)$ for $-1 \le x \le 3$. This is just before it starts to explode and show tetration level growth.}
    \label{fig:beta_1_1_real}
\end{figure}

\begin{figure}
    \centering
    \includegraphics[scale = 0.3]{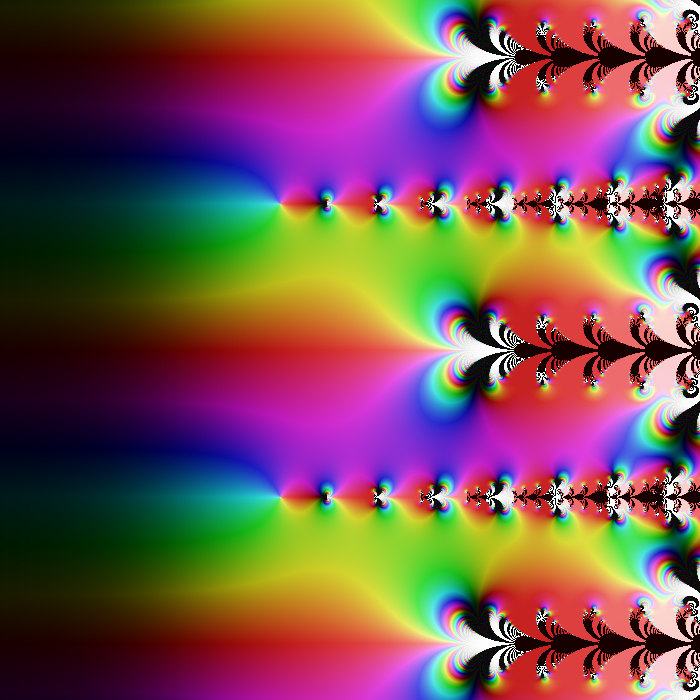}
    \caption{This is the function $\beta_{1,1}(s)$ for $-5 \le \Re(s) \le 10$ and $|\Im(s)| \le 7.5$. You can see the period and the super-exponential nature. You can also see the poles coagulating at $\Im(s) = (2k+1)\pi$. Recall that overflow errors are assigned a zero; a black pixel.} 
    \label{fig:beta_1_1}
\end{figure}

\begin{figure}
    \centering
    \includegraphics[scale = 0.3]{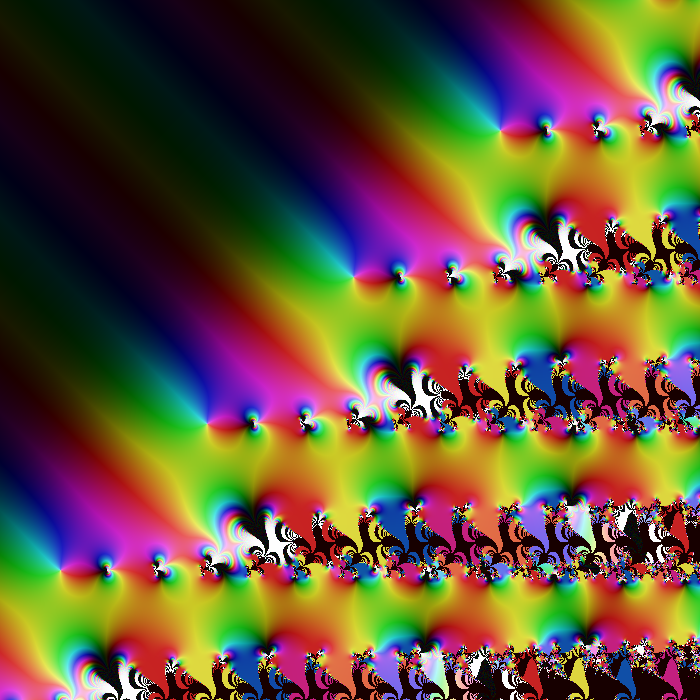}
    \caption{This is the function $\beta_{1+i,1+i}(s)$ for $-5 \le \Re(s) \le 10$ and $|\Im(s)| \le 7.5$. You can see the period and the super-exponential nature. You can also see the poles coagulating at $(1+i) (s-j) = (2k+1)\pi i$. Recall that overflow errors are assigned a zero; a black pixel.}
    \label{fig:beta_1_I}
\end{figure}

\section{The Second Construction of $\beta$}\label{sec5}
\setcounter{equation}{0}

Although we've successfully constructed $\beta$, there is still the trouble of accurately calculating it. For this, we turn to the most effective manner of programming in the $\beta$ function. This method was spawned by an observation of Sheldon Levenstein; in which he employed a clever change of variables. For this, we'll return to the infinite composition notation and make a change of variables.

To begin, let's let $\log(w)/\lambda = s$. Then our infinite composition changes its flavour a tad:

\begin{align*}
    \beta_{\lambda , \mu}(s) &= \OmSum_{j=1}^\infty \frac{e^{\mu z}}{e^{\lambda(j-s)} + 1}\,\bullet z\\
    &= \OmSum_{j=1}^\infty \frac{e^{\mu z}}{\frac{e^{\lambda j}}{w} + 1}\,\bullet z\\
    &= g_{\lambda , \mu}(w)\\
\end{align*}

Now, intrinsic to $g_{\lambda,\mu}(w)$ is that it's holomorphic for $w \neq -e^{\lambda j}$ for $j \ge 1$, $\Re(\lambda) > 0$ and $\mu \in \mathbb{C}$. We can see this either by the change of variables, or by simply showing the infinite composition converges compactly normally on this domain. Hidden in this is that $g_{\lambda, \mu}(w)$ is holomorphic for $|w| < e^{\Re \lambda}$. And that it has a Taylor expansion in a neighborhood of zero.

This brings us to the second manner of constructing $\beta_{\lambda,\mu}$. We can recursively discover the Taylor series of $g_{\lambda , \mu}$. If,

$$
g_{\lambda,\mu}(e^\lambda w) = \frac{e^{\mu g_{\lambda ,\mu}(w)}}{1+1/w}\\
$$

And if,

$$
a_k = \frac{d^k}{dw^k}\Big{|}_{w=0} g_{\lambda,\mu}(w)\\
$$

Then,

$$
e^{\lambda k} a_k = \frac{d^{k-1}}{dw^{k-1}}\Big{|}_{w=0} \left(\mu g_{\lambda,\mu}'(w) \frac{e^{\mu g_{\lambda ,\mu}(w)}}{1+1/w} + \frac{we^{\mu g_{\lambda,\mu}(w)}}{(1+w)^2}\right)\\
$$

Which is discoverable as an expression in the previous $k-1$ terms; and which by inspection, the sequence on the right must grow like $\mathcal{O}(k!)$ (its domain of holomorphy is $|w| < 1$). Thereby, we must expect that $a_k = \mathcal{O}(e^{-\lambda k} k!)$.

This gives us a Taylor expansion that is surprisingly easy to calculate, with fast convergence:

$$
g_{\lambda , \mu}(w)  = \sum_{k=1}^\infty a_k \frac{w^k}{k!}\\
$$

Which implies that the $\beta$ function has an exponential series representation:

$$
\beta_{\lambda , \mu}(s)  = \sum_{k=1}^\infty \frac{a_k}{k!} e^{k \lambda s}\\
$$

Valid for $\Re(\lambda s) < \Re \lambda$. One can show this without relying on infinite compositions by simply analysing the recursive process defining $a_k$; and using Cauchy's test, determine the radius of convergence of the series; which is $e^{\Re\lambda}$. Thereby, it is not necessary to use infinite compositions to construct $\beta$, but it's significantly simpler.

The function $g_{\lambda , \mu}$ is of importance to us besides this construction. We'll detail this in the next chapter.

\chapter{Mock Abel and Mock Schr\"{o}der Coordinates}

\section{Mock Abel and Mock Schr\"{o}der Coordinates}\label{sec6}
\setcounter{equation}{0}

As a brief note; we begin now to denote $\log$ as the logarithm $\log$ base $b = e^\mu$; and denote $\exp$ as the exponential $\exp$ base $b = e^\mu$. This will be the convention for the remainder of this paper; and $\ln$ will take the place of the natural logarithm base $b = e$.

The functions $\beta$ and $g$ are two peas in a pod. They play a role not dissimilar to the role the inverse Abel function and the inverse Schr\"{o}der function play (respectively) in Complex Dynamics. They describe very similar functional equations. To refresh the reader, the function $\alpha$ is an Abel function of $f$ if:

$$
\alpha(f(s)) = \alpha(s) + 1\\
$$

And similarly, if $\Phi$ is the Schr\"{o}der function of $f$, then, for some multiplier $c$:

$$
\Phi(f(s)) = c \Phi(s)\\
$$

Now, these are commonly referred to as coordinate changes. Where an application of $s \mapsto f(s)$ is mapped to $z \mapsto z+1$; or $s \mapsto f(s)$ is mapped to $z \mapsto cz$; through either the Abel coordinate or the Schr\"{o}der coordinate (respectively). These coordinate changes have countless applications in complex dynamics, and serve as a focal point at describing the dynamics of iterated functions.

Now, we don't care so much for Abel functions, and Schr\"{o}der functions; we care more for their inverses. Thereby, we care for the functions:

$$
\alpha^{-1}(s+1) = f(\alpha^{-1}(s))\\
$$

And,

$$
\Phi^{-1}(c s) = f(\Phi^{-1}(s))\\
$$

The inverse Abel function of $f = \exp$ is, in and of itself, tetration. And the $\beta$ function does not exactly satisfy this equation. It instead satisfies what I refer to as a mock Abel equation. It's almost Abel, but it is not exactly Abel. Where, quite frankly, it satisfies the equation:

$$
\beta(s+1) = \exp  \left(\beta(s) + \epsilon\right)
$$

Where $\epsilon \to 0$ as $\Re(s) \to \infty$. Where we've chosen the appropriate exponential $\exp(z) = e^{\mu z}$ and $\epsilon$ depends on $\lambda$ and $\mu$ in addition to $s$. Similarly, we have the same thing for $g$, where it ``almost" satisfies the inverse Schr\"{o}der equation. That being:

$$
g(cw) = \exp \left( g(w) + \epsilon\right)\\
$$

Where $\epsilon \to 0$ as $|w| \to \infty$. And we've assigned $e^\lambda = c$. Which is little different than the above. But are related by the change of variables $\log(w)/\lambda = s$. It's important to note that $\epsilon$ has a removable singularity at $w=\infty$. As this $\epsilon = -\log(1+1/w)$. So we can wholeheartedly declare that $g$ satisfies the Schr\"{o}der equation at $w = \infty$ on the Riemann-Sphere. 

Please remember we are choosing $\log$ as we chose $\exp = e^{\mu z}$, as the direct inverse, so it depends on $\mu$. But as we don't move $\mu$, we choose this syntax; and hope there is no confusion.

The rest of this article will focus on finding a function $\tau$ such that:

$$
\beta(s+1) + \tau(s+1) = \exp\left(\beta(s) + \tau(s)\right)\\
$$

And a function $r(w)$ such that:

$$
g(cw) + r(cw) = \exp\left( g(w) + r(w)\right)\\
$$

Which is to say: how well can we use the mock solutions to produce the actual solutions? Are they analytic? If so, where? What about smooth; are they mainly just smooth? What are its points of discontinuities? Are they dense?

\section{The $\tau$ error and Levenstein's $\rho$ error}

There are two ways of measuring the error between $\beta$ and tetration. The first way is a direct manner; the second way is summative in nature. We denote these processes $\tau$ and $\rho$ respectively. They are related by a simple identity:

$$
\tau^n = \sum_{j=1}^n \rho^j\\
$$

And for the remainder of this paper we will assume that $\tau$ and $\rho$ are related in this manner. To begin, it is much more natural to work with $\tau$; but from an analytic perspective; it's much more convenient to work with $\rho$. It isn't very hard to describe the $\tau$ process.

$$
\tau^n(s) = \log^{\circ n} \beta(s+n) - \beta(s)
$$

Or,

$$
\beta(s+1) + \tau^{n}(s+1) = \exp\left( \beta(s) + \tau^{n+1}(s)\right)\\
$$

Which describes a recursive process:

\begin{equation}\label{eq:TAU}
\tau^{n+1}(s) = \log\left(1 + \frac{\tau^{n}(s+1)}{\beta(s+1)}\right) - \log(1+e^{-\lambda s})\\
\end{equation}

This is derived from identifying $\log \beta(s+1) = \beta(s) - \log(1+e^{-\lambda s})$ and using run-of-the-mill $\log$ identities. This process tends to converge geometrically; but it has its anomalies and branching problems. So as it converges geometrically; it may not converge uniformly. The rest of this paper is mostly a quest to find out how uniformly it does converge.

To construct the $\rho$ process from the $\tau$ process, we need only apply a bit of algebra. To illustrate:

$$
F_n(s) = \beta(s) + \tau^n(s)\\
$$

Satisfies a similar mock Abel equation:

$$
F_n(s+1) = \exp(F_n(s) + \epsilon_n)\\
$$

Where $\epsilon_n$ is smaller and vanishes as $n\to\infty$. The question is, how well does it vanish. It is here, where we can make the change of variables to $\rho$ and we get an expression for $\epsilon_n$. It's just $\rho^{n+1}$. 

\begin{align*}
    F_n(s) &= \beta(s) + \sum_{j=1}^n \rho^j(s)\\
    F_{n}(s+1) &= \exp\left(F_n(s) + \rho^{n+1}(s)\right)\\
    &= \exp \left( F_{n+1}(s)\right)\\
\end{align*}

This defines an inherited recursive process from $\tau$, which is a degree more convenient than $\tau$. That is:

\begin{align*}
    \rho^{n+1}(s) &= \log F_n(s+1) - F_n(s)\\
                  &= F_{n+1}(s) - F_n(s)\\
\end{align*}

We begin by assigning $\rho^0 = 0$ and $\rho^1 = -\log(1+e^{-\lambda s})$. Upon this point, the process can be further reduced into Levenstein's form:

\begin{equation}\label{eq:RHO}
    \rho^{n+1}(s) = \log\left(1+\frac{\rho^n(s+1)}{\beta(s+1) + \sum_{j=1}^{n-1}\rho^j(s+1)}\right)\\
\end{equation}

The processes defined in \eqref{eq:TAU} and \eqref{eq:RHO} are one and the same. But each has a different benefit than the other. The first proposition of this paper, is that the $\tau$ method can be reduced to Levenstein's form:

\begin{proposition}\label{PropLev}
The recursive process in \eqref{eq:TAU} can be reduced to \eqref{eq:RHO}.
\end{proposition}

\begin{proof}
This result can be broken into a simple chain of equations. As such:

\begin{align*}
    \beta(s+1) + \tau^n(s+1) &= \exp\left(\beta(s) + \tau^{n+1}(s)\right)\\
    F_{n-1}(s+1) + \rho^{n}(s+1) &= \exp\left( F_n(s) + \rho^{n+1}(s)\right)\\
    \log\left( F_{n-1}(s+1) + \rho^n(s+1)\right) & = F_n(s) + \rho^{n+1}(s)\\
    F_n(s) + \rho^{n+1}(s) &= \log F_{n-1}(s+1) + \log\left(1+\frac{\rho^n(s+1)}{F_{n-1}(s+1)}\right)\\
    \rho^{n+1}(s) &= \log\left(1+\frac{\rho^n(s+1)}{F_{n-1}(s+1)}\right)\\
\end{align*}

Where we've cancelled $F_n(s) = \log F_{n-1}(s+1)$. It is necessary that we require $\rho^0 = 0$ and $\rho^1 = - \log(1+e^{-\lambda s})$; and that this recursive process begins at $n=1$ and continues for $n \ge 1$; it fails for $n<1$.
\end{proof}

We enter now two manners of determining how well $\beta$ approximates tetration. Where $\tau^n \to \tau$ uniformly and where the following sum converges compactly normally:

$$
\sum_{j=1}^\infty \rho^j(s)\\
$$

Is where we have an analytic tetration. Where $\tau^n \to \tau$ point wise is where we have at least convergence to a not necessarily analytic tetration. And similarly for $\sum_j \rho^j$. The remainder of this paper will be devoted into showing how and where the following claims hold:

\begin{align*}
    \tau^n(s) &\to \tau(s)\\
    \sum_{j=1}^\infty |\rho^j(s)| &< \infty\\
\end{align*}

We begin by deriving an Asymptotic theorem of the second kind. This is a strengthening of The Asymptotic Theorem of the first kind \ref{thmASYM1}. This theorem shall be reserved for a later section, as it requires a rather technical lemma.

\section{Understanding the domains of interest}\label{sec8}
\setcounter{equation}{0}

To begin this section we have to understand the orbits of $\beta(s)$, $\beta(s+n)$. And this requires a rudimentary understanding of the orbits $\exp^{\circ n}(z)$ for $\exp = e^{\mu z}$. Up until this point in the paper we have not mentioned something which was placed surreptitiously at the beginning.

This is to say, for $\mathcal{B} \subset \mathbb{C}$, and $s \in \mathbb{C}/\mathcal{B} = \mathcal{P}$ the orbits:

$$
\beta(s+n)\,\,\text{have a nice look to them}\\
$$

Where, additionally; there is a set $\mathcal{E} \subset \mathcal{P}$, such that a tetration can be made on $\mathcal{P}/ \mathcal{E}$ and:

$$
\int_{\mathcal{E}} dA = 0\\
$$

For $dA$ the standard area Lebesgue measure in $\mathbb{R}^2$. What we mean by nice is difficult to explain; but it means our process has a normal structure to it. This is quite the mouthful, and is difficult to visualize, so the author will elaborate, and break this into pieces. This section is intended to define where the orbits are nice, and where they aren't.

Commencing, normality is a popular term in complex dynamics, and is a difficult concept to grasp at first, but surprisingly simple once you get the gist of what it means. Supposing you have a sequence of functions:

$$
f_n : X \to Y\\
$$

For them to be a normal family; it simply means there is some subsequence $n_k$ such that:

$$
\lim_{k\to\infty} f_{n_k} \to f\,\,\text{converges uniformly on compact subsets}\\
$$

On $\mathcal{P} = \mathbb{C}/\mathcal{B}$ our error terms $\tau^n$ and $\rho^j$ will be normal. From normality, it is possible to derive holomorphy of the limit. But this holomorphy is confined to where the orbits of $\beta$ are nice.

Now, what we mean by ``nice" may seem a tad artificial, but it's the perfect condition for the development of $\tau$'s convergence. We call this condition ``weak normality'' as it is not quite normality; but it takes from normality what we need.

As to that we introduce:

\begin{definition}[Weak Normal Family]
A family of functions $f_n : X \to Y$, for $X, Y \subset \widehat{\mathbb{C}}$ domains, are a weak normal family if, for all $\mathcal{N} \subset X$:

$$
\limsup_{n\to\infty} \left|\left|\frac{1}{f_n(x)}\right|\right|_{x \in \mathcal{N}}  < \infty
$$
\end{definition}

This is only a slight difference from normal families. A consequence of being a normal family is that you are locally bounded. Which means for all $n$ and for all compact sets $\mathcal{K} \subset X$ the value $||f_n(x)||_{x \in \mathcal{K}} < M$. We have modified this result, such $f_n$ is not necessarily locally bounded, but its multiplicative inverse is locally bounded. In such a sense, it is where $h_n(x) = 1/f_n(x)$ are a normal family; but not quite.

This definition works for us, but it is lacking. The definition should account for when $f_n(x_0) = 0$; but since $\beta \neq 0$, we can ignore handling these cases. I think handling these cases would just be too confusing.

The idea of the definition to follow, is that, where-ever $\{\beta(s+n)\}_{n=0}^\infty$ are a weak normal family, we can make a good guess of the asymptotics of $\tau$ and $\rho$.

So let us now define two sets of incredible importance to us. Let us call them the weak Fatou set, and the weak Julia set; which we denote $\mathcal{P}$ and $\mathcal{B}$. This is precisely the domains of weak normality, and non-weak normality. Precisely where this condition works, and where it doesn't. These domains are related by the identity $\mathcal{P} = \mathbb{C}/\mathcal{B}$ and $\mathcal{B} = \mathbb{C}/\mathcal{P}$. These, effectively, replace the role of Fatou sets and Julia sets in this discussion; compared to the typical discussion of tetration.

\begin{definition}[The Weak Fatou Set]\label{defWFAT}
The weak Fatou set $\mathcal{P} \subset \mathbb{C}$ of $\beta$ is an open domain, in which for all $s \in \mathcal{P}$, there exists a neighborhood $\mathcal{N} = \{y \in \mathbb{C}\,|\,|y-s| < \delta\}$ such that $\mathcal{N} \subset \mathcal{P}$, and the orbits $\beta(\mathcal{N} + n)$ are a weak normal family.
\end{definition}

Showing $\mathcal{P}$ is the domain we'll be able to derive a tetration on is difficult. So, to begin, we'll explain why it matters to categorize $\mathcal{P}$ and $\mathcal{B}$ by playing with a toy model in the next section.

\section{Linearizing the recursion of $\tau$}\label{sec9}
\setcounter{equation}{0}

The goal of this section is to provide a framework for constructing the inverse Abel function of $\exp$ on the weak Fatou set. This result is particularly hard to envision, as we haven't drawn out exactly what the weak Fatou set looks like. We will stave off describing properties of the weak Fatou set for later in this paper; and instead dive into why it's so important to us.

The first result we need to derive to prove an inverse Abel function exists, is to show $\tau^n(s)$ is a normal family for $s \in \mathcal{P}$. This can be equivalently stated that $\tau^n(s)$ is locally bounded on $\mathcal{P}$. It won't come out exactly like this, but for intents and purposes, it will be effectively this result.

We take this section as a momentary breather on all the iterated $\log$'s we've introduced so far. We're going to simplify the discussion greatly through linear approximations. This unearths the heart of our method; but leaves out much of the subtleties and difficulties of nesting logarithms.

To begin, we take:

$$
\tau^{n+1}(s) = \log \left(1+\frac{\tau^{n}(s+1)}{\beta(s+1)}\right) - \log(1+e^{-\lambda s})\\
$$

And we want to show, using Banach's fixed point theorem, that this expression converges geometrically. This can be a tad tricky, but we do show this in the main theorems of this paper. But first, it is necessary to derive normality.

To derive normality; we have to sort of guess how $\tau^n$ grows as $n\to\infty$. This requires giving a quick guess of what $\rho^j$ looks like. Much of this is a quest to use the linearization $\log(1+e^{-\lambda s}) \approx \frac{e^{-\lambda s}}{\mu}$, for large $\Re(s)$, effectively. Once you can effectively describe $\tau$; for large values of $s$ the logarithm looks precisely like this approximation; and $\tau$ can be nested as this linear approximation.

It is at this point where we must enter in a discussion of $\mu$, and where $\beta$'s dependence on $\mu$ starts to make an appearance. It doesn't spoil the convergence too much; but it does cause a bit of a headache, in how to handle these things. 

The linear approximation we are trying to make is:

$$
\tau^{n+1}(s+k) \approx \frac{\tau^n(s+k+1)}{\mu \beta(s+k+1)} - \frac{e^{-\lambda (s+k)}}{\mu}\\
$$

Which is very much accurate for $s \in \mathcal{P}$; especially if you grow $k$. By the definition of the weak Fatou set; we can expect:

$$
\frac{1}{\mu \beta(s+k+1)} = \mathcal{O}(1)\\
$$

We want to use this, to show that this approximation stays bounded. In turn, we can use the boundedness of this approximation, to show that $\tau$ itself is bounded. To begin, this approximation reduces to the summation:

$$
\tau^{n+1}(s+k) \approx -\sum_{j=0}^n \dfrac{e^{-\lambda (s+k+j)}}{\mu^{j+1} \prod_{c=1}^j \beta(s+k+c)}\\
$$

For simplicity, we will call:

$$
\widetilde{\tau}^{n+1}(s+k) = -\sum_{j=0}^n \dfrac{e^{-\lambda (s+k+j)}}{\mu^{j+1} \prod_{c=1}^j \beta(s+k+c)}\\
$$

Which satisfies the recursion:

$$
\widetilde{\tau}^{n+1}(s+k) = \frac{\widetilde{\tau}^n(s+k+1)}{\mu \beta(s+k+1)} - \frac{e^{-\lambda (s+k)}}{\mu}\\
$$

For this to converge for $\mu \in \mathbb{C}$--we need a non-trivial result about $\beta$. This line of questioning may seem erratic; but, at the heart of $\tau$'s convergence, is this linear approximation.

We need to define a bound before continuing, and call it $A_\mu$; which will determine our linear approximation is bounded everywhere $\Re \lambda > \log A_\mu$. For the sake of completeness, we can denote $A_\mu$:

$$
\limsup_{k\to\infty} \left|\left| \frac{1}{\mu \beta(s+k+1)} \right|\right|_{s \in \mathcal{N}} = A_\mu\\
$$

This bound will work for a compact neighborhood $\mathcal{N}$ of a point $s \in \mathcal{P}$, by the properties of the weak Fatou set; as $\mathcal{P}$ is open; and the above limits were done under this topology. This ultimately means that we get the best form $\widetilde{\tau}^n$, which is:

$$
||\widetilde{\tau}^{n+1}(s+k)||_{\mathcal{N}} \le \sum_{j=0}^n \left|\left|e^{-\Re(\lambda s) + \Re \lambda k + (\log A_\mu-\Re \lambda)j}\right|\right|_{\mathcal{N}}\\
$$

Using this asymptotic the value $\mu$ is no longer irrelevant; it depends on $\lambda$ with quite some discretion. And consequently as $n\to\infty$ the sum:

$$
-\sum_{j=0}^\infty \dfrac{e^{-\lambda (s+k+j)}}{\mu^{j+1} \prod_{c=1}^j \beta(s+k+c)}\\
$$

Converges and is analytic. We summarize in the following theorem:

\begin{theorem}[The Linearization Theorem]\label{thmLIN}
The linearized recursion:

\begin{align*}
\widetilde{\tau}^0 &= 0\\
\widetilde{\tau}^{n+1}(s) &= \dfrac{\widetilde{\tau}^{n}(s+1)}{\mu \beta(s+1)} - \frac{e^{-\lambda s}}{\mu}\\
\end{align*}

Converges uniformly on the compact neighborhood $\mathcal{N}\subset\mathcal{P}$ of the weak Fatou set when $\Re \lambda > \log A_\mu$. For:

$$
\limsup_{k\to\infty}\left|\left|\frac{1}{\mu \beta(s+k+1)}\right|\right|_{\mathcal{N}} = A_\mu\\
$$

\end{theorem}

\begin{proof}
Pick a neighborhood $\mathcal{N} \subset \mathcal{P}$ of the weak Fatou set. The value:

$$
||\widetilde{\tau}^{n+1}(s+k)||_{\mathcal{N}} \le \sum_{j=0}^n \left|\left| \frac{e^{-\lambda(s+k+j)}}{\mu^{j+1}\prod_{c=1}^j \beta(s+k+c)}\right|\right|_{\mathcal{N}}\\
$$

Pick $\Re \lambda > \log \left( A_\mu + \epsilon\right)$. For large enough $K$, and $k>K$ we have:

$$
\left|\left| \frac{1}{\mu \beta(s+k)}\right|\right|_{\mathcal{N}} \le
A_\mu + \epsilon\\$$

Where $\epsilon$ is arbitrarily small for large enough $K$. And therefore:

$$
\left|\left| \frac{1}{\mu^{j+1}\prod_{c=1}^j \beta(s+k+c)}\right|\right|_{\mathcal{N}} \le e^{j \log\left(A_\mu + \epsilon\right)}\\
$$

Therefore:

$$
||\widetilde{\tau}^{n+1}(s+k)||_{\mathcal{N}} \le \sum_{j=0}^n \left|\left|e^{-\lambda(s+k) - (\lambda - \log (A_\mu + \epsilon))j}\right|\right|_{\mathcal{N}}\\
$$

Therefore $\widetilde{\tau}^{n+1}(s+k) \to \widetilde{\tau}(s+k)$ converges uniformly for $s \in \mathcal{N}$ while $n\to\infty$ because this series converges for $\Re \lambda > \log(A_\mu + \epsilon)$.  By the functional equation:

$$
\widetilde{\tau}(s+k-1) = \frac{\widetilde{\tau}(s+k)}{\mu \beta(s+k)} - \frac{e^{-\lambda(s+k-1)}}{\mu}\\
$$

We know that $\widetilde{\tau}(s)$ converges; or that:

$$
\sum_{j=0}^\infty \left|\left| \frac{e^{-\lambda(s+j)}}{\mu^{j+1}\prod_{c=1}^j \beta(s+c)}\right|\right|_{\mathcal{N}} < \infty\\
$$

for $\Re \lambda > \log(A_\mu + \epsilon)$; since $K$ is now irrelevant, we can set $\epsilon = 0$. Which gives the result.
\end{proof}

We will use this result to derive that $\tau^n$ is normal on the weak Fatou set; for specific $\lambda$. This can be done because as we increase $k$ we get closer and closer to the linear case.

\section{Some examples of what to expect}\label{sec10}
\setcounter{equation}{0}

Before we move on to proving the real difficult stuff; it's helpful to visualize some of these constants and sets we've introduced. For that we'll fix a couple values for $\mu$ and $\lambda$. This is a place where we'll take a breather. The case we'll begin with is one well studied in the field of iteration theory; with results dating to Euler.

Let us begin this breath, by setting $\mu = \ln(2)/2$ for the principle branch of $\log$ with base $e$. And let us only let $\lambda$ vary. This is precisely:

$$
\beta(s) = \OmSum_{j=1}^\infty \frac{\sqrt{2}^z}{1+e^{\lambda(j-s)}}\,\bullet z\\
$$

We start by determining the weak Fatou set. It isn't hard for a veteran to notice that $\lim_{\Re s\to\infty} \beta(s) \to 2$. But it may take some explaining for some one unfamiliar. There is an attractive fixed point of $h(z) = \sqrt{2}^z$ at $z=2$; and the orbit of $h^{\circ n}(0)$ converges to $2$; as $0$ is in its Fatou set. The orbits of $\beta$ behave asymptotically like the orbits of $h^{\circ n}(0)$ because $\beta(-\infty) = 0$.

Now we can do this compactly everywhere, excepting near singularities; where we may experience a slight divergence. So, although not effectively discovered, the weak Fatou set is:

$$
\mathcal{P} \supset \{s \in \mathbb{C}\,|\, |s-s_j| > \delta,\,\lambda(j-s_j) = (2k+1)\pi i,\, j\ge 1,\, j,k \in \mathbb{Z}\}\\
$$

For $\delta$ sufficiently small, but not too small. From here:

$$
\lim_{k\to \infty}\beta(s+k) \to 2\\
$$

And the constant $A_\mu$ is given as:

$$
A_\mu =\limsup_{k\to\infty} \left|\left|\frac{2}{\log(2)\beta(s+k)}\right|\right|_{\mathcal{N}} \to \frac{1}{\log(2)}\\
$$

So the linear approximation will converge for $\Re \lambda > - \log\log(2)$; and just as well, so will the actual non-linear approximation. To demonstrate, in Figure \ref{fig:ROOTTWO} is a graph to about 20 digit accuracy of this inverse Abel function when $\lambda = 1$:

\begin{figure}
    \centering
    \includegraphics[scale = 0.3]{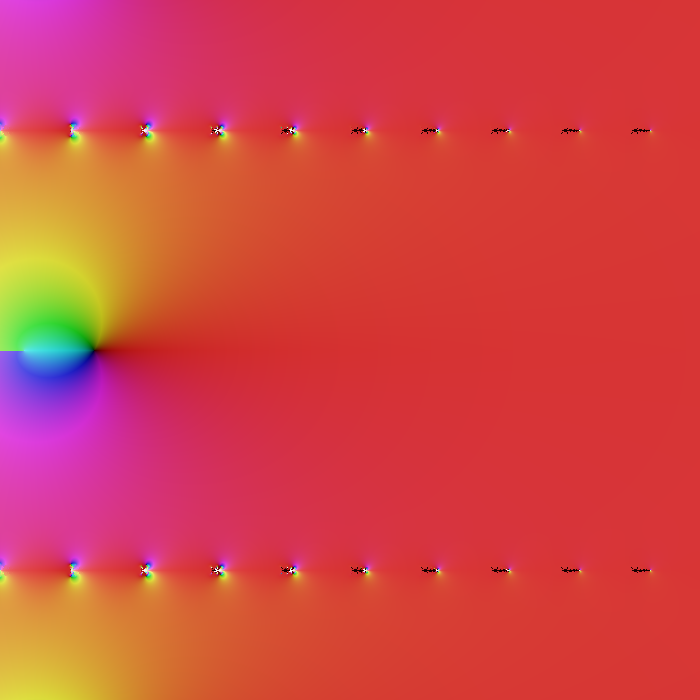}
    \caption{The inverse Abel function of $\exp$: $\beta(s) + \tau(s)$ for $\mu = \ln(2)/2$ and $\lambda = 1$. Accurate to about twenty digits.}
    \label{fig:ROOTTWO}
\end{figure}

Although there are small points of divergence in Figure \ref{fig:ROOTTWO}; they are translation invariant; and belong to the weak Julia set. They are, again, what we want to throw away from our result. For the most part; we have a well behaved tetration. Thus, we can expect the weak Fatou set to essentially be $\mathbb{C}$ minus $\lambda(j-s) = (2k+1) \pi i$ upto a fractal or two near these singularities. Which is essentially where $\beta \approx 1$ or a forward orbit of $1$; which means the $\log$ may hit a $0$; which causes a negligible branch-cut in the area measure.

A similar graph can be made for all $\Re\lambda > - \log \log (2)$. As soon as you shrink $\lambda$ smaller you get more chaotic behaviour. Very similar results can be observed for $b = e^{\mu}$ in the Shell-Thron region. Where this is precisely where $A_\mu$ is non-zero. Everywhere else, you can expect the value $A_\mu = 0$ (in a point wise limit), if it exists.\\

For the second example we will consider $\mu =1$ and $\lambda = 1$. This is a rather chaotic case, and isn't as simple as it seems. To begin, let us declare $\beta$,

$$
\beta(s) = \OmSum_{j=1}^\infty \frac{e^z}{1+e^{j-s}}\,\bullet z\\
$$

Understanding the dynamics of this $\beta$ is very difficult, as it's nearly the dynamics of $e^z$. The weak Fatou set throws many curve balls at us; it's precisely the empty set. To begin, the linearization converges everywhere excepting the singularities.  But it does not converge uniformly anywhere. And this is because even though the result converges everywhere, the weak Fatou set is empty. Particularly $\mathbb{R} \subset \mathcal{B}$; the real line belongs to the weak Julia set. The induced inverse Abel function on $\mathbb{R}$ will be smooth; it will be nowhere analytic though.

This is an in depth result shown multiple times by Sheldon Levenstein--largely through graphical analysis. Where $\tau : \mathbb{R}^+ \to \mathbb{R}$ is only infinitely differentiable; it is analytic nowhere. This is because; although the linear approximation converges really fast and really well on $\mathbb{R}^+$; at each point $s \in \mathbb{R}^+$, there is no neighborhood about it which converges; we don't have convergence that isn't trivial under a Lebesgue area measure.

And thus; the weak Fatou set, when $\mu = 1$ and $\lambda =1$, explicitly excludes the real-line, because it is not open in $\mathbb{C}$; despite displaying god awful fast convergence. 

Additionally, calculating $\tau$ is expressly very difficult. The value $\beta(5)$ is already an over flow; and $\beta(10+i)$ is an overflow just as well. This is, quite because $\beta(s+k)$ grows to $\infty$ rather fast and sporadically; but at least it does so regularly.  For that reason it's very difficult to get a good picture of the weak Fatou set $\mathcal{P}$ through computations; but it is the empty set.

\begin{figure}
    \centering
    \includegraphics[scale=0.3]{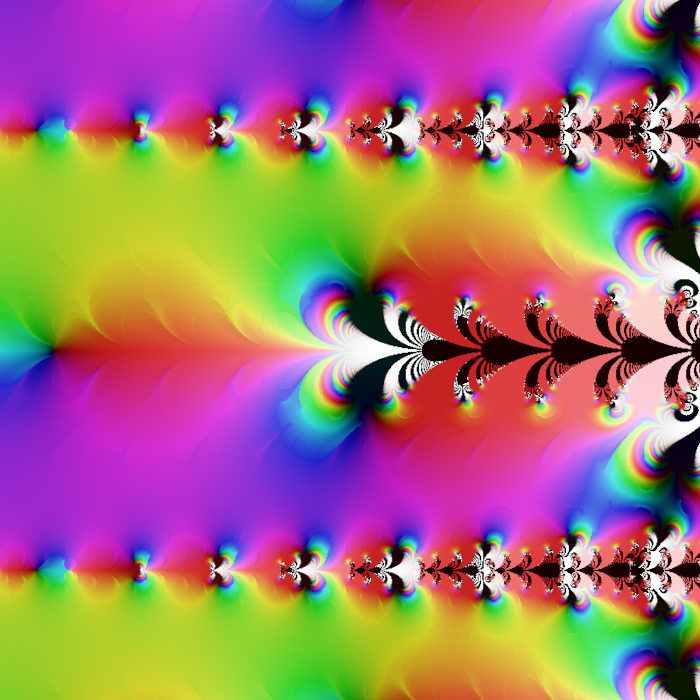}
    \caption{A rough approximation of $\beta(s) + \tau(s)$ for: $\mu =1$ and $\lambda = 1$. Expect the teeth and branch cuts which appear in this graph to increase indefinitely as we get closer and closer approximations. Despite looking holomorphic; it is in fact not holomorphic; we are reduced to an asymptotic series.}
    \label{fig:E_1}
\end{figure}

In Figure \ref{fig:E_1} you can see logarithmic singularities and their branch cuts appearing; and you can see them appearing near the real-line.  More will appear under further iterations; but they will never be dense. But you can still expect the weak Julia set $\mathcal{B} = \mathbb{C}$ to be the complex plane.

In this instance we can expect:

$$
\lim_{k\to\infty} \frac{1}{\beta(s+k)} = 0\\
$$

For all $s \neq j + (2k+1)\pi i$. But:

$$
\limsup_{k\to\infty} \left|\left| \frac{1}{\beta(s+k)}\right|\right|_{\mathcal{N}} = \infty\\
$$

Where $\mathcal{N}$ is a compact neighborhood about $s$ in $\mathbb{C}$. So, although:

$$
\limsup_{k\to\infty} \left|\left| \frac{1}{\beta(s+k)}\right|\right|_{\mathbb{R}^+} \to 0\\
$$

This expression does not do so uniformly in $\mathbb{C}$. The moment $s$ moves from $\mathbb{R}^+$, we experience a chaos and divergence. Which forbids a normal convergence. Hence, implying that the real line is not part of the weak Fatou set; because the weak Fatou set is open by definition.

So, when $\mu = 1$ and $\lambda =1$ we can expect that $\tau$ converges; but whether it converges uniformly on compact disks in $\mathbb{C}$ is a whole other problem; and it doesn't converge. On the real line it will be infinitely differentiable and converge fast; but it will be nowhere analytic because these bounds cannot be made in a neighborhood of the real-line.

We'll see later that the weak Fatou set is trivial and does not exist. But describing this result is very difficult. Nonetheless Figure \ref{fig:E_1} gives us a good picture of what the inverse Abel function will look like; just expect more sickle like branch-cuts; and at best, you can only make an asymptotic series at each point. This graph was done with a 300 term polynomial. Despite not being holomorphic, the asymptotic series is rather accurate.

This allows us to let $\lambda$ be as close to $0$ as possible. Where we can now take $\Re \lambda$ as close to zero as possible. In Figure \ref{fig:E_25} is a graph when $\mu =1$ and $\lambda = 0.25$. There are an aggregate of singularities and branch cuts which arise for this case; and if you could take further iterations there would be more. Meaning, again, that the weak Julia set is all of $\mathbb{C}$.

\begin{figure}
    \centering
    \includegraphics[scale=0.3]{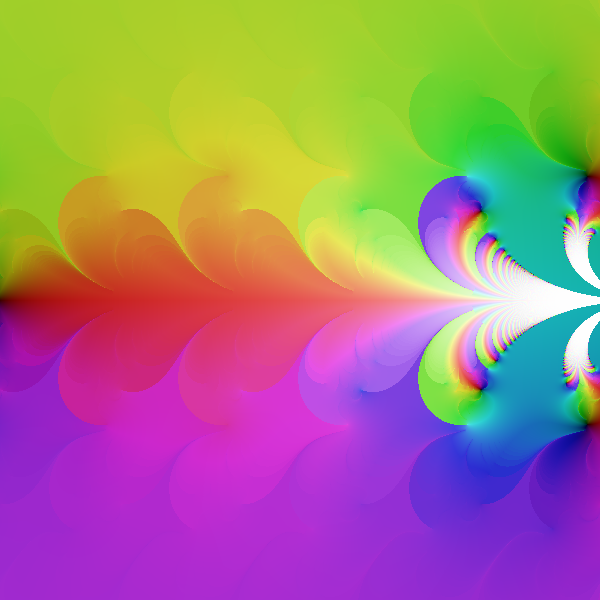}
    \caption{An approximation of the inverse Abel function when $\mu = 1$ and $\lambda = 0.25$.}
    \label{fig:E_25}
\end{figure}

Determining this limit superior converges to infinity everywhere compactly is a much more difficult thing though. It's rather non-trivial to show there is a domain where it does converge to infinity compactly.\\

We can also introduce more anomalous constructions. Taking $\mu = 1+i$ and $\lambda = 1+i$ we get a different looking beast. In concurrence $e^{1+i}$ does not belong to the Shell-Thron region (depicted in Figure \ref{fig:Sh-Th}) there similarly exists logarithmic branch-cuts which appear in $\tau$; as with the previous case.

This means, as before; like with $e$ and $\lambda=1$; we can expect the weak Fatou set to, for all intents and purposes be, erratic and misbehaved. Nonetheless; in Figure \ref{fig:tet_1_I_1_I}, we can see a very weird looking construction. And although there are a plethora of singularities, they are sparse. And they contribute a measure zero effect on the translation invariant domain $\mathcal{P}$. Where $\mathcal{P}$ is where the linearization is well behaved.

\begin{figure}
    \centering
    \includegraphics[scale=0.5]{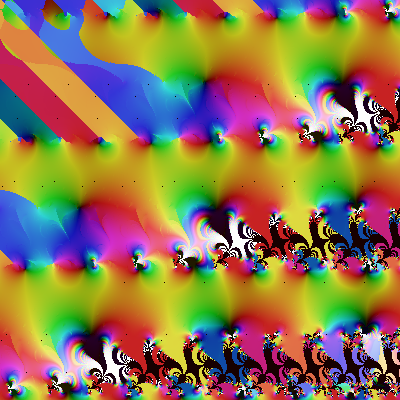}
    \caption{The function $\beta(s) + \tau(s)$ for $\mu = 1+i$ and $\lambda = 1+i$.}
    \label{fig:tet_1_I_1_I}
\end{figure}

The teeth, and the non-analytic points will appear more often as you increase the iterations; but they will never be dense. And again, despite the weak Fatou set being oddly shaped; and full of cuts; it is translation invariant and can induce a tetration.

We'd also like to talk about when $\lambda =1$ rather than $\lambda = 1+i$. We'll keep $\mu = 1+i$ and only vary $\lambda$. This produces the very different graph in Figure \ref{fig:tet_1_I_1}; but satisfies the same inverse Abel equation. The biggest difference is that the period is now $2\pi i$ as opposed to $\dfrac{2 \pi i}{1+i}$. This moves the logarithmic singularities substantially. You'll get for $0 < \Im(s) < \pi$ a very regular structure; and the chaos happens somewhere for $0 > \Im(s) > -\pi$. Which is largely because this is where the linearization is well behaved, versus where the linearization is chaotic.\\

\begin{figure}
    \centering
    \includegraphics[scale=0.5]{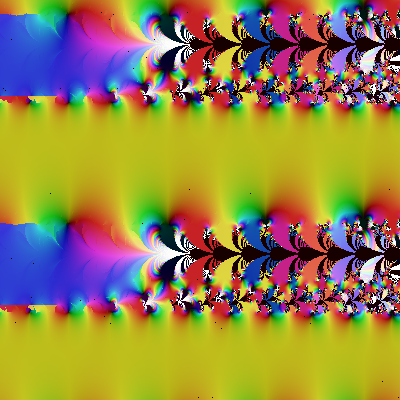}
    \caption{The inverse Abel function for $\mu = 1+i$ and $\lambda = 1$.}
    \label{fig:tet_1_I_1}
\end{figure}

And lastly we'll look at when $\lambda = 1$ and $\mu = 0.3 + i$; which lies within the Shell-Thron region in Figure\ref{fig:Sh-Th}. In Figure \ref{fig:tet_pt3_1_I} we get a good look at this function. You can see a couple of branch cuts; and some good chaos. But for the most part we have a well behaved structure.\\

\begin{figure}
    \centering
    \includegraphics[scale=0.4]{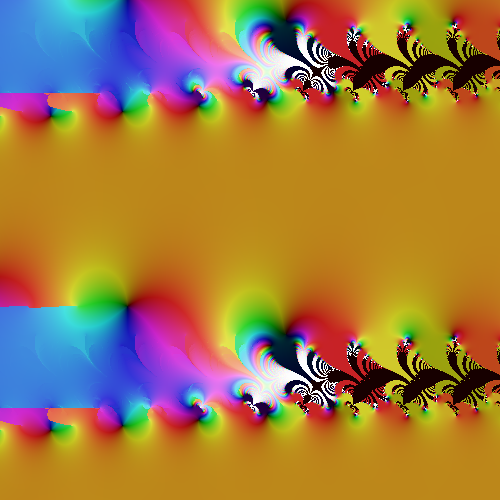}
    \caption{The inverse Abel function for $\mu =0.3 +i$ and $\lambda=1$.}
    \label{fig:tet_pt3_1_I}
\end{figure}

\section{Characterizing the weak Julia set}\label{sec11}
\setcounter{equation}{0}

In this section we will characterize the weak Julia set, as we've characterized the weak Fatou set. We will do so in a similar manner as the Julia set is characterized independently from the Fatou set; but, through weak normality.

\begin{definition}[The weak Julia set]\label{defWJUL}
The weak Julia set $\mathcal{B} \subset \mathbb{C}$ of $\beta$ is a closed domain, in which for all $s \in \mathcal{B}$, for all neighborhoods $\mathcal{N} = \{y \in \mathbb{C}\,|\,|y-s| < \delta\}$ of $s$; the value:

$$
\limsup_{k\to\infty} \left|\left|\frac{1}{\beta(s+k)}\right|\right|_{\mathcal{N}} = \infty\\
$$

\end{definition}

So, for example when $\mu =1$ and $\lambda =1$; we can expect the real line $\mathbb{R}\subset \mathcal{B}$. We'd like to spend some time describing where to expect the weak Julia set to exist. And describe some of its properties.\\

To understand the weak Julia set; it's important to split the plane into two separate limits. When you take $e^{\mu z}$; the plane can be split into two domains:

\begin{align*}
    \mathcal{L}^+ &= \{ z \in \mathbb{C}\,|\, |\arg(\mu z)| < \pi/2\}\\
    \mathcal{L}^- &= \{ z \in \mathbb{C}\,|\, \pi/2 < |\arg(\mu z)| \le \pi\}\\
\end{align*}

Which $e^{\mu z}$ is holomorphic on each. But, we can attach infinity to each of these domains, and a very different picture starts to appear.

\begin{align*}
    \mathcal{L}_\infty^+ &= \mathcal{L}^+ \cup \infty\\
    \mathcal{L}_\infty^- &= \mathcal{L}^- \cup \infty\\
\end{align*}

And we arrive at two separate holomorphic functions on the Riemann sphere:

\begin{align*}
    e^{\mu z} : \mathcal{L}_\infty^{+} & \to \widehat{\mathbb{C}}\\
    e^{\mu z} : \mathcal{L}_\infty^{-} & \to \widehat{\mathbb{C}}\\
\end{align*}

In which we have two very different limits:

\begin{align*}
    e^{\mu \infty} &= \infty\,\,\text{as}\,\,|z|\to\infty\,\,\text{while}\,\,z \in \mathcal{L}^+\\
    e^{\mu \infty} &= 0\,\,\text{as}\,\,|z|\to\infty\,\,\text{while}\,\,z \in \mathcal{L}^-
\end{align*}

To find where the weak Julia set is; we are looking for values in the forward orbits, and the backwards orbits, of these exact points on the Riemann sphere. We're looking for where $\beta(s+n) \approx 0$ and where $\beta(s+n) \approx \infty$. Where we can expect, in the second case; for a significant domain $\beta(s+n+1) \approx 0$. 

This can be clarified by some basic topology, on the Riemann sphere $\widehat{\mathbb{C}}$, the function $e^{\mu z}$ has a fifty fifty chance of being $\approx\infty$ or $\approx 0$ in a neighbourhood of $z \approx \infty$. These are where the weak Julia sets potentially arise. And in that spirit; the weak Julia set appears where many large values appear, because large values cause tiny values. But where tiny values are; they do not cause large values. But large values cause neighborhoods of large values and tiny values; and this causes the weak normality to fail.

These chaotic neighborhoods are precisely where we are too close to $\infty$ on the Riemann sphere. When $\beta(s) \approx \infty$ without a sufficient calmness to it; the orbits $\beta(s+n)$ cluster values of $\infty$ and $0$ together, causing the orbits $\frac{1}{\beta(s+n)}$ to cluster too large at $\infty$ to allow us to take a supremum norm.

And this is precisely what happens in a neighborhood of the real line when $\mu =1$ and $\lambda = 1$ which results in $\mathbb{R} \subset \mathcal{B}$. As we increase $x$ on $\mathbb{R}$ we grow super-exponentially to $\infty$; and so a neighbhorhood of $x$ must be equally large; but then the next orbit; once adding in an imaginary component; can cause $\beta(x+1 + i\delta) \approx 1,0$

\section{The Asymptotic theorem of the second kind}\label{sec12}
\setcounter{equation}{0}

The point of this section is to handle a result existant in the weak Fatou set and the Weak Julia set; and their counter parts the Fatou set and Julia set. We will start with a staple of iterated exponentials, which is that $e^z$'s Julia set is all of $\mathbb{C}$. So the Fatou set is empty.

Similarly, taking $\mu = 1$ and $\lambda=1$, the weak Fatou set is empty. This is the most noticeable similarity between the weak Fatou set and the Fatou set. There is a weakening in the criterion, but doesn't give us more values to play with. 

Therefore; the distinction between the Fatou/Julia set isn't the give all end all of tetration. Some results lay out universally. The theorem to follow is an example of this.

We want to use the following idea by induction. Let $\rho^n$ be as they were.

\begin{align*}
F_n(s+1+k) &= \exp\left( F_{n+1}(s+k)\right)\\
F_{n-1}(s+1 + k) + \rho^n(s+1+k) &= \exp\left( F_{n+1}(s+k)\right)\\
\rho^n(s+1+k) &= \exp\left( F_{n}(s+k) + \rho^{n+1}(s+k)\right)-F_{n-1}(s+1 + k)\\
&\to 0\,\,\text{as}\,\,k\to\infty\\
\end{align*}

Which from here, we want to show that $\rho^{n+1}(s+k) \to 0$; which is obvious if you think about it, because:

$$
F_{n-1}(s+1+k)\left(e^{\rho^{n+1}(s+k)} - 1\right) =  \exp\left( F_{n}(s+k) + \rho^{n+1}(s+k)\right)-F_{n-1}(s+1 + k)\\
$$

And this expression tends to zero uniformly as $k\to\infty$. Since $F_{n-1}$ doesn't tend to zero uniformly; we must have:

$$
\left(e^{\rho^{n+1}(s+k)} - 1\right) \to 0\,\,\text{as}\,\,k\to\infty\\
$$

Meaning:

$$
\rho^{n+1}(s+k) \to 0\,\,\text{as}\,\,k\to\infty\\
$$

Which concludes the induction step. This allows us a normality result for both the weak Julia set, and the weak Fatou set. 

\begin{theorem}[The Asymptotic Theorem of the second kind.]\label{thmASYM2}
For all $j \in \mathbb{N}$, and all $(s,\lambda) \in \mathbb{L}$ the functions $\rho^j$ satisfy:

$$
||\rho^j(s+k)||_{s\in\mathcal{N}} \to 0\,\,\text{as}\,\,k \to \infty\\
$$

For $\mathcal{N}$ a compact neighborhood of $s$. Similarly, for all $n \in \mathbb{N}$ the functions $\tau^n$ satisfy:

$$
||\tau^{n}(s+k)||_{s\in\mathcal{N}} \to 0 \,\,\text{as}\,\,k\to\infty\\
$$
\end{theorem}

\begin{proof}
We go by induction on $\rho^j$. To begin $\rho^0 = 0$ and $\rho^1 = -\log(1+e^{-\lambda s})$ which both satisfy the asymptotic. Assume for all $0 \le j \le n$ that $\rho^j$ satisfies the asymptotic condition. Let:

$$
F_n(s) = \beta(s) + \sum_{j=1}^n \rho^j(s)\\
$$

Now observe that:

\begin{align*}
F_{n-1}(s+k+1) + \rho^n(s+k+1) &= \exp F_{n+1}(s+k)\\
&= \exp\left(F_{n}(s+k) + \rho^{n+1}(s+k)\right)\\
\rho^n(s+k+1) &= \exp F_{n}(s+k) \exp \rho^{n+1}(s+k) - F_{n-1}(s+k+1)\\
&= F_{n-1}(s+k+1) \left(\exp \rho^{n+1}(s+k)-1\right)\\
\end{align*}

By the induction hypothesis:

$$
\rho^n(s+k+1) \to 0\,\,\text{as}\,\,k\to\infty\\
$$

But furthermore, tends to zero compactly normally for $s \in \mathcal{N}$, a compact neighborhood about $s$. Since $F_{n-1}(s+k+1) \not\to 0$ compactly normally, we must have:

$$
\left(\exp \rho^{n+1}(s+k)-1\right) \to 0\,\,\text{as}\,\,k\to\infty\\
$$

And does so compactly normally; hence the first part of the theorem. Since:

$$
\tau^n(s+k) = \sum_{j=1}^n \rho^j(s+k)\\
$$

We have the second statement of the theorem as well.
\end{proof}

So regardless of the weak Fatou set and the weak Julia set distinction; we still have an asymptotic theorem. It is where we ask how this asymptotic works that we enter a more difficult discussion. And is where we need the weak Fatou set, and the weak Julia set, to make this distinction.

The weak Fatou set is precisely where:

$$
\rho^j(s) = \mathcal{O}(r^j)\\
$$

For some $0 < r < 1$; of which this norm is done compactly; which causes a geometric convergence, which causes a holomorphic tetration. So understanding where this asymptotic is geometric compactly is precisely where we get a holomorphic tetration.

\chapter{An inverse Abel function on the weak Fatou set}

\section{An inverse Abel function on the weak Fatou set}\label{sec13}
\setcounter{equation}{0}

We are going to begin by sketching a construction for the inverse Abel function; and then putting that argument in action. We can begin by recalling our important constant:

$$
A_\mu = \limsup_{k\to\infty} \left|\left|\frac{1}{\mu \beta(s+k)}\right|\right|_{\mathcal{N}}
$$

Which converges for $\mathcal{N} \subset \mathcal{P}$ a compact neighborhood in the weak Fatou set. We will restrict: $\Re \lambda > \log A_\mu$. And we are going to look at the asymptotics of $\tau^n$. By example:

$$
\tau^{n+1}(s+k) = \log\left(1+\frac{\tau^n(s+k+1)}{\beta(s+k+1)}\right) - \log(1+e^{-\lambda (s+k)})
$$

If we let $k > K$ be large enough; we can set:

$$
\left|\left|\log\left(1+\frac{\tau^n(s+k+1)}{\beta(s+k+1)}\right)\right|\right|_{\mathcal{N}} \le C\left|\left|\frac{\tau^n(s+k+1)}{\mu\beta(s+k+1)}\right|\right|_{\mathcal{N}}\\
$$

For $C \approx 1$; which tends to $1$ as we let $K \to \infty$. Additionally; $\tau^n(s+k)$ tends to zero as $k\to\infty$. And we arrive at the rough estimate:

$$
\left|\left|\log\left(1+\frac{\tau^n(s+k+1)}{\beta(s+k+1)}\right)\right|\right|_{\mathcal{N}} \le C (A_\mu + \delta)\left|\left|\tau^n(s+k+1)\right|\right|_{\mathcal{N}}\\
$$

For a $\delta > 0$. Where $C (A_\mu + \delta) = A_\mu + \delta'$ for $\delta' \to 0$ as $K \to \infty$. So we can arrive at the bound:

$$
||\tau^{n+1}(s+k)||_{\mathcal{N}} \le (A_\mu+\delta')\left|\left|\tau^n(s+k+1)\right|\right|_{\mathcal{N}} +(1+\delta')||e^{-\lambda s}||_{\mathcal{N}}\frac{e^{-\lambda k}}{\mu}\\\\
$$

And now this is precisely the linear approximation. And consequently, we get a result the author has hinted at repeatedly since the beginning of this paper.

\begin{theorem}[The Normal Theorem]
For $\mathcal{N} \subset \mathcal{P}$ a compact neighborhood in the weak Fatou set; and $\Re \lambda > \log A_\mu$; there exists $K$ such for $k>K$ and $n \in \mathbb{N}$, the sequence of functions:

$$
||\tau^{n}(s+k)||_{\mathcal{N}} \le M\\
$$

For a constant $M$ depending on $\mathcal{N}$ and $K$.
\end{theorem}

\begin{proof}
Starting from the bound:

$$
||\tau^{n+1}(s+k)||_{\mathcal{N}} \le (A_\mu+\delta')\left|\left|\tau^n(s+k+1)\right|\right|_{\mathcal{N}} + (1+\delta')\frac{||e^{-\lambda (s+k)}||_{\mathcal{N}}}{\mu}\\
$$

We arrive at the expression:

$$
||\tau^{n+1}(s+k)||_{\mathcal{N}} \le (1+\delta')||e^{-\lambda(s+k)}||_{\mathcal{N}}\sum_{j=0}^n (A_\mu+\delta')^j e^{-\Re\lambda j}\\
$$

For large enough $K$, we have $\delta'$ arbitrarily small; so when $\Re \lambda > \log(A_\mu+\delta')$; and setting $r = e^{-(\lambda - \log (A_\mu+\delta'))}$; we're given the bound:

$$
||\tau^{n+1}(s+k)||_{\mathcal{N}} \le (1+\delta') ||e^{-\lambda(s+k)}||_{\mathcal{N}} \sum_{j=0}^nr^j\\
$$

And consequently; setting 

$$
M = \frac{(1+\delta')}{1-r}||e^{-\lambda(s+K)}||_{\mathcal{N}}
$$

We arrive at the theorem's statement.
\end{proof}

There is a stronger sentiment in this statement too; which is that for $k>K$ we can expect the slightly stronger corollary:

\begin{corollary}
For $\mathcal{N} \subset \mathcal{P}$ a compact neighborhood in the weak Fatou set; and $\Re \lambda > \log A_\mu$; there exists $K$ such for $k>K$ and $n \in \mathbb{N}$, the sequence of functions:

$$
||\tau^{n}(s+k)||_{\mathcal{N}} \le C e^{-\Re\lambda k}\\
$$

For some constant $C \in \mathbb{R}^+$.
\end{corollary}

It is at this point, where we must use The Banach Fixed Point Theorem. This can be achieved with the simple identity:

$$
|\log(1+x) - \log(1+x')| \le \frac{1+\delta'}{\mu}|x-x'|\\
$$

for $\delta' \to 0$ as $|x|,|x|' \to 0$. So in conjunction if we take:

\begin{align*}
|\tau^{n+1}(s+k) - \tau^{n}(s+k)| &\le \left|\log\left(1+\frac{\tau^{n}(s+k+1)}{\beta(s+k+1)}\right) - \log\left(1+\frac{\tau^{n-1}(s+k+1)}{\beta(s+k+1)}\right)\right|\\
&= \frac{1+\delta'}{\mu|\beta(s+k+1)|}\left| \tau^n(s+k+1) - \tau^{n-1}(s+k+1)\right|\\
&\vdots\\
&= \dfrac{(1+\delta')^n}{\mu^n \prod_{j=1}^n |\beta(s+k+j)|}|\log(1+e^{-\lambda (s+k+n)})|\\
\end{align*}

Now, if we choose our compact neighborhoods correctly; we can arrive at our theorem; which for $r = e^{\log (A_\mu + \delta'') - \lambda}$, where $\Re \lambda > \log (A_\mu + \delta'')$--we arrive at:

$$
||\tau^{n+1}(s+k) - \tau^{n}(s+k)||_{\mathcal{N}} \le Cr^n\\
$$

For some $C \in \mathbb{R}^+$. Which implies that:

$$
||\rho^{n+1}(s+k)||_{\mathcal{N}} \le Cr^n\\
$$

And that:

$$
\sum_{j=1}^\infty ||\rho^j(s+k)||_{\mathcal{N}} \le \frac{C}{1-r}\\
$$

Which implies the inverse Abel function will be holomorphic on $\mathcal{N}$ for large enough $k>K$.

\begin{theorem}[The Convergence Theorem]\label{thmABLCVG}
For $\mathcal{N} \subset \mathcal{P}$ a compact neighborhood in the weak Fatou set; and $\Re \lambda > \log A_\mu$; there exists $K$ such for $k>K$:

$$
\lim_{n\to\infty} \tau^n(s+k) = \tau(s+k)\\
$$

Converges uniformly for $s \in \mathcal{N}$.
\end{theorem}

\begin{proof}
Starting from the bound:

$$
||\tau^{n+1}(s+k) - \tau^n(s+k)||_{\mathcal{N}} \le \left| \left|\frac{1+\delta'}{\mu |\beta(s+k+1)|}
\right|\right|_{\mathcal{N}}\left|\left| \tau^{n}(s+k+1) - \tau^{n-1}(s+k+1)\right|\right|_{\mathcal{N}}\\
$$

Where $\delta' \to 0$ as $K \to \infty$; we can bound:

$$
\left|\left|\frac{1+\delta'}{\mu |\beta(s+k+1)|}
\right|\right|_{\mathcal{N}} \le A_\mu + \delta''\\
$$

For $\delta'' \to 0$ as $K \to \infty$. Iterating this relationship takes us to:

\begin{align*}
||\tau^{n+1}(s+k) - \tau^n(s+k)||_{\mathcal{N}} &\le (A_\mu + \delta'')^n||\tau^1(s+k+n) - \tau^0(s+k+n)||_{\mathcal{N}}\\
&\le (A_\mu+\delta'')^n||\log(1+e^{-\lambda (s+k+n)})||_{\mathcal{N}}\\
\end{align*}

Letting $\Re \lambda > \log(A_\mu + \delta'')$; then this expression is bounded for some $C \in \mathbb{R}^+$:

$$
||\tau^{n+1}(s+k) - \tau^n(s+k)||_{\mathcal{N}} \le Ce^{-(\lambda-\log (A_\mu+\delta''))n} = Cr^n\\
$$

For $0 < r < 1$. This means that:

$$
||\rho^{n+1}(s+k)||_{\mathcal{N}} \le Cr^n\\
$$

And consequently:

$$
\sum_{j=1}^\infty ||\rho^{j}(s+k)||_{\mathcal{N}} \le \sum_{j=0}^\infty Cr^j = \frac{C}{1-r}\\
$$

Therefore the series:

$$
\tau(s+k) = \sum_{j=1}^\infty \rho^j(s+k)\\
$$

Converges normally and $\tau(s+k)$ is holomorphic on $\mathcal{N}$ for some $k > K$ when $\Re \lambda > \log A_\mu$.
\end{proof}

We have not effectively described the value $K$; nor its dependence on $\mathcal{N}$; and so it is difficult to call this an effective solution to the inverse Abel function. But nonetheless; in no trivial manner; we have found a function:

$$
F(s) = \beta(s) + \tau(s)\\
$$

Which is holomorphic somewhere; in which:

$$
F(s+1) = \exp F(s) = e^{\mu F(s)}\\
$$

The natural way to extend this function to its maximal domain, is using the implicit function theorem. By which:

$$
F(s) = \exp y\\
$$

And the function $y$ exists locally everywhere $F(s) \neq 0$ (because $\exp$ is always surjective on $\mathbb{C}/\{0\}$ for every $\mu$) and where-ever $\frac{d}{dy} \exp (y) \neq 0$; which is everywhere. By a monodromy theorem, we can connect all of these neighborhoods to arrive at a function $y$ in which $y(s) = F(s-1)$. This function can be extended to the left, excluding points where $F(s_0) = 0$; in which $F(s_0 -1)$ will be a logarithmic singularity with a branch cut. We can iterate this process, and we arrive at an inverse Abel function that is expanded to a much larger domain.

It becomes the following corollary:

\begin{corollary}[The Pull Back Corollary]\label{corPB}
For $\mathcal{N} \subset \mathcal{P}$ a compact neighborhood in the weak Fatou set; and $\Re \lambda > \log A_\mu$; for all $k \in \mathbb{Z}$:

$$
\lim_{n\to\infty} \tau^n(s+k) = \tau(s+k)\\
$$

Is holomorphic for $s \in \mathcal{N}_k^*$; where $\mathcal{N}/\mathcal{N}_k^*$ is measure zero under an area measure (it consists of branch cuts). Additionally:

$$
\lim_{k\to\infty} \mathcal{N}/\mathcal{N}_k^* = \emptyset\\
$$
\end{corollary}

We enter now, into a more delicate conversation of $\mathcal{N}$ and $A_\mu$. If we have $\mathcal{N}^1$ and another $\mathcal{N}^2$; producing $A_\mu^1$ and $A_\mu^2$. The value $A_\mu^3$ for $\mathcal{N}^3 = \mathcal{N}^1 \cup \mathcal{N}^2$ is just $A_\mu^3 = \max(A_\mu^1,A_\mu^2)$. This implies that $\lambda$ has a deeper connection between these ideas.

But upon extension, it is very much identical. The value $\mathcal{N}$ is a domain; whether it be a union of neighborhoods, or a neighborhood. And quite exactly, since $A_\mu$ is discoverable in each; there is no difficulty. But it requires a slight refinement to be exactly the inverse Abel function.

\begin{theorem}[The Inverse Abel Theorem]\label{thmINVABEL}
Let $s \in \mathcal{G} \subseteq \mathcal{P}$ be a domain within the weak Fatou set closed under $s \mapsto s+1$ and $s \mapsto s-1$, such that; on all compact subsets $\mathcal{N} \subset \mathcal{G}$:

$$
\limsup_{k\to\infty} \left|\left| \frac{1}{\mu \beta(s+k)}\right|\right|_{\mathcal{N}} \le A\\
$$

Then, when $\Re\lambda > \log A$, the function $\tau(s)$ is holomorphic on a domain $\mathcal{G}^*$ where $\mathcal{G}/\mathcal{G}^*$ is measure zero up to an area measure (it consists only of branch cuts).
\end{theorem}

\begin{proof}
Let us take an arbitrary compact domain $\mathcal{N}\subset\mathcal{G}$; then necessarily:

$$
\limsup_{k\to \infty} \left|\left|\frac{1}{\mu\beta(s+k)}\right|\right|_{\mathcal{N}} = A_\mu \le A\\
$$

For $\Re \lambda > \log A$ we know, from Corollary \ref{corPB} that:

$$
F(s+k) = \beta(s+k) + \tau(s+k)\\
$$

Is holomorphic on $\mathcal{N}_k^*\subset \mathcal{N}$; which $\mathcal{N}/\mathcal{N}_k^*$ is area measure zero for all $k \in \mathbb{Z}$. Since $\mathcal{G}$ is closed under the translations $s \mapsto s+1$ and $s\mapsto s-1$; we have:

$$
\mathcal{G}^* = \bigcup_{k \in \mathbb{Z}}\bigcup_{\mathcal{N}_k^*\subset \mathcal{G}} \mathcal{N}_k^*+k\\
$$

Which means for all forward orbits and all backward orbits; all compact neighborhoods (minus the branch cuts from the backwards orbits) within $\mathcal{G}$; the final inverse Abel function is holomorphic.

Upon which, if we take the area measure integral:

$$
\int_{\mathcal{G}/{\mathcal{G}^*}} \,dA = 0\\
$$

Because the only difference between $\mathcal{G}$ and $\mathcal{G}^*$ are a bunch of branch cuts in the domain; caused by a countable amount of zeroes (at worse). This is an exercise in set theory arithmetic once you identify:

$$
\mathcal{G} = \bigcup_{k \in \mathbb{Z}}\bigcup_{\mathcal{N}\subset \mathcal{G}} \mathcal{N}+k\\
$$
 And,

$$
\int_{\mathcal{N}/{\mathcal{N}_k^*}} \,dA = 0\\
$$

The measure theory implores the result. In which:

$$
\mathcal{G}/\mathcal{G}^* = \bigcup_{k \in \mathbb{Z}}\bigcup_{\mathcal{N} \subset \mathcal{G}} \mathcal{N}+k/\mathcal{N}_k^*+k\\
$$

And:

$$
\int_{\mathcal{G}/\mathcal{G}^*}\, dA \le \sum_{k \in \mathbb{Z}} \sum_{\mathcal{N} \subset \mathcal{G}} \int_{\mathcal{N}/\mathcal{N}^*_k} \, dA = 0\\
$$
\end{proof}

For the remainder of this paper: we will focus on examples of this theorem in action, consequences of this theorem, and slight modifications of this theorem.\\

Before we advance into the next section the author would like to add a comment. The convergence of $\tau^n$ is dependent on $\lambda$, but if as we vary $\lambda$, and we have $\Re \lambda > A$; convergence will be locally uniform in $\lambda$. As a result; where-ever $\Re \lambda > A$ we can be assured that $F$ is holomorphic in $\lambda$. So the weak Fatou set may change as we move $\lambda$, and the value $A_\mu$ may move, but so long as we have these compact bounds for varying $\lambda$; we are guaranteed holomorphy in $\lambda$. Of which, the point we are trying to make is that this result is at least locally holomorphic in $\lambda$. This is apparent from the locality of the bounds. The same comment holds for $\mu$; we are locally holomorphic in $\mu$.

\section{The Shell-Thron Region; a Toy Model}\label{sec14}
\setcounter{equation}{0}

We are going to look at where the $\beta$ approximation is the most valuable. This is in the Shell-Thron region. But we are going to ease into the discussion. We recall the definition of the Shell-Thron region:

\begin{definition}[The Shell-Thron Region]\label{def:Sh-Th}
The Shell-Thron region $\mathfrak{S}$ is defined as the region:

$$
\mathfrak{S} = \{b \in \mathbb{C}\,|\,\lim_{n\to\infty}\exp_b^{\circ n}(1)\,\,\text{converges}\}\\
$$
\end{definition}

And from this; we will introduce the Julia set, and the Fatou set. We will again, make the change of variables $b = e^\mu$.

\begin{definition}[The Exponential Fatou Set]\label{def:FATEXP}
The Exponential Fatou set $\mathcal{F}_\mu$ is the Fatou set of $e^{\mu z} = \exp_b(z)$. It is an open domain, in which for all $z \in \mathcal{F}_{\mu}$, there exists a neighborhood $\mathcal{N} = \{ y \in \mathbb{C} |\,|y-z| < \delta\}$ such that $\mathcal{N} \subset \mathcal{F}_\mu$, and the orbits $\exp^{\circ n}(\mathcal{N})$ are a normal family. Which means there exists some subsequence $n_k$ such that $\lim_{k\to\infty} \exp^{\circ n_k}(z)$ converges uniformly on $\mathcal{N}$--equivalently, $||\exp^{\circ n}(z)||_{z\in\mathcal{N}}$ is a bounded sequence.
\end{definition}

And its complement, the exponential Julia set:

\begin{definition}[The Exponential Julia Set]\label{def:FATEXP}
The Exponential Julia set $\mathcal{J}_\mu$ is the Julia set of $e^{\mu z} = \exp_b(z)$. It is a closed domain, in which for all $z \in \mathcal{J}_{\mu}$, every neighborhood $\mathcal{N} = \{ y \in \mathbb{C} |\,|y-z| < \delta\}$ the orbits $\exp^{\circ n}(\mathcal{N})$ are not a normal family--not uniformly bounded. Equivalently, $\mathcal{J}_\mu = \mathbb{C} / \mathcal{F}_\mu$.
\end{definition}

We have opted to call these the exponential Fatou/Julia set because this is a rather inadequate definition of the Fatou set and the Julia set. But for entire functions it is equivalent to the standard definition involving much deeper topological arguments. So to make a distinction that this is only the definition for our purposes, I added the quip exponential. 

We want to use the exponential Fatou/Julia set to help describe the weak Fatou/Julia set. And we want to show that on the interior of $\mathfrak{S}$, that the weak Julia set is small enough for the weak Fatou set to be interesting. We enter then, into the more topological nature of our paper; what the weak Fatou set, and what the weak Julia set look like.\\

To begin, we will look at $\mu = \ln(2)/2$ where $e^\mu = \sqrt{2}$. And we want to effectively describe the weak Fatou/Julia set. This case will give us a good picture of how everything will behave in the Shell-Thron region. To qualify a value as belonging to the weak Julia set; we need to only qualify where:

$$
\beta(s+k) \in \mathcal{J}_\mu\,\,\text{for}\,\,k>K\\
$$

This is drawn from the relationship that:

$$
\limsup_{k\to\infty} \left|\left|\frac{1}{\exp_b^{\circ k}(z)}\right|\right|_{\mathcal{N}} < \infty\\
$$

When $\mathcal{N} \subset \mathcal{F}_\mu$; so that the only time this identity can fail is when $z \in \mathcal{J}_\mu$ and $\mathcal{N}$ is a neighborhood of $z$. Since for large $k$ we can expect $\beta(s+k)$ to look like the orbit $\exp^{\circ k}(\beta(s))$; this is a good indicator of where the weak Fatou set resides, and where the weak Julia set resides.

Now, what's especially important to remember, is that $\beta(-\infty) = 0$ and for $\beta(\Re(s) < 0)$ we are in a neighborhood of zero. So, as we push forward with $\beta(s+k)$ we look a lot like the orbits $\exp_{\sqrt{2}}^{\circ k}(\mathcal{N})$ for $\mathcal{N}$ a neighborhood of $0$. But additionally, the weak Fatou set, is larger than just the points encompassed by this. To understand, we have to examine the weak Julia set further.

The line $[4,\infty) \subset \mathcal{J}_{\log(2)/2}$ belongs to the exponential Julia set. Now clearly $[4,\infty) \in \mathcal{P}$ the weak Fatou set; but, if $\beta(s): \mathcal{A} \to [4,\infty)$ then, we can get close to $\mathcal{A} \subset \mathcal{B}$ the weak Julia set. Which means, for values $z= \beta(s)$ which belong to the Julia set; we are not necessarily in the weak Julia set, but we are likely so; or there are neighboring points.

For example, consider:

$$
\beta(x+\pi i) = \OmSum_{j=1}^\infty \frac{\sqrt{2}^z}{1-e^{j-x}}\,\bullet z\\
$$

Which has singularities at $x = k \in \mathbb{N}_{\ge 1}$. The function $\beta(x+\pi i)$ is real valued for $x \in \mathbb{R}/\mathbb{N}_{\ge 1}$; and tends to infinity as $\lim_{x\to k^+} \beta(x+\pi i) = +\infty$. Thus within this domain, we must have $x\in (0,\delta)$ that $\beta(x + k+\pi i) \in \mathcal{J}_{\log(2)/2}$. 

There is a famous result in complex dynamics, which we reference from John Milnor \cite{milnor_2000}. If $z \in \mathcal{J}$ is in the Julia set of a function $f$; then for any neighborhood $\mathcal{N} = \{y \in \mathbb{C}\,|\,|y-z| < \delta\}$, the orbits $f^{\circ n}(\mathcal{N})$ are dense in $\mathcal{J}$. 

So, since we know that $\exp_{\sqrt{2}} : [4,\infty) \to [4,\infty)$; and each $\beta(x+k+\pi i) \in [4,\infty)$ while $x \in (0,\delta)$; we can expect that this displays similar chaos as the Julia set for any open neighborhood about this. Which is to say, for any neighborhood of $x$, as $k \to \infty$; we are approaching every point in the Julia set. This does not force every value of the Julia set to be in the weak Julia set, but it does tell us that; when $x\in(0,\delta)$ for small enough delta, we must be in the weak Julia set.  Because, if $(0,\delta) \subset \mathcal{N}$, a neighborhood:

$$
\limsup_{k\to \infty} \left|\left|\frac{1}{\beta(s+k+\pi i)}\right|\right|_{\mathcal{N}} = \infty
$$

Despite the absolutely obvious result on the real-line that:

$$
\limsup_{k\to \infty} \left|\left|\frac{1}{\beta(x+k+\pi i)}\right|\right|_{x \in [\delta',\delta]} = 0\\
$$

It is again, our strict attention to neighborhoods that saves the day. For each neighborhood of a point in the Julia set, its orbits are dense in the Julia set. This is the exponential function; hence entire; hence the Julia set includes the point at infinity; hence every neighborhood of the Julia set gets arbitrarily close to $\infty$ on the Riemann sphere; hence the orbits get arbitrarily close to $0$ in a dense manner for every neighborhood. Which, implies the supremum norm diverges. In Figure \ref{fig:ROOT_TWO_SING} we get a close look at parts of the weak Julia set near a singularity.

\begin{figure}
    \centering
    \includegraphics[scale = 0.4]{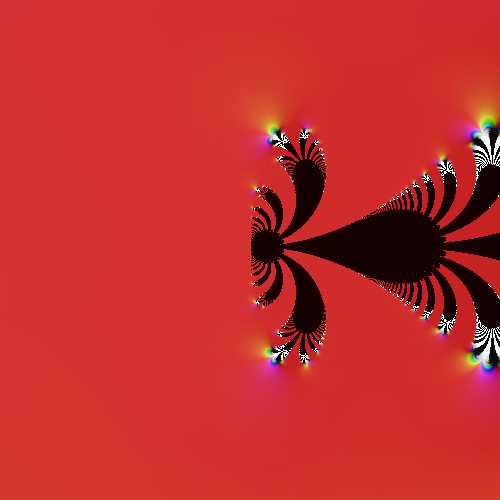}
    \caption{A close up of the singularity at $s = 5+ \pi i$ (a $0.2$x$0.2$ window) of $\beta_{1,\log(2)/2}(s)$. This gives a good picture of the weak Julia set, which continues as we push forward in the iterate. The black region is not precisely the weak Julia set; the weak Julia set is nestled somewhere within it.}
    \label{fig:ROOT_TWO_SING}
\end{figure}

To better see the weak Julia set; we can produce a more exact mapping. And from that, we can display the entire weak Julia set of $\beta_{1,\log(2)/2}$. In Figure \ref{fig:ROOT_TWO_JULIA} we have highlighted the weak Julia set by assigning a white pixel; and a black pixel for the weak Fatou set. The entire weak Julia set will be a grid of these fractals.

\begin{figure}
    \centering
    \includegraphics[scale=0.5]{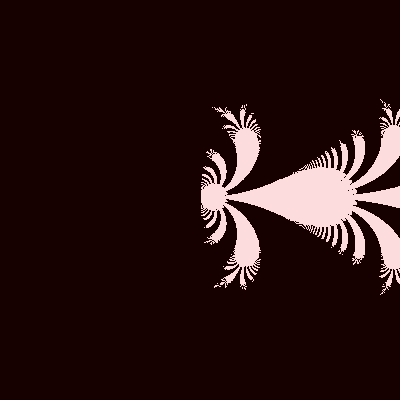}
    \caption{A better look at the weak Julia set of $\beta_{1,\log(2)/2}$; which occurs periodically at each $j + (2k+1)\pi i$ for all $j,k \in \mathbb{Z}$. This is graphed on a $0.2$x$0.2$ window. This describes the entire weak Julia set.}
    \label{fig:ROOT_TWO_JULIA}
\end{figure}

\begin{figure}
    \centering
    \includegraphics[scale = 0.4]{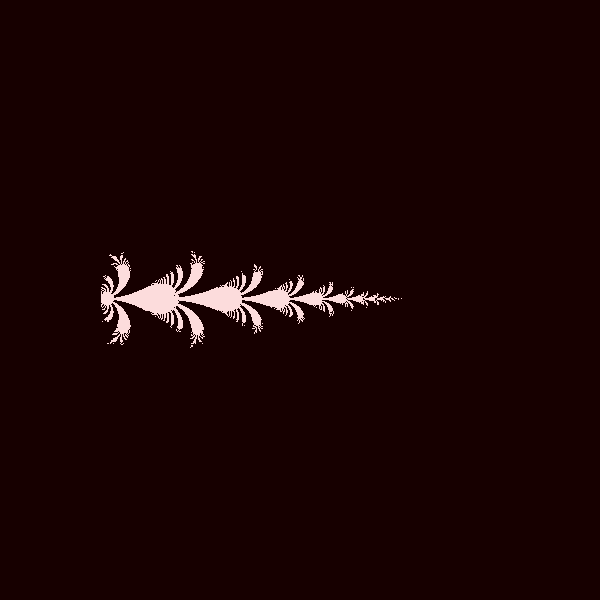}
    \caption{A bigger picture about the same singularity; it describes the exact fractal which spawns from $j+(2k+1)\pi i$ for $j,k \in 
    \mathbb{Z}$.}
    \label{fig:ROOT_TWO_JULIA_WIDE}
\end{figure}

\section{The Shell-Thron Region}\label{sec15}
\setcounter{equation}{0}

The point of this section is to get the best result we're able to get on the Shell-Thron Region. We'll stick specifically to the interior of the Shell-Thron region. For that we will call $\mathfrak{S}^{o}$ the interior of the Shell-Thron region $\mathfrak{S}$.

Then, there exists a geometrically attracting fixed point $\omega_\mu$, such that:

\begin{align*}
\exp_b(\omega_\mu) &= \omega_\mu\\
\lim_{n\to\infty} \exp_b^{\circ n}(0) &= \omega_\mu\\
\left|\frac{d}{dz}\Big{|}_{z=\omega_\mu} \exp_b(z)\right| = |\mu \omega_\mu| &< 1\\
\end{align*}

We've attached two graphs in Figure \ref{fig:Sh-Th-fixed} and Figure \ref{fig:Sh-Th-fixed-var}. These describe the fixed points found in the Shell-Thron region. The black pixels describe divergence; and the coloured portion are the value of the fixed point for $b^z$ or through the change of variables $e^{\mu z}$.

\begin{figure}
    \centering
    \includegraphics[scale=0.4]{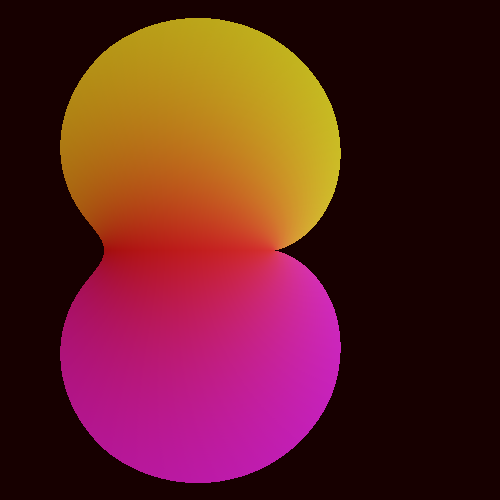}
    \caption{This is a graph of $\omega$ for the variable $b$. This describes the Shell-Thron region; and assigns the value of the fixed point of $\exp_b(z) = b^z$. This is precisely a graph of $\lim_{n\to\infty} \exp_b^{\circ n}(0)$; where divergence is assigned a black pixel.}
    \label{fig:Sh-Th-fixed}
\end{figure}

\begin{figure}
    \centering
    \includegraphics[scale = 0.4]{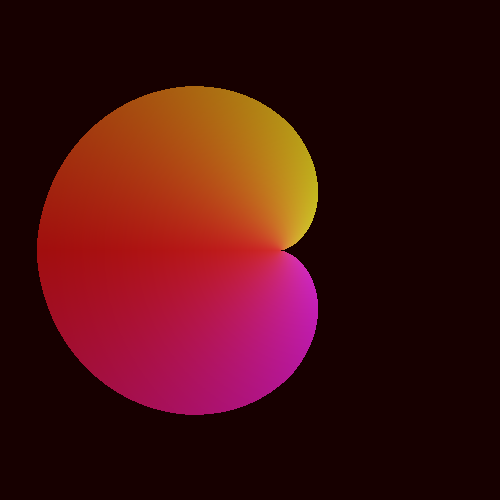}
    \caption{This is a graph of $\omega$ for the variable $\mu$. This describes the Shell-Thron region; and assigns the value of the fixed point of $e^{\mu z}$. This is precisely a graph of $\lim_{n\to\infty} \exp_b^{\circ n}(0)$; where divergence is assigned a black pixel; but we've made the change of variables $b = e^\mu$.}
    \label{fig:Sh-Th-fixed-var}
\end{figure}

Additionally, we can expect $0$ to be in the Fatou set of $\exp_b$; and therefore; there's a neighborhood $\mathcal{N}$ about zero, in which $\exp_b^{\circ n}(\mathcal{N}) \to \omega_\mu$. This means we are in the best possible situation. Any orbit of $0$ eventually coalesces to $\omega_\mu$; and additionally, any orbit of $\beta(s) \approx 0$ means for large enough $k>K$ that $\beta(s+k) \to \omega_\mu$.

In such a case; we obtain a holomorphic tetration everywhere in the weak Fatou set; so long as $\Re \lambda > -\log(|\mu \omega_\mu|)$. Proving this result isn't particularly hard; just nuanced. To begin, we can implore the identity:

$$
\log^{\circ n} \beta(s+n+k) \sim \beta(s+k)\,\,\text{while}\,\,s \in \mathcal{P}\,\,\text{as}\,\,k\to\infty\\
$$

And thus:

$$
\exp^{\circ k} \beta(s) \sim \beta(s+k)\,\,\text{while}\,\,s \in \mathcal{P}\,\,\text{as}\,\,k\to\infty\\
$$

Therefore, since we know for all $\Re(s) < -K$, large enough $K$ that $\beta(s)$ is in an arbitrarily small neighborhood of zero; it must belong to the Fatou set of $\exp_b(z)$; and further its iterates converge to the fixed point $\omega_\mu$. Thereby, we are given the following theorem:

\begin{theorem}[The Inverse Abel Function On The Shell-Thron Region Theorem]\label{thmINVABLSHTH}
Let $b \in \mathfrak{S}^{o}$ be in the interior of the Shell-Thron region, and let $b = e^\mu$. Let $\mathcal{P}_\lambda$ be the weak Fatou set of $\beta_{\lambda,\mu}(s)$. If $\Re \lambda > - \log(|\mu \omega_\mu|)$, where:

$$
\omega_\mu = \lim_{n\to\infty} \exp^{\circ n}_b(0)\\
$$

There exists a holomorphic inverse Abel function $F_\lambda : \widetilde{\mathcal{P}_\lambda} \to \mathbb{C}$ such that:

$$
\int_{\mathcal{P}_\lambda/\widetilde{\mathcal{P}_\lambda}} \, dA = 0\\
$$

For the standard Lebesgue area measure $dA$ on $\mathbb{R}^2$; and,

$$
F_\lambda(s+1) = e^{\mu F_\lambda(s)}\\
$$

While,

$$
F_\lambda(s+2 \pi i/\lambda) = F_\lambda(s)\\
$$
\end{theorem}

\begin{remark}
The function $F_\lambda$ is holomorphic in $\lambda$; and most importantly, determines the period of $F$ and the position of the essential singularities. At each $s$ where $\lambda(j-s) = (2k+1)\pi i$ for $j,k \in \mathbb{Z}$ is precisely where the essential singularities will appear in our function $\beta_\lambda(s)$; and typically produce fractal behaviour in $\beta$. So this value controls where the worst of our fractals will appear, and whether they'll appear on a square grid; or on a more sheared grid ($\lambda$ is real, versus $\lambda$ complex).

The period of $F_\lambda$ is precisely $2 \pi i/\lambda$. When the period is complex and not purely imaginary; expect some very weird fractals and a plethora of logarithmic singularities--but again, they are negligible under an area map.
\end{remark}

\begin{proof}
By The Inverse Abel Theorem \ref{thmINVABEL}; we need only show that:

$$
\limsup_{k\to\infty} \left|\left|\frac{1}{\mu \beta(s+k)}\right|\right|_\mathcal{N} \le \frac{1}{|\mu \omega_\mu|}\\
$$

For all $\mathcal{N} \subset \mathcal{P}_\lambda$. This follows by the asymptotic, for all $s \in \mathcal{P}$:

$$
\beta(s+n) \sim \exp^{\circ n}(\beta(s))\\
$$

This follows because $\exp(\beta(s)) \sim \beta(s+1)$ for large $\Re(s)$. It does so normally, as well. Which means that:

$$
\left| \left| \exp(\beta(s)) - \beta(s+1)\right|\right|_{\mathcal{N}} \to 0\\
$$

And, going by induction, assume a similar result holds for $k < n$; then,

\begin{align*}
\left|\left| \exp^{\circ n} \beta(s) - \beta(s+n)\right|\right| &\le \left|\left| \exp^{\circ n-1} \exp \beta(s) - \exp^{\circ n-1} \beta(s+1)\right|\right|\\
&\,+ \left|\left|\exp^{\circ n-1} \beta(s+1) - \beta(s+n)\right|\right|\\
&\to 0\,\,\text{as}\,\,\Re(s) \to \infty\\
\end{align*}

Now, for $\Re(s) < -R$ for $R$ significantly large; and $s \in \mathcal{P}$:

$$
\lim_{n\to\infty} \exp^{\circ n} \beta(s) = \omega_\mu\\
$$

But this equates to $\beta(s+n)$ eventually. Which implies for all $s \in \mathcal{P}$:

$$
\lim_{n\to\infty} \beta(s+n) = \omega_\mu\\
$$

And furthermore, does so normally:

$$
\left|\left| \beta(s+n) - \omega_\mu\right|\right|_{\mathcal{N}} <\epsilon\\
$$

For large enough $n>N$ found for any $\epsilon > 0$. Thereby, we know that:

$$
\limsup_{k\to\infty} \left|\left|\frac{1}{\mu \beta(s+k)}\right|\right|_\mathcal{N} \le \frac{1}{|\mu \omega_\mu|}\\
$$
\end{proof}

\section{The Fractals of the weak Fatou set}\label{sec16}
\setcounter{equation}{0}

To anoint this chapter with some visuals; we have included a bunch of graphs of the weak Julia set/weak Fatou set for various $\mu$. The weak Julia set is assigned a white pixel; and the weak Fatou set is assigned a black pixel. We have kept the multiplier $\lambda = 1$ fixed; but since we are holomorphic in $\lambda$, expect a continuous deformation of the following graphs, as we move $\lambda$. 

Increases or decreases in the real part of $\lambda$ equate to a vertical compression, or stretching (respectively), of the strip where the singularities border. Increases or decreases in the imaginary part of $\lambda$ result in lateral shifts of the singularities for each row of singularities. This looks like a shear; by which a square becomes a parallelogram.

For that reason, we can identify the black region as where we are holomorphic, and the white fractals as where we are non-holomorphic. The remainder of this paper will talk about the non-holomorphic case; which gives The Asymptotic Theorem of the third kind. But, for the moment; on the black area we have an inverse Abel function (up to a measure zero area set).

Corresponding to each weak Fatou/Julia set to follow are the following inverse Abel functions. One can discern precisely where the discontinuities appear; right where the weak Julia set happens. Thus we juxtapose each fractal, with what the holomorphic function coinciding with it looks like.

Additionally; due to the laws of bifurcations; the actual weak Julia set will be a subset of the white fractal; there is no work around to this, as we are already approaching layers of overflows in this construction. By which, these are merely approximations of the weak Julia set; they are most definitely more complex.

A detailed manner of how these results are coded in Pari-GP will be written up later. They are largely amalgamations of code privy to the Tetration forum. Additionally, these heuristics are derived from similar short cut mathematics; coding iterated exponentials is a nasty business.

\begin{figure}
    \centering
    \includegraphics[scale=0.4]{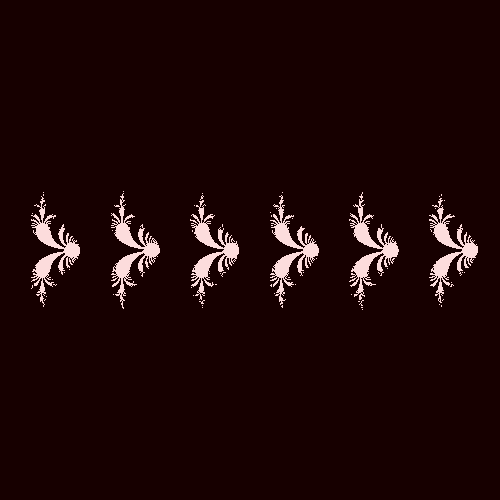}
    \caption{The weak Julia/Fatou set for $\mu = 1-e$ and $\lambda = 1$. The domain is graphed over $0 \le \Re(s) \le 2 \pi$ and $0 \le \Im(s) \le 2\pi$. The corresponding inverse Abel function is shown in Figure \ref{fig:INV_ABEL_1_E}.}
    \label{fig:JUL_1_E}
\end{figure}

\begin{figure}
    \centering
    \includegraphics[scale=0.4]{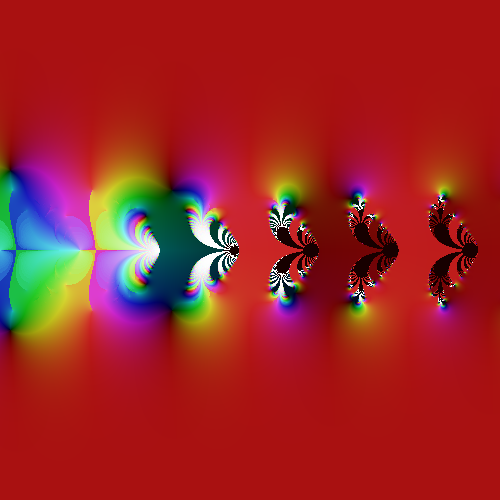}
    \caption{The function $\beta_{\lambda,\mu}(s) + \tau_{\lambda,\mu}(s)$ for $\lambda =1$ and $\mu = 1-e$. This is graphed over the domains $0 \le \Re(s) \le 2\pi$ and $0 \le \Im(s) \le 2\pi$. The corresponding weak Julia set is found in Figure \ref{fig:JUL_1_E}.}
    \label{fig:INV_ABEL_1_E}
\end{figure}

\begin{figure}
    \centering
    \includegraphics[scale=0.4]{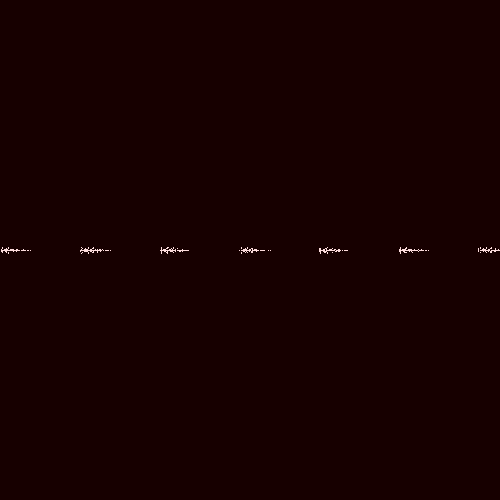}
    \caption{The weak Julia/Fatou set for $\mu = \ln(3)/3$ and $\lambda = 1$. The domain is graphed over $0 \le \Re(s) \le 2 \pi$ and $0 \le \Im(s) \le 2\pi$. The corresponding inverse Abel function is shown in Figure \ref{fig:INV_ABEL_CUBE_3}.}
    \label{fig:JUL_CUBE_3}
\end{figure}

\begin{figure}
    \centering
    \includegraphics[scale=0.4]{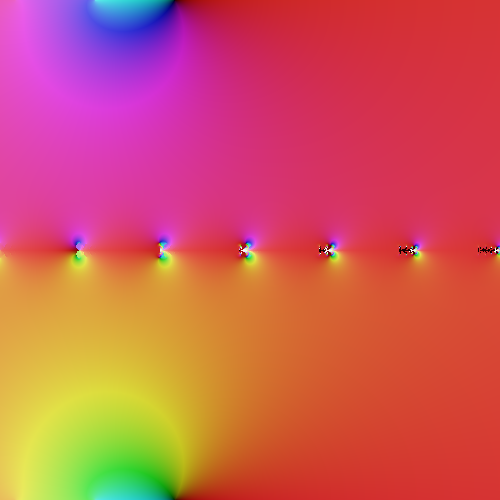}
    \caption{The function $\beta_{\lambda,\mu}(s) + \tau_{\lambda,\mu}(s)$ for $\lambda =1$ and $\mu = \ln(3)/3$. This is graphed over the domains $0 \le \Re(s) \le 2\pi$ and $0 \le \Im(s) \le 2\pi$. The corresponding weak Julia set is found in Figure \ref{fig:JUL_CUBE_3}.}
    \label{fig:INV_ABEL_CUBE_3}
\end{figure}

\begin{figure}
    \centering
    \includegraphics[scale=0.4]{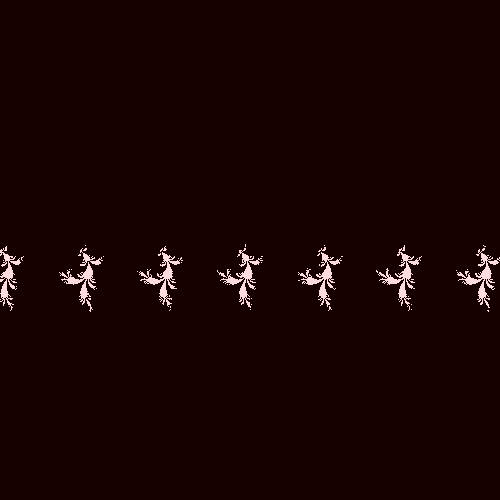}
    \caption{The weak Julia/Fatou set for $\mu = \ln(3)/3 - 0.5i$ and $\lambda = 1$. The domain is graphed over $0 \le \Re(s) \le 2 \pi$ and $0 \le \Im(s) \le 2\pi$. The corresponding inverse Abel function is shown in Figure \ref{fig:INV_ABEL_CUBE_3_MINUS_POINTFIVE_I}.}
    \label{fig:JUL_CUBE_3_MINUS_POINTFIVE_I}
\end{figure}

\begin{figure}
    \centering
    \includegraphics[scale=0.4]{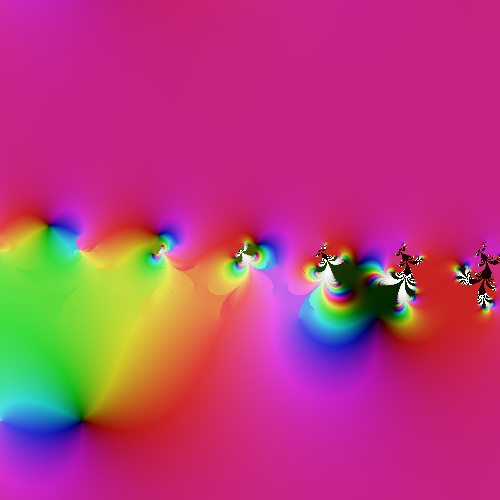}
    \caption{The function $\beta_{\lambda,\mu}(s) + \tau_{\lambda,\mu}(s)$ for $\lambda =1$ and $\mu = \ln(3)/3-0.5i$. This is graphed over the domains $0 \le \Re(s) \le 2\pi$ and $0 \le \Im(s) \le 2\pi$. The corresponding weak Julia set is found in Figure \ref{fig:JUL_CUBE_3_MINUS_POINTFIVE_I}.}
    \label{fig:INV_ABEL_CUBE_3_MINUS_POINTFIVE_I}
\end{figure}

\begin{figure}
    \centering
    \includegraphics[scale=0.4]{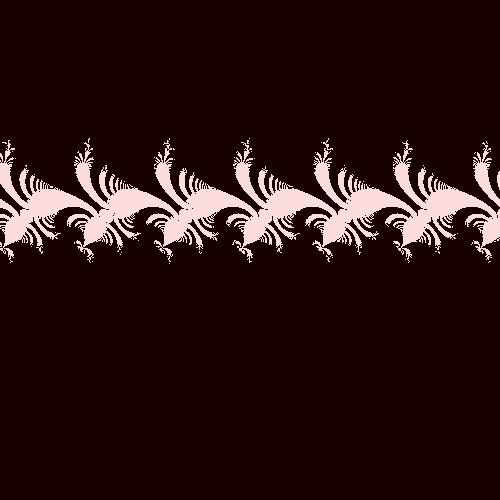}
    \caption{The weak Julia/Fatou set for $\mu = 0.3 + i$ and $\lambda = 1$. The domain is graphed over $0 \le \Re(s) \le 2 \pi$ and $0 \le \Im(s) \le 2\pi$.}
    \label{fig:JUL_point_3_plus_I}
\end{figure}

\begin{figure}
    \centering
    \includegraphics[scale=0.4]{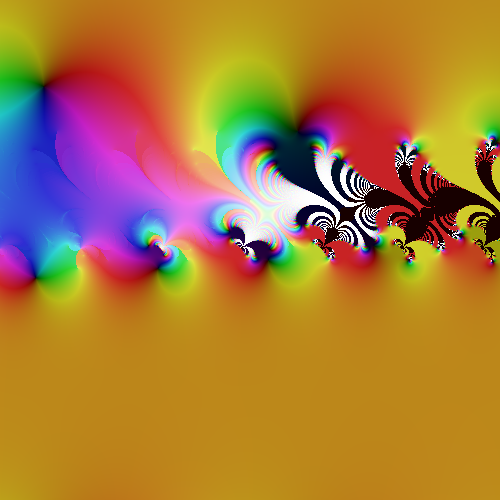}
    \caption{The function $\beta_{\lambda,\mu}(s) + \tau_{\lambda,\mu}(s)$ for $\lambda =1$ and $\mu = 0.3+i$. This is graphed over the domains $0 \le \Re(s) \le 2\pi$ and $0 \le \Im(s) \le 2\pi$. The corresponding weak Julia set is found in Figure \ref{fig:JUL_point_3_plus_I}.}
    \label{fig:INV_ABEL_POINT_3_PLUS_I}
\end{figure}

\begin{figure}
    \centering
    \includegraphics[scale=0.4]{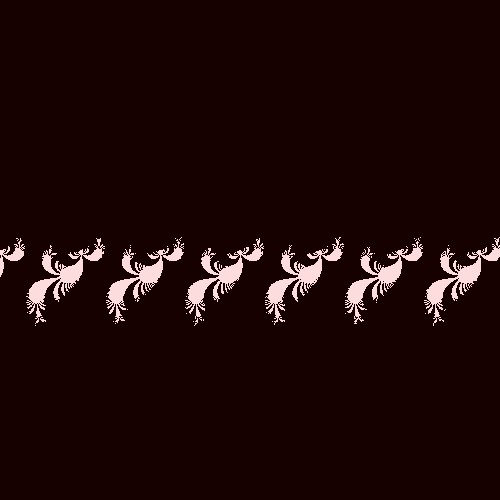}
    \caption{The weak Julia/Fatou set for $\mu = -0.3-i$ and $\lambda = 1$. The domain is graphed over $0 \le \Re(s) \le 2 \pi$ and $0 \le \Im(s) \le 2\pi$. The corresponding inverse Abel function is shown in Figure \ref{fig:INV_ABEL_NEGATIVE_POINT_3_MINUS_I}.}
    \label{fig:JUL_NEGATIVE_POINT_3_MINUS_I}
\end{figure}

\begin{figure}
    \centering
    \includegraphics[scale=0.4]{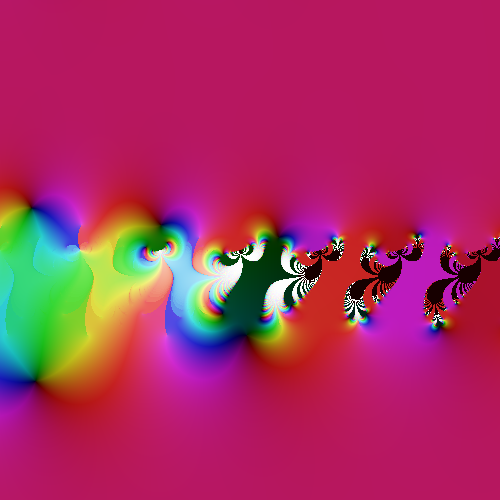}
    \caption{The function $\beta_{\lambda,\mu}(s) + \tau_{\lambda,\mu}(s)$ for $\lambda =1$ and $\mu = -0.3-i$. This is graphed over the domains $0 \le \Re(s) \le 2\pi$ and $0 \le \Im(s) \le 2\pi$. The corresponding weak Julia set is found in Figure \ref{fig:JUL_NEGATIVE_POINT_3_MINUS_I}.}
    \label{fig:INV_ABEL_NEGATIVE_POINT_3_MINUS_I}
\end{figure}

\chapter{Asymptotics on the weak Julia set}

\section{Asymptotics on the weak Julia set}\label{sec18}
\setcounter{equation}{0}

This chapter will be devoted to describing the best result we are able to attain on the weak Julia set. This result is no better epitomized than in the case $\mu = 1$ ($b = e$) and $\lambda =1$. This case was handled in great detail by Sheldon Levenstein and myself. The reason this case is so interesting is because $\mathcal{B}  = \mathbb{C}$. The weak Julia set is precisely the complex plane. The results of this section can be more clearly attributed to Sheldon Levenstein; as the author had mistaken a distinction with supremum norms; but much of the work was mutual.

The raw paradoxical nature of this result is that:

$$
\lim_{n\to\infty} \frac{1}{\beta(s+n)} = 0\,\,\text{for all}\,\,s \neq j + (2k+1)\pi i\,\,\text{for}\,\,j,k\in\mathbb{Z}\\
$$

But, in no way, shape, or form does this limit converge normally. In fact, for any $\mathcal{N} \subset \mathbb{C}$:

$$
\limsup_{n\to\infty} \left|\left|\frac{1}{\beta(s+n)}\right|\right|_{\mathcal{N}} = \infty\\
$$

This tells us, instantly, that:

$$
\lim_{n\to\infty}\tau^{n}(s) = \tau(s)\\
$$

Converges pointwise, but converges nowhere uniformly. To that note, the question we are addressing in this chapter is: how close does it come to converging uniformly? In doing so; we will be evaluating the asymptotic series which is induced.

To that end:

$$
F^n(s) = \beta(s) + \tau^n(s) = \sum_{k=0}^\infty a_k^n(s_0)(s-s_0)^k\\
$$

Where, directly:

\begin{align*}
\frac{1}{\limsup_{k\to \infty}|a_k^n(s_0)|^{1/k}} &= R^n(s_0)\\
\lim_{n\to\infty} R^n(s_0) &= 0\\
\end{align*}

Proving this absolutely for $s_0 \in \mathcal{B}$ (in the weak Julia set), will be our first goal; where the second goal is to attempt at deriving how this converges. This requires a rabbit hole of analysing the summative term $\rho^j$.

To get the hang of this case; we're going to make a change of variables in $z$. This means, we aren't going to look at $\bullet z$, but rather an alteration of this expression. And the first thing we're going to show in this chapter, is that this change of variables is valid:

$$
\frac{1}{\beta(s)} = \OmSum_{j=1}^\infty (1+e^{\lambda(j-s)})e^{ -\frac{\mu}{z}}\bullet z\\
$$

This isn't as foreign a result as it looks. Letting $I(z) = 1/z$, then:

$$
I\left(\frac{e^{\mu I(z)}}{1+e^{-\lambda s}}\right) = \left(1+e^{-\lambda s}\right)e^{-\mu/z}\\
$$

Setting:

$$
\tilde{q}_j(s,z) = I\left(\frac{e^{\mu I(z)}}{1+e^{\lambda (j- s)}}\right)
$$

And keeping with our previous notation:

$$
q_j(s,z) = \frac{e^{\mu z}}{1+e^{\lambda (j- s)}}\\
$$

So,

$$
I(q_j(s,I(z))) = \tilde{q}_j(s,z)\\
$$

To simplify the following; we'll write it in the bullet notation:

$$
I \bullet q_j(s) \bullet I \bullet z = \tilde{q}_j(s) \bullet z\\
$$

Attaching an $\OmSum_{j=1}^\infty$:

$$
h(s) = \OmSum_{j=1}^\infty\tilde{q}_j(s) \bullet z =\OmSum_{j=1}^\infty I \bullet q_j(s) \bullet I \bullet z\\
$$

Now, $I \bullet I = 1$; so the above identity is a conjugation. Which, is to say; we can group terms as such:

$$
I \bullet q_j(s) \bullet I \bullet I \bullet q_{j+1}(s)\bullet I \bullet z = I \bullet q_j(s) \bullet q_{j+1}(s) \bullet I \bullet z\\
$$

So when continuing this relationship:

$$
h(s) = I\left(\OmSum_{j=1}^\infty{q}_j(s) \bullet I \bullet z\right) = 1/\beta(s)\\
$$

Which, because, if we choose $z \in \mathbb{C}/\{0\}$, then $I(z) \in \mathbb{C}/\{0\}$--and convergence will be constant in $z$; so the $I$ on the inside disappears (so to speak). Now, the idea of this change of variables, is to rule out the following possibility:

$$
\limsup_{k\to\infty}||h(s+k)||_{\mathcal{N}} =0\\
$$

Which is, almost directly apparent, because it implies $h(s+k) \to 0$ and:

$$
h(s+1) \sim e^{-\mu/h(s)}\\
$$

And the essential singularity at $e^{-\mu/h(\infty)}$ if $h\to 0$ is incompatible with this asymptotic. But we want to accomplish a stronger statement of this fact. Which is that, if $h(s+k) \to 0$ as $k\to\infty$, then $s \in \mathcal{B}$. This reduces into a strong classification of the weak Fatou/Julia set. We summarize in the next theorem.

\begin{theorem}[The weak Fatou/Julia distinction Theorem]\label{thmWDSTNC}
Let $h(s) = 1/\beta(s)$, let $\mathcal{P} = \mathbb{C}/\mathcal{B}$ be the weak Fatou set of $\beta$, and let $\mathcal{B} = \mathbb{C}/\mathcal{P}$ be the weak Julia set. If,

$$
\lim_{k\to\infty} h(s_0+k) = 0,\infty
$$

Then $s_0 \in \mathcal{B}$. If,

$$
0 < \delta < |h(s_0+k)| < D < \infty\\
$$

For all $k > K$ for some $0 < \delta < D$, then $s_0 \in \mathcal{P}$.
\end{theorem}

\begin{proof}
This theorem is a tad in depth, as we have to handle a fair amount of cases. To begin, let us assume that:

$$
\lim_{k\to\infty} h(s_0+k) = 0\\
$$

We want to show that,

$$
\limsup_{k\to\infty} ||h(s+k)||_{\mathcal{N}} = \infty\\
$$

For any arbitrary neighborhood $\mathcal{N}$ about $s_0$. Pick $K$ large enough so that $|h(s_0+k)| < \epsilon$ for arbitrarily small $\epsilon$ when $k>K$. Pick an arbitrary neighborhood $\mathcal{N}$ about $s_0$. We are going to use the asymptotic:

$$
\lim_{k\to\infty} \left|\left|h(s+k+1)e^{\mu/h(s+k)} - 1\right|\right|_{\mathcal{N}} = 0
$$

Which tends to zero geometrically, simply by $h$'s functional equation:

$$
h(s+1) = (1+e^{-\lambda s}) e^{-\mu/h(s)}\\
$$

Now, by cases, assume that:

$$
\lim_{k\to\infty} \left|\left|h(s+k+1) - e^{-\mu/h(s+k)}\right|\right|_{\mathcal{N}} = 0
$$

If it doesn't tend to zero, by contradiction we are done, as it implies $\left|\left |e^{-\mu/h(s+k)}\right|\right|_{\mathcal{N}}\to \infty$; which is the statement of $s_0 \in \mathcal{B}$. Thereby, if we call $H(z) = e^{-\mu/z}$, we are given the stronger asymptotic:

$$
\lim_{k\to\infty} \left|\left|h(s+k+n) - H^{\circ n}(h(s+k))\right|\right|_{\mathcal{N}} = 0
$$

Choose $n=n_k$ such that we can limit $n\to\infty$ as $k \to \infty$ and the above limit remains arbitrarily small, then, by the triangle inequality:

$$
\left|\left|h(s+k+n)\right|\right|_{\mathcal{N}} \ge \left|\left|H^{\circ n}(h(s+k))\right|\right|_{\mathcal{N}}-\left|\left|h(s+k+n) - H^{\circ n}(h(s+k))\right|\right|_{\mathcal{N}}
$$

The iterate $H^{\circ n}(z)$ does not converge in a neighborhood of $z\approx 0$--there's an essential singularity; whereby its supremum norm gets arbitrarily large. Thereby showing that $s_0 \in \mathcal{B}$.\\

Now assume that $h(s_0+k) \to \infty$; well, any supremum norm about this gets arbitrarily large, so the limit superior will tend to infinity, thereby $s_0 \in \mathcal{B}$.\\

To show that the direct opposite of these statements imply we are in the Fatou set, is equally as formulaic to show. We want to derive that since we are bounded away from zero and infinity; we must be weakly normal. Again we will start from the asymptotic,

$$
\lim_{k\to\infty} \left|\left|h(s+k+1)e^{\mu/h(s+k)} - 1\right|\right|_{\mathcal{N}} = 0
$$

Now, additionally, we must have:

$$
h(s_0+k+1) \to e^{-\mu/h(s_0+k)}\,\,\text{as}\,\,k\to\infty\\
$$

For the point $s_0$. Thereby,

$$
\left|\left| h(s+k+1) - e^{-\mu/h(s+k)}\right|\right|_{\mathcal{N}_k} <\epsilon_k
$$

For a shrinking sequences of $\epsilon_k$ and a decreasing sequence of neighborhoods $\mathcal{N}_k$. Thereby,

$$
h(s_0+k) \to \ell\\
$$

In which,

$$
H(\ell) = e^{-\mu/\ell} = \ell\\
$$

Furthermore, we must have that this fixed point is attracting or at worse, neutral. Which means $|H'(\ell)| \le 1$ because the iterate $H^{n}(h(s_0+k)) \to \ell$, and this can't happen when the multiplier $H'(\ell) > 1$.

Now, as $H^{\circ n}(h(s_0+k)) \sim h(s_0+k+n) \to \ell$, and this is part of a Fatou set of $H$, there must be an open neighborhood in $s_0$ in which $|H^{\circ n}(h(s+k)) - \ell| < \epsilon$. Thereby, calling this neighborhood $\mathcal{N}$:

$$
\limsup_{k\to\infty} ||h(s+k)||_{\mathcal{N}} < \infty\\
$$
\end{proof}

From this, we achieve the corollary:

\begin{corollary}[Levenstein's Counter Example]

Letting $\mu =1$ and $\lambda=1$; where:

$$
\beta(s) = \OmSum_{j=1}^\infty \frac{e^{z}}{1+e^{j-s}}\bullet z\\
$$

Then $\mathbb{R} \subset \mathcal{B}$; the real-line is in the weak Julia set.
\end{corollary}

\begin{proof}
The function $\beta(x) : \mathbb{R} \to \mathbb{R}^+$. And $\beta(x+1) \sim e^{\beta(x)}$, which forces a super exponential divergence to infinity. Therefore, each point $x_0 \in \mathbb{R}$, satisfies:

$$
\frac{1}{\beta(x_0+k)} \to 0 \,\,\text{as}\,\,k\to\infty\\
$$

Therefore, $x_0\in \mathbb{R}$ is in the weak Julia set.
\end{proof}

\section{Non-uniformity on the weak Julia set; an example model}\label{sec19}
\setcounter{equation}{0}

This section will focus exactly on $\lambda =1$ and $\mu =1$ and specifically on $x \in \mathbb{R}^+$. To begin, we call upon a quick graph in Figure \ref{fig:Scr_beta}. We want to show that, because $\mathbb{R}^+ \subset \mathcal{B}$, this inverse Abel function is not holomorphic. Now, each derivative exists; because each term is holomorphic; and the convergence is so fast we can derive uniformity for all derivatives. But only on $\mathbb{R}^+$. So the inverse Abel function in Figure \ref{fig:Scr_beta} is not analytic, but it is smooth.

This poses a very difficult challenge. As, proving this inverse Abel function is smooth is constructive; proving the Taylor series diverges at each point is more difficult. This becomes not a problem of proving pointwise convergence, or proving uniform convergence, or proving normal convergence (uniform on compact subsets in $\mathbb{C}$). We have to prove the buck stops at normal convergence. And, it does so, precisely because $\mathbb{R}^+ \subset \mathcal{B}$.

\begin{figure}
    \centering
    \includegraphics[scale =0.4]{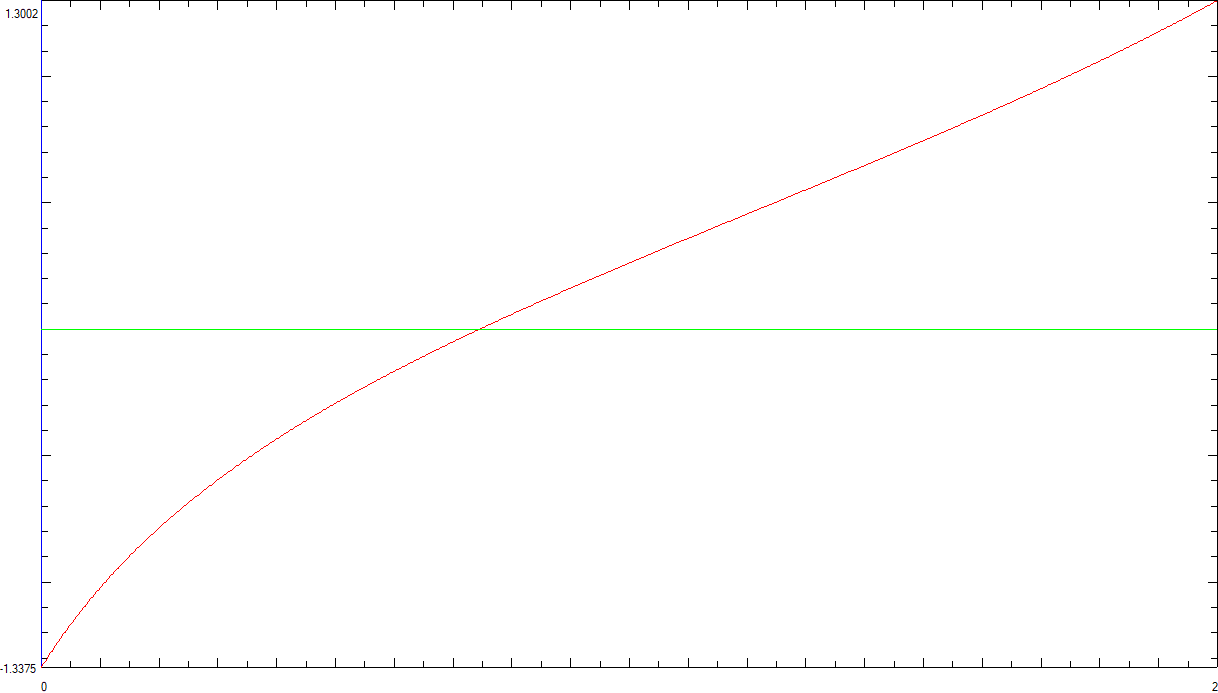}
    \caption{The red line is the inverse Abel function of $\mu =1$ and $\lambda=1$; $f(x) = \beta(x) + \tau(x)$. Done over the region $0 \le x \le 2$. The green line is the error between $f(x+1) = e^{f(x)}$. }
    \label{fig:Scr_beta}
\end{figure}

The example of $\mu =1$ and $\lambda =1$ when $x \in \mathbb{R}^+$ is intended to highlight the weak Julia set in general. There exists no holomorphic inverse Abel function on the weak Julia set (derived from $\beta$; we mean). This is a nit picky result; as we're proving the inexistence of a property; and that can always be tricky.

Commencing, pick an $x_0 \in \mathbb{R}^+$ and an arbitrary neighborhood $\mathcal{N} \subset \mathbb{C}$ about $x_0$. Then:

$$
\tau^{n+1}(x_0 +k) = \log\left(1+\frac{\tau^{n}(x_0+k+1)}{\beta(x_0+k+1)}\right) - \log(1+e^{-(x_0+k)})
$$

By the Asymptotic Theorem of the second kind \ref{thmASYM2}; for each $n \in \mathbb{N}$, as $k \to \infty$, $\tau^n(x_0+k) \to 0$ and $||\tau^n(s+k)||\to0$. Thereby, if it were to converge uniformly on $s \in \mathcal{N}$ (a neighborhood of $x_0$); we'd require that:

$$
\lim_{k \to \infty} \tau(s+k) = 0\\
$$

This in turn requires that:

\begin{align*}
    0=\lim_{k\to \infty} \tau(s+k) &= \lim_{k \to \infty}\log\left(1+\frac{\tau(s+k+1)}{\beta(s+k+1)}\right) - \log(1+e^{-(s+k)})\\
    &= \lim_{k \to \infty}\log\left(1+\frac{\tau(s+k+1)}{\beta(s+k+1)}\right)\\
    0 &= \lim_{k \to \infty}\frac{\tau(s+k+1)}{\beta(s+k+1)}\\
\end{align*}

Therefore, we can reduce to the linear case, and:

$$
\tau^{n+1}(s+k) \sim -\sum_{j=0}^{n} \frac{e^{-(s+j+k)}}{\prod_{c=1}^j \beta(s+k+c)}\\
$$

The idea of the proof to follow, is that:

$$
\left|\left|\sum_{j=0}^{n} \frac{e^{-(s+j)}}{\prod_{c=1}^j \beta(s+c)}\right|\right|_{\mathcal{N}} = \infty\\
$$

Despite the fact that it converges everywhere $x_0 \in \mathbb{R}^+$; this converges no where normally. This is quite a fascinating result, that requires a rather obvious trick, that once pointed out, is basic arithmetic.

The value:

$$
\left|\left|\frac{1}{\beta(s+c)}\right|\right|_{\mathcal{N}} \ge R > e^{2}\\
$$

For large enough $c$. Therefore, the series cannot converge. We fill in the details in the proof below.

\begin{theorem}[Levenstein's Theorem]\label{thmLVN}
Let $\mu =1$ and $\lambda =1$. The process $\tau^n(s)$ does not converge uniformly in any neighborhood of a point $x_0 \in \mathbb{R}^+$.
\end{theorem}

\begin{proof}
Pick a point $x_0 \in \mathbb{R^+}$, and pick an arbitrary neighborhood $s \in \mathcal{N}$ about the point $x_0$. By The Asymptotic Theorem of the second kind \ref{thmASYM2} the following limit holds:

$$
\lim_{k\to\infty} ||\tau^n(s+k)||_{\mathcal{N}} = 0\\
$$

From which; by the functional equation:

$$
\tau^{n+1}(s+k) \sim \frac{\tau^{n}(s+k+1)}{\beta(s+k+1)} - e^{-(s+k)}\\
$$

Which equates, for our purposes, to the statement:

$$
\left|\left|\tau^{n+1}(s+k)-\frac{\tau^{n}(s+k+1)}{\beta(s+k+1)} + e^{-(s+k)}\right|\right|_{\mathcal{N}} \to 0\\
$$

Therefore:

$$
\left|\left|\tau^{n+1}(s+k)+\sum_{j=0}^n\frac{e^{-(s+j+k)}}{\prod_{c=1}^j \beta(s+k+c)}\right|\right|_{\mathcal{N}} = P^n(k)\\
$$

Where $P^n(k) \to 0$ as $k\to\infty$. And, here's the magic; we can choose $K$ arbitrarily large so that for $k+c>K$:

$$
\left|\left|\frac{1}{\beta(s+k+c)}\right|\right|_{\mathcal{N}} \ge R > e^{2}\\
$$

Because $\beta(s+k+c)$ gets arbitrarily close to zero; and since this works for all neighborhoods; there must be one for each $k$ and $c$ for large enough $k+c>K$. Which allows us to guess the divergence of $\tau$ by a factor of $R^n$. And therefore:

$$
\left|\left|\sum_{j=0}^n\frac{e^{-(s+j+k)}}{\prod_{c=1}^j \beta(s+k+c)}\right|\right|_{\mathcal{N}} 
$$

Can only converge if $\dfrac{e^{-j}}{\prod_{c=1}^j \beta(s+k+c)} \to 0$; which cannot happen because $R$ can be made arbitrarily large. Therefore:

$$
||\tau^n(s+k)||_{\mathcal{N}} \ge \left|\left|\sum_{j=0}^n\frac{e^{-(s+j+k)}}{\prod_{c=1}^j \beta(s+k+c)}\right|\right|_{\mathcal{N}} - P^n(k)\\
$$

Depending on how $P^n(k)$ behaves as $n\to\infty$ the result is shown because the series on the right diverges. We'll play a devilish trick here. Set $n=n(k)$; and let us grow $n$. There exists a value $n>N$ such that:

$$
0<P^n(k) < 1\\
$$

This actually completes the proof. But for perfection; force $R > e^{2}$, so we can pull out a factor of $R^{n}$; where we can now say: $e^{- n}R^{n} > e^n$. By which,

\begin{align*}
    ||\tau^{n}(s+n)||_{\mathcal{N}} &\approx |e^{-s}|\sum_{j=0}^{n-1} |e^{-j}|R^{j}e^{-n}\\
    &\approx M \sum_{j=0}^n e^{j}e^{-n}\\
    &\not\to 0\,\,\text{as}\,\,n\to\infty\\
\end{align*}

Because, as simple as it is $R >e^{2}$; this expression cannot tend to zero. This can't happen if $\tau^n(s+k) \to \tau(s+k)$ converges uniformly on $\mathcal{N}$ for any $k > K$ as $n\to\infty$.

Thereby, for $x > x_0$, and $k>K$ large enough; since every compact neighborhood diverges about $x_0$; the inverse Abel function:

$$
\beta(x) + \tau(x) = f(x):\mathbb{R}^+ \to \mathbb{R}\\
$$

Is nowhere analytic.
\end{proof}

To get a better idea of what's going on in this proof, we can paraphrase Levenstein's explanation. If we take:

\begin{align*}
    \beta(s-1) + s &\approx 2\pi i n\\
    \beta(s) &\approx \exp(\beta(s-1))\\
    \beta(s) & \approx \exp(-s)\\
\end{align*}

Now this approximation is arbitrarily close for large $\Re(s)$. By which, since $\tau(s) \approx -\exp(-s)$--which is arbitrarily close; and this approximation is very good; we can expect to find points in which $\beta(s) = -\tau(s)$--which causes a singularity at $\tau(s-1)$. These singularities will begin to cluster as we approach the real-line; but additionally, eventually appear everywhere.

The reason it happens more prominently for the real line; is because we have the fastest growth here; but near any orbit of $\beta(s) = 1$ we will have very similar results. And since the weak Julia set (see below) is the entire complex plane, we can expect it to happen without warning at any point.

In Figure \ref{fig:E_1} we have only employed a couple iterations (as more causes overflows); but attempting to grab Taylor series causes the radius of convergence to sink drastically with each successive iteration. This was explained beautifully by Sheldon Levenstein; and quite robustly; it will be included in the appendix, as it isn't necessary for proof.

\begin{figure}
    \centering
    \includegraphics[scale=0.5]{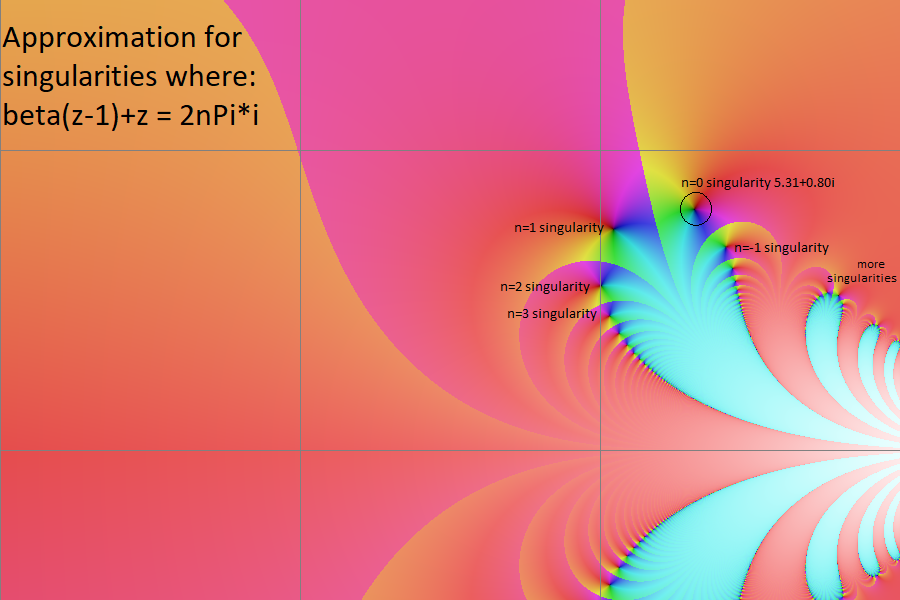}
    \caption{A graph of $\beta(z-1) +z + 2\pi i n$ (for $\mu = 1$ and $\lambda = 1$) where $n \in \mathbb{Z}$ is chosen to minimize the imaginary argument. Singularities are seen clustering towards $\mathbb{R}^+$. This graph is courtesy of Sheldon Levenstein.}
    \label{fig:Singularities}
\end{figure}

The author is going to put a little trip up at this point, which acts as a visual proof of a stronger statement of Levenstein's Theorem. We will draw the weak Julia set of $\mu =1$ and $\lambda = 1$; displayed in Figure \ref{fig:JULIA_E}. 

\begin{figure}
    \centering
    \includegraphics[scale = 0.3]{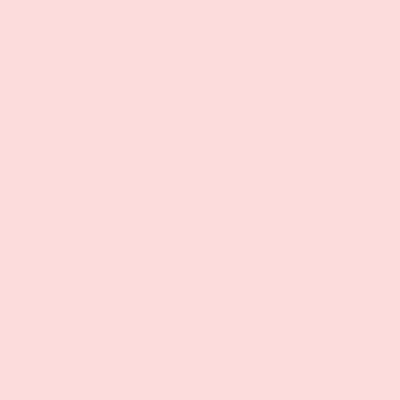}
    \caption{The weak Julia/Fatou sets of $\beta$ for $\mu =1$ and $\lambda = 1$. Mapped over a very large domain. It's all Julia set.}
    \label{fig:JULIA_E}
\end{figure}

\section{Non-uniformity on the weak Julia set}\label{sec20}
\setcounter{equation}{0}

This section is intended to generalize the results of the last section. We want to show that everywhere in the weak Julia set our iteration is nowhere holomorphic. This is significantly more difficult, but follows a similar layout as the previous section. The idea is to look for points:

$$
\beta(s-1) + \lambda s/\mu = 2\pi i n/\mu\\
$$

Which then:

$$
\beta(s) \approx \exp(\mu \beta(s-1)) \approx \exp(-\lambda s)\\
$$

Upon which, we employ that:

$$
\tau(s) \approx -\exp(-\lambda s)\\
$$

So that $\tau(s-1)$ is a singularity, because $\beta(s) + \tau(s) = 0$. The idea is to use a similar proof as in Levenstein's theorem, but we have to be more careful. We also would like to be more thorough, as this encompasses Levenstein's result; and is highly non-trivial. Where, Levenstein's theorem can be apparent just from programming in the beta method; the Taylor series will have awful convergence, and it worsens as you iterate further--which comes as a benefit of $\beta_{1,1}(x)$ growing super-exponentially and being real valued; and getting arbitrarily close to zero near the real-line.

We are going to begin by fixing a point $s_0 \in \mathcal{B}$ in the weak Julia set. We want to show, in every neighborhood of $s_0$, $\mathcal{N}$, there exists a $k>K$ and an $s_0^*$ such that:

$$
\beta(s_0^*+k) = - \tau(s_0^*+k)\\
$$

This result is a tad technical, but isn't too hard to discern if you understand what is going on. The point is that, 

$$
\liminf_{k\to\infty}\inf_{s\in \mathcal{N}}|\beta(s+k)| \to 0\\
$$

While at the same time, $\beta(s +k)$ gets arbitrarily large; thereby there must be a point somewhere nearby where $s_0^*$ causes $\beta(s_0^*+k) = -\tau(s_0^*+k)$. To do this, we need to only control $K$, and use The Asymptotic Theorem of the second kind \ref{thmASYM2}.

\begin{lemma}\label{lmaSing}
For all $s_0 \in \mathcal{B}$ (the weak Julia set), and for all neighborhoods $\mathcal{N}$ about $s_0$, there exists a value $k$ and a value $s_0^*$ such that:

$$
\beta(s_0^* + k) + \tau(s_0^* + k) = 0\\
$$

Or,

$$
\tau(s_0^* + k) = \infty\\
$$

Where $\infty$ is interpreted as non-convergence.
\end{lemma}

\begin{remark}
In this lemma we are assuming that $\tau^{n}(s_0+k)$ converges as $n\to\infty$. If it doesn't, well then, we are already done and $\tau$ is not holomorphic on $\mathcal{N}$ (a neighborhood about $s_0$). This is important to remember as the greater goal of this section. Which is to show it may converge, but it can't converge uniformly. 
\end{remark}

\begin{proof}
By The Asymptotic theorem of the second kind \ref{thmASYM2}, there exists a sequence $K_n$ such for $k > K_n$:

$$
||\tau^{n}(s + k)||_{\mathcal{N}} < \epsilon\\
$$

For arbitrary $\epsilon > 0$. Additionally, since $s_0 \in \mathcal{B}$, we must have:

$$
\limsup_{k\to\infty} \left|\left|\frac{1}{\beta(s+k)}\right|\right|_{\mathcal{N}} = \infty\\
$$

And by inverting the limit, we are given:

$$
\liminf_{k\to\infty}\inf_{s\in\mathcal{N}}|\beta(s+k)| \to 0
$$

Therefore, there exists $k_n > K_n$ such there exists an $s_n$ such that:

$$
|\beta(s_n+k_n)| < \epsilon\\
$$

Thereby, since $-\tau^n(s+k_n) : \mathcal{N} \to \mathbb{D}_{|z| < \epsilon}$; they must share a point $s_n^*$ such that they agree (for large enough $n$ and small enough $\epsilon$). This is guaranteed because $\beta(s_n+k_n)$ can be chosen to look like $e^{-\lambda(s_n+k_n)}/\mu$ and $\tau(s_n+k_n) \approx -e^{-\lambda(s_n+k_n)}/\mu$. As $\beta(s+k_n)$ will be arbitrarily large as $s$ moves; thereby $\tau$ and $\beta$ must intersect:

$$
\beta(s_n^* + k_n) + \tau^n(s_n^*+k_n) = 0\\
$$

From here, it's a little routine. We can choose a sequence $s_{n}^* \to s_0^*$ (because this is a bounded sequence).  If $\tau(s_0^*+k)$ doesn't converge, we are done. Otherwise, there must be a point $N$ such for all $n > N$, $K_n = K_N$. Thereby, there exists a $k_N > K_N$, such that:

$$
-\tau(s_0^* + k_N) = \beta(s_0^* + k_N)\\
$$
\end{proof}

This is the author's interpretation of Levenstein's argument. As Levenstein comes from an engineering background; I interpret that he interprets, if you can't hold it in your hands and build it; it may not exist. For that reason, he extended into looking for where $\beta(s-1) + \lambda s / \mu = 2 \pi i n/\mu$, where quite beautifully, our value $s_0^*$ is right next door.

Levenstein programmed a very efficient protocol for finding these values for $\mu = 1$ and $\lambda =1$. In many ways, we are just re-interpreting the method to arbitrary subsequences.

To display why this shows non-uniform convergence on $\mathcal{B}$ is, again, slightly technical; but more routine than complicated. Every point $s_0$ for every neighborhood $\mathcal{N} = \{s \in \mathbb{C}\,|\,|s-s_0| < 1/n\}$ has a point $s^*$ and an orbit $k_{\mathcal{N}}$, such that $\beta(s^*+k_{\mathcal{N}}) +  \tau(s^*+k_{\mathcal{N}}) = 0$. Thereby, the radius of convergence $R^n(s_0)$ about $s_0$ is  $< 1/n$. Letting $n\to\infty$; and we are guaranteed that $\tau(s)$ is not holomorphic.

This is done because if $\tau(s)$ is not holomorphic on $\mathcal{N}$, then neither is $\tau(s-1)$. This is the second point of us introducing translation invariant domains. We'll make this argument more cohesive in the following theorem.

\begin{theorem}[Generalized Levenstein's Theorem]\label{thmGenLev}
Let $s_0 \in \mathcal{B}$ (the weak Julia set), then for any neighborhood $\mathcal{N} = \{s\in \mathbb{C}\,|\,|s-s_0| < \delta,\,\delta>0\}$, the function $\tau(s)$ is not holomorphic.
\end{theorem}

\begin{proof}
Par Lemma \ref{lmaSing}, we know for any neighborhood $\mathcal{N}$ there exists a $k$ such that:

$$
\tau(s+k)\,\,\text{is not holomorphic on}\,\, \mathcal{N}\\
$$

As,

$$
\tau(s) = \log^{\circ k}\left(\beta(s+k) + \tau(s+k)\right) - \beta(s)\\
$$

It cannot be holomorphic on $\mathcal{N}$. Since, by Lemma \ref{lmaSing} we can shrink $\mathcal{N}$ arbitrarily small; we must have $\tau(s)$ is not holomorphic on all neighborhoods about $s_0$.
\end{proof}

This result assures us that the best result we can possibly get on the weak Julia set is an asymptotic result. There can be no holomorphy. To better understand what's going on, we have to better understand the derivatives of $\tau^n$.

\section{The Asymptotic Theorem of the third kind} \label{sec21}
\setcounter{equation}{0}

The final form of our asymptotic theorems is to understand this asymptotic as it works for derivatives and Taylor coefficients. This allows us to derive a faux asymptotic series. I say ``faux'' because it will not take the form of the usual idea of an asymptotic series; and will be a tad more involved.

Now that we know that $\tau$ is no where holomorphic on $\mathcal{B}$; we ask how close it comes to being holomorphic. The answer, qui c'est, is as close as it possibly can. Now, a corollary of The Asymptotic Theorem of the second kind \ref{thmASYM2}; which we've yet to mention; is going to play a large role.

\begin{corollary}\label{cor2ASYM2}
For all $n \in \mathbb{N}$ and all $l \in \mathbb{N}$, there exists a constant $C_{nl} > 0$; such for all $(s,\lambda) \in \mathbb{L}$ and $k > K$:

$$
\left|\left| \frac{d^l}{ds^l} \tau^n(s+k)\right|\right|_\mathcal{N} \le C_{nl} e^{-\Re\lambda k}\\
$$

For $\mathcal{N}$ a compact neighborhood of $s$ in $\mathbb{L}$.
\end{corollary}

\begin{proof}
Since $\tau^n(s+k) \to 0$ uniformly as $k \to \infty$; then so do its derivatives in $s$. Additionally, these terms converge like $\mathcal{O}(e^{-\lambda k})$ because $\tau^n(s+k) = \mathcal{O}(e^{-\lambda k})$.
\end{proof}

We are now at a point where we have to introduce what the Asymptotic Theorem of the third kind is. This is a tad technical; but simple once you understand the grand motions. There exists a sequence of functions $\sigma^j(s_0)$ which are not necessarily continuous. Such that, if $P_m(s_0;s)$ is a polynomial:

$$
P_m(s_0; s) = \sum_{j=0}^m \sigma^j(s_0) (s-s_0)^j\\
$$

Then,

$$
P_m(s_0+1;s) - \exp(P_m(s_0;s)) = \mathcal{O}(s-s_0)^{m+1}\\
$$

This implies there's an asymptotic series associated to $F(s_0) = \beta(s_0) + \tau(s_0)$; which doesn't necessarily converge; but if it doesn't, it satisfies this functional equation. By which:

$$
F(s) \sim P_m(s_0,s)\,\,\text{like}\,\,(s-s_0)^{m+1}\,\,\text{as}\,\,s\to s_0\\
$$

This result can be used to show; for instance, that $F_{1,1}(x)$ is infinitely differentiable for $x \in \mathbb{R}^+$. But additionally, this applies to everywhere in the weak Julia set where we converge. And if it doesn't converge; we are still guaranteed an asymptotic series which almost converges.

We've built up a lot of tools in this paper such that this result isn't very difficult to prove. To get a picture of how this works is a little nuanced--we have to pay close attention to the sequence $\rho^j$. As $\rho^j(s+k)$ gets arbitrarily small as we increase $k$, and each $\rho^j(s_0 + k)$ is analytic in a neighborhood of $s_0 \in \mathcal{B}$; the only way that $\tau$ diverges is if arbitrarily large derivatives of $\rho^j$ get arbitrarily large. 

Thereby, the result we are actually proving; is better phrased in a modded out space. If we mod out by $\mathcal{O}(s-s_0)^{m+1}$; then the algebra becomes a bit cleaner. Let us call the spaces:

\begin{align*}
    \mathbb{P}_m &= \{p \in \mathbb{C}[s],\,\text{the space of polynomials with coefficients in}\,\mathbb{C}\,|\, \deg(p) = m\}\\
    \mathcal{H}_{s_0} &= \{f\,\text{is holomorphic in a neighborhood of}\,s_0\in\mathbb{C}\}\\
\end{align*}

And the projection $\iota_m: \mathcal{H}_{s_0} \to \mathbb{P}_m$ which takes:

\begin{align*}
    \iota_m(f(s)) &= f(s) + \mathcal{O}(s-s_0)^{m+1}\\
    \iota_m\left(\sum_{j=0}^\infty a_j (s-s_0)^j\right) &= \sum_{j=0}^m a_j (s-s_0)^j
\end{align*}

Then the result we are really saying is that:

$$
\iota_m(\tau^n(s)) = \iota_m\left(\sum_{j=1}^n \rho^j(s)\right) \to \sum_{j=0}^m \left(\sigma^j(s_0)-\frac{\beta^(j)(s_0)}{j!}\right)(s-s_0)^j\\
$$

Where:

$$
\tau(s) = \sum_{j=0}^m \left(\sigma^j(s_0) - \frac{\beta^{(j)}(s_0)}{j!}\right)(s-s_0)^{j} + \mathcal{O}(s-s_0)^{m+1}\\
$$

Where as $m\to\infty$ we either diverge or converge; depending on whether we are holomorphic or not. Now, to the untrained eye, this may look very close to a proof of holomorphy. For that, it may help to remember the case $\mu = 1$ and $\lambda = 1$ for $x \in \mathbb{R}^+$. 

The function $\tau(x)$ will be smooth; and we'll have the exact same expansion. What's to be remembered, this expansion will be discontinuous as we leave the real-line. The expansion is precisely only valid for $x_0 \in \mathbb{R}^+$. Now the $\mathcal{O}(s-s_0)^{m+1}$ term may be complex; but it doesn't reflect the accuracy of an expansion in the complex plane. All we've done is created a polynomial $P_m$ such that:

$$
P_m(x_0+1;s) - \exp(P_m(x_0;s)) = \mathcal{O}(s-s_0)^{m+1}\\
$$

Where it just so happens, that in this case we can assign these as derivatives of $\beta(s) + \tau(s)$:

$$
P_m(x_0;s) = \sum_{j=0}^m \left(\frac{d^j}{dx^j}\Big{|}_{x=x_0} \beta(x) + \tau(x)\right)\frac{(s-x_0)^j}{j!}\\
$$

There just happens to be a path along the real-line where it is continuous. This doesn't necessarily happen for arbitrary points in the complex plane. This is precisely why we've chosen the notation $\sigma^j(x_0)$; it isn't necessarily the derivative in a meaningful sense of the word. It's just an induced coefficient of a polynomial; that can play the role of a derivative. 

To over-explain; we are not always guaranteed the identity:

$$
\frac{d}{dx_0}\sigma^j(x_0) = \sigma^{j+1}(x_0) (j+1)
$$

But we are in the special case of $\mu =1$ and $\lambda = 1$ and $x_0 \in \mathbb{R}^+$. We cannot claim this identity for $s_0 \in \mathbb{C}$; at least not in any sense of complex limits.\\

To get things rolling, we'll begin by reducing our cases to large $k>K$ in the orbit.

\begin{lemma}\label{lmaASYM3}
If $P_m(s_0+1;s)$ exists and additionally $\sigma^0(s_0+1) \neq 0$, then there exists $P_m(s_0;s)$.
\end{lemma}

\begin{proof}
There exists a polynomial $P_m(s_0;s)$ such that:

$$
\exp P_m(s_0;s) - P_m(s_0+1;s) = \mathcal{O}(s-s_0)^{m+1}\\
$$

So long as the first term of $P_m(s_0+1;s)$ is non-zero (there's no logarithmic singularity). By way, there is no polynomial $p$ in $s$ such that $\exp p = \mathcal{O}(s)$; but there's always a polynomial $p$ in $s$ such that $\exp p = B + \mathcal{O}(s)$ for $B \neq 0$.
\end{proof}

From here we have to discuss our supremum norms. We'll begin by writing, if $a(s) \in \mathcal{H}_{s_0}$; is holomorphic about $s_0$:

$$
||a(s)||_{\mathcal{N},m} = ||\iota_m(a(s))||_{\mathcal{N}}\\
$$

Which means the supremum norm on the Taylor polynomial up to $m$ terms of $a$ about a point $s_0$. The give all end all of The Asymptotic Theorem of the third kind, is that we converge uniformly under this norm for fixed $m$. A second norm that is going to be very important, is a derivative of this norm. We'll call this slightly differently:

$$
||a(s)||_m = \sum_{j=0}^m |a^{(j)}(s)|\\
$$

By which, these two norms are related as:

$$
1/\delta ||a(s_0)||_m \le ||a(s)||_{\mathcal{N},m} \le ||a(s_0)||_m\\
$$

For some $\delta > 0$. This is derived from bounding terms of the sum as:

$$
||a(s)||_{\mathcal{N},m} = \left|\left|\sum_{j=0}^m a^{(j)}(s_0) \frac{(s-s_0)^j}{j!}\right|\right|_{\mathcal{N}} \le \sum_{j=0}^m |a^{(j)}(s_0)|\\
$$

Which is true so long as $|s-s_0| \le 1$; as we are concerned with arbitrarily small neighborhoods $\mathcal{N}$ of $s_0$; it's perfectly valid to make this comparison. Bounding from below is just as straight forward.\\

Within the above norms we want to derive convergence of the $\tau$ process. We're going to follow a similar procedure as done in The Inverse Abel Theorem \ref{thmINVABEL}; but we're going to have to be a tad more careful. But the steps are similar.

We are given two different modes of convergence; which we've detailed fairly well before. But, now, we have to talk about point wise norms, as opposed to norms on compact sets. By which, we have to create a new distinction. The use of these norms is to avoid much of the trouble of the weak Julia set.

\begin{lemma}
Let $m \in \mathbb{N}$ and let $\iota_m$ be centered about $s_0+k$; such that $\beta(s)$ is holomorphic in a neighborhood of $s_0$. Then,

$$
\lim_{k\to\infty}  \left|\left|\frac{1}{\beta(s+k)}\right|\right|_{\mathcal{N},m} < \infty\\
$$
\end{lemma}

\begin{remark}
It's very important to notice that we are using the projection $\iota_m$ about $s_0 + k$; not about $s_0$. So we are changing the projection/norm as we move $k$. This means that:

$$
||a(s+k)||_{\mathcal{N},m} \le \sum_{j=0}^m |a^{(j)}(s_0 + k)|\\
$$

This allows us to talk fluidly about the orbits under this norm.
\end{remark}

\begin{proof}
If $s_0 \in \mathcal{P}$, the weak Fatou set; then $\beta(s+k) \to \omega \neq 0$; as detailed in the first half of Theorem \ref{thmINVABLSHTH}. Thereby, for $s_0 \in \mathcal{P}$:

$$
\lim_{k\to\infty}  \left|\left|\frac{1}{\beta(s+k)}\right|\right|_{\mathcal{N},m} =\frac{1}{|\omega|}\\
$$

Because all the derivatives of $\beta(s+k) \to 0$ as $k\to\infty$.\\

The more difficult case is when $s_0 \in \mathcal{B}$, the weak Julia set. We essentially just have to show that the norm about $s_0$ diverges. We'll start with $m=0$; and, note that:

$$
\lim_{k\to\infty}\frac{1}{|\beta(s_0 + k)|} \to 0\\
$$

This can be done by the process of elimination. It cannot tend to a nonzero constant, as that would mean $s_0$ is in the weak Fatou set. If this tended to infinity, it would imply there's a fixed point $\exp 0 = 0$; which is clearly impossible.

The difficult part is to prove that this limit exists. But once it exists it must be $0$. But every orbit $\beta(s+k)$ of a neighborhood $\mathcal{N}$ about $s_0$ of the Julia set is dense in the Julia set; and hence each point gets arbitrarily close to an orbit sending to $\infty$. Hence, eventually $\beta(s_0+k)$ gets arbitrarily close to an orbit which sends to infinity. This causes a subsequence to converge. But no subsequence can converge to $0$ or a constant; which means $|\beta(s_0 + k)| \to \infty$. The Riemann Sphere is compact; and therefore bounded; and therefore every subsequence must tend to $\infty$, and therefore the sequence tends to $\infty$. Therefore, so do its derivatives in $s_0$. Therefore:

$$
\lim_{k\to\infty}  \left|\left|\frac{1}{\beta(s+k)}\right|\right|_{\mathcal{N},m} = 0\\
$$
\end{proof}

This gives us a normality theorem which is very valuable to us. Under the norm $||...||_{\mathcal{N},m}$ the process $\tau^n$ is bounded. We are performing a very similar procedure now; as we did at the start of Section \ref{sec13}. We are going to reduce to the linear case. 

We can now invoke the constant:

$$
A_\mu^m = \lim_{k\to\infty} \left|\left|\frac{1}{\mu\beta(s+k)}\right|\right|_{\mathcal{N},m} < \infty\\
$$

Which will play the exact role that $A_\mu$ played; but for a different norm. From Corollary \ref{cor2ASYM2} we know that:

$$
\lim_{k\to\infty} \left|\left|\tau^n(s+k)\right|\right|_{\mathcal{N},m} = 0\\
$$

Because all of the derivatives of $\tau$ in $s_0$ converge to zero. We can now play the game of linearizing the iteration in this norm. Each:

$$
\left|\left|\tau^{n+1}(s+k)\right|\right|_{\mathcal{N},m} \le \left|\left|\log\left(1+\frac{\tau^n(s+k+1)}{\beta(s+k+1)}\right)-\log(1+e^{-\lambda(s+k)})\right|\right|_{\mathcal{N},m}\\
$$

Where,

$$
\left|\left|\log\left(1+\frac{\tau^n(s+k+1)}{\beta(s+k+1)}\right)\right|\right|_{\mathcal{N},m} \le (A_\mu^m+\delta)\left|\left|\tau^n(s+k+1)\right|\right|_{\mathcal{N},m}\\
$$

So that we can reduce to the linearization:

\begin{align*}
    \left|\left|\tau^{n+1}(s+k)\right|\right|_{\mathcal{N},m} &\le (A_\mu^m+\delta)\left|\left|\tau^n(s+k+1)\right|\right|_{\mathcal{N},m}+(1+\delta)\left|\left|e^{-\lambda(s+k)}\right|\right|_{\mathcal{N},m}\\
    &\vdots\\
    &\le (1+\delta)\left|\left|e^{-\lambda(s+k)}\right|\right|_{\mathcal{N},m} \sum_{j=0}^n (A_\mu^m+\delta)^je^{-\lambda j}\\
\end{align*}

And we're given the following lemma:

\begin{lemma}
For all $(s_0,\lambda) \in \mathbb{L}$ such that $\Re \lambda > \log A^m_\mu$; for all $n$ and all $k>K$, for some large $K$:

$$
||\tau^n(s+k)||_{\mathcal{N},m} \le M\\
$$

For some $M \in \mathbb{R}^+$ and $\mathcal{N}$ an arbitrarily small neighborhood about $s_0$.
\end{lemma}

From this lemma, the result appears more obvious. We are still given the contraction mapping:

$$
\left|\left|\log(1+a(s)) - \log(1+b(s))\right|\right|_{\mathcal{N},m} \le \frac{1+\delta}{|\mu|}||a(s)-b(s)||_{\mathcal{N},m}\\
$$

For $a,b \approx 0$. And from this we can state the Asymptotic Theorem of the third kind.

\begin{theorem}[The Asymptotic Theorem of the third kind]\label{thmASYM3}
For all $m \in \mathbb{N}$ and all $(s_0,\lambda) \in \mathbb{L}$--setting $\iota_m$ about $s_0 + k$. For all $\Re \lambda > \log A_\mu^m$, there exists $K$ such for $k>K$:

$$
\lim_{n\to\infty} \tau^n(s+k) = \tau(s+k)\,\,\text{under the}\,\,||...||_{\mathcal{N},m}\,\,\text{norm}\\
$$

Where $\mathcal{N}$ is a neighborhood about $s_0$; and:

$$
\lim_{k\to\infty} \left|\left|\frac{1}{\mu \beta(s+k)}\right|\right|_{\mathcal{N},m} = A_\mu^m\\
$$
\end{theorem}

\begin{proof}
This theorem is a weakening of The Convergence Theorem \ref{thmABLCVG} for $A_\mu^m \neq 0$; therefore we only need to prove the case of $A_\mu^m = 0$. Beginning with the inequality:

\begin{align*}
    ||\tau^{n+1}(s+k)- \tau^n(s+k)||_{\mathcal{N},m} &= \left|\left|\log\left(1+\frac{\tau^n(s+k+1)}{\beta(s+k+1)}\right) - \log\left(1+\frac{\tau^{n-1}(s+k+1)}{\beta(s+k+1)}\right)\right|\right|_{\mathcal{N},m}\\
    &\le (A_\mu^m + \delta) ||\tau^{n}(s+k+1)- \tau^{n-1}(s+k+1)||_{\mathcal{N},m}\\
    &\vdots\\
    &\le (A_\mu^m + \delta)^n ||\log(1+e^{-\lambda(s+k+n)})||_{\mathcal{N},m}\\
\end{align*}

We have to justify this inequality; and the result follows. What is meant by this inequality is a bit of a word salad. Each of these norms are about a Taylor polynomial around $s_0 + k$. Each of the terms are multiplied by $e^{-\lambda n}$; and that's all we care about. Giving us:

$$
||\log(1+e^{-\lambda(s+k+n)})||_{\mathcal{N},m} \le Ce^{-\Re \lambda n}\\
$$

Now, $\delta > 0$ can be made as small as possible by increasing $k>K$ as large as we want. Choose a $K$ such that $\Re\lambda > \log(A_\mu^m + \delta)$ (which always exists); then set $0<r = (A_\mu^m + \delta)e^{-\Re \lambda} < 1$:

$$
 ||\tau^{n+1}(s+k)- \tau^n(s+k)||_{\mathcal{N},m} \le Cr^n\\
$$

Thereby:

$$
||\tau^l(s+k) - \tau^n(s+k)||_{\mathcal{N},m} \le \sum_{j=n}^{l-1} Cr^j\\
$$

Setting:

$$
\sum_{j=n}^{l-1} Cr^j < \epsilon\\
$$

Arbitrarily small for $n,l > N$; completes the proof.
\end{proof}

This theorem provides the quick corollary:

\begin{corollary}
For $\mu > 1/e$ and $\lambda \in \mathbb{R}^+$ the function:

$$
F(x) = \beta(x) + \tau(x)\\
$$

For

$$
e^{\mu F(x)} = \exp F(x) = F(x+1)\\
$$

Is infinitely differentiable for $x > x_0$ where $F(x_0+1) = 0$; but it's nowhere holomorphic. 
\end{corollary}

\begin{proof}
Since the norms $||a||_m$ and $||a||_{\mathcal{N},m}$ are equivalent; convergence under either is equivalent. Therefore all the derivatives up to $m$ of $\tau^n(x)$ converge; since $m$ is arbitrary, all derivatives converge. Since $F$ is real valued and monotone and divergent as $x\to\infty$, there's a point $x_0+N$ such that $F(x_0+N) = \exp^{\circ N-1}(0)$ (recalling in this case $\exp^{\circ N-1}(0)$ gets arbitrarily large); which means $F(x_0 + 1) = 0$; where $F(x_0) = \log(0) = -\infty$. This function is nowhere holomorphic because every $x>x_0$ is in the weak Julia set.
\end{proof}

\chapter{The Elliptic View of Iterated Exponentials}\label{chp5}

\section{A fundamental domain of $\beta$}

\begin{remark}
This is where we are going to play gloves off. Much of this work will be far more complicated than earlier work. This section requires a strong understanding of complex analysis, and additionally, complex dynamics. We will try to be as forward and explanatory as possible; but we aren't going to hold your hand. Additionally; this section serves as supplementary commentary on our discussion. For this reason, we will not argue too extensively; instead we will be focused on fleshing out deeper avenues of research.
\end{remark}

The function $\beta$ has a much tighter relationship to iterated exponentials; and its dynamics; than we've discussed. But first, it's helpful to reduce the size of the domain we're working with. Such, it is possible to find a single  simply connected domain, in which $\beta_{\lambda,\mu}$ is fully determined. The function $\beta$ has a fundamental domain of:

$$
\mathcal{H} = \{s \in \mathbb{C}\,|\,0 \le \Re(s) < 1,\,0 \le \Im(s) < |\Im( 2 \pi i/\lambda)|\}\\
$$

Such the entire complex plane is constructed from iterations of the following identities:

\begin{align*}
    \beta(s\pm2\pi i / \lambda) &= \beta(s)\\
    \beta(s+1) &= \frac{\exp \beta(s)}{1+e^{-\lambda s}} = \frac{e^{\mu\beta(s)}}{1+e^{-\lambda s}}\\
    \beta(s-1) &= \log(\beta(s)) + \log(1+e^{-\lambda(s-1)})\\
\end{align*}

We have described how this object converges; but we've yet to really comment on its geometric structure when related to $e^{\mu z}$. Alors, to study the $\beta$ function, we need only look at $\beta(s) : \mathcal{H} \to \mathbb{C}/\{0\}$. But this fundamental domain is not the domain we are interested in exactly. We would like to shear this domain by $\lambda$.

We will begin by describing the best fundamental domain possible. This begins by taking the points: $1+\pi i/\lambda$, $1-\pi i/\lambda$, $-\pi i/\lambda$, $\pi i/\lambda$, and draw the parallelogram connecting them. We will call this domain $\mathcal{Q}$.

Now under the generation of the above iterative formulas $\mathcal{Q} = \mathcal{H}$; of which, both of these domains are fundamental domains. And every point in the orbit of $\mathcal{Q}$ equals every point in the orbit of $\mathcal{H}$, and both cover the orbit of $\mathbb{C}$.

We can refer to the tetrations these functions induce based solely on the domain $\mathcal{Q}$. The weak Julia set and weak Fatou sets can be reduced to their behaviour on $\mathcal{Q}$. Such, $\tau$ is holomorphic or asymptotic on domains in $\mathcal{Q}$.

We are going to look at the tetration function induced on $\mathcal{Q}$, and we are going to construct a function on the torus from it. To do this, we're going to have to look at Kneser's tetration briefly; and look at regular iteration of tetration. This produces what is usually called the $\theta$-mapping.

\section{The $\theta$-mapping in general situations}

Let us fix a base of the exponential $e^{\mu z}$; and let's assume there exists two tetration functions $F_1$ and $F_2$, such that $F_1 \neq F_2$. Let's additionally assume that we can construct an inverse function to $F_1$, and let's call it $S_1$. Then, what is known as a theta mapping is:

$$
\theta(s) = S_1(F_2(s)) - s\\
$$

Which is a periodic function, satisfying:

$$
\theta(s+1) = \theta(s)\\
$$

This mapping is used to measure the difference between various tetrations; and is considered a manner of measuring how close two tetrations are to each other. This can be quite laborious, considering that $S_1$ can be very misbehaved. But what is helpful, is that we need only invert a single tetration to construct a $\theta$ mapping.

This can be done rather easily if we stick to using regular iterations. For that reason, our discussion of $\theta$ terms will take two forms. The first form being when $b = e^{\mu}$ is in the interior of the Shell-Thron region; and when $b >e^{1/e}$. We won't talk about the other possible Abel functions that can be used to create our $\theta$ mapping. Though, in truth, a $\theta$ mapping is a mapping between two tetrations, we will focus only on the Abel function induced from regular iteration.

For this conversation, we will need Koenig's Linearization Theorem. A proof can be found in Milnor \cite{milnor_2000}.

\begin{theorem}[Koenig's Linearization Theorem]\label{thmKOE}
Let $f$ be a holomorphic function with a fixed point $z_0$ such $\gamma = f'(z_0)$ and:

$$
0 < |\gamma| < 1\\
$$

Then there exists a unique function $\phi: \mathcal{A} \to \mathbb{C}$ such $\phi'(z_0) = 1$ and:

$$
\phi(f(z)) = \gamma \phi(z)\\
$$

Where $\mathcal{A}$ is the immediate basin of attraction of the fixed point $z_0$.
\end{theorem}

Some quick notes on this theorem are needed; specifically, a definition of the immediate basin of attraction is in order. This is simply defined as:

\begin{definition}
The immediate basin of attraction $\mathcal{A}$ of $z_0$ is the maximal connected domain about the fixed point $z_0$ such:

$$
\lim_{n\to\infty} f^{\circ n}(z_0) = z_0\\
$$
\end{definition}

We can point to Devaney \cite{devaney_2021} for a proof that this domain is simply connected; despite being rather chaotic on the boundary. This domain will play the role of the domain of our Abel function. From this function, we can construct an Abel function.

If we take:

$$
\alpha(z) = \ln \phi(z)/\ln(\gamma)
$$

Then,

$$
\alpha(f(z)) = \alpha(z) + 1\\
$$

This function will be holomorphic almost everywhere on $\mathcal{A}/\{z_0\}$. To be precise, it will be holomorphic on $\mathcal{A}$ excluding a branch cut spawned from the point $\phi(z_0) = 0$. This won't matter much to us; and is largely a sigh of relief because it's a much easier branch cut then what we're used to so far in this paper. And as our function will avoid the fixed point excepting at infinity; it will be largely negligible.

Now, if we assume that $\alpha$ is an Abel function for $f(z) = e^{\mu z}$; then any tetration function $F(s)$ constructs a $\theta$ function:

$$
\theta(s) = \alpha (F(s)) - s\\
$$

What's particularly interesting is when $F$ has a period. By which, let us assume that $F(s + 2\pi i / \lambda) = F(s)$. Then, we must have $\theta$ satisfying:

\begin{align*}
    \theta(s+1) = \theta(s)\\
    \theta(s+2\pi i/\lambda) = \theta(s)-2\pi i/\lambda\\
\end{align*}

Which implies it is periodic, and ``nearly'' linear. This instantly tells us there must be either singularities, or branch cuts, or domains cut from its fundamental domain. That fundamental domain is again $\mathcal{Q}$. Understanding how $\theta$ behaves, gives us a very quick manner of understanding how $F$ is behaved.

\section{Treating $\theta$ mappings up to elliptic mappings}

The point of this section is to talk about the permissible periods of $F_\lambda(s)$ for $e^\mu$ in the Shell-Thron region; and comparing how unique they are. As we've found from The Inverse Abel Function on the Shell-Thron region \ref{thmINVABLSHTH}; we know that a holomorphic function can be constructed on the weak Fatou set, so long as $\Re \lambda > -\ln|\mu \omega|$; $\omega$ is the attracting fixed point. We know that $0 <|\mu \omega| < 1$, so the values of $\lambda$ where our construction fails are a significant vertical strip.

To begin the foray, we can create the fundamental domain $\mathcal{Q}^* \subset \mathcal{Q}$ such that, $F_\lambda$ is holomorphic on $\mathcal{Q}^*$ up to branch cuts. This domain is precisely the weak Fatou set $\mathcal{P}$ and its intersection with $\mathcal{Q}$. Thereby, $\mathcal{Q}^* = \mathcal{P} \cap \mathcal{Q}$. Thereby,

$$
\theta(s) = \alpha(F_\lambda(s)) - s\\
$$

Is holomorphic on $\mathcal{Q}^*$; but since this is $1$-periodic, and eventually the further right we go the branch cuts of $F$ disappear, we must have that $\theta$ is holomorphic on $\mathcal{Q}^*$; and there are no branch cuts induced by $F$. We recall that $\alpha$ is the Abel function of $e^\mu$ constructed through the regular iteration method.

We wish to show two things in this section. The first being; a natural boundary appears about the point $\lambda \to - \ln (\mu \omega)$; which invokes a period $2\pi i /\lambda = -2\pi i/\ln(\mu \omega)$. This is the period of the regular iteration. Thereupon the goal is; it's impossible to stretch the width of our fundamental domain further from this.

This means the regular iteration serves as a sort of boundary to how much the period can be stretched. It's helpful to think of this with the case $\mu = \ln(2)/2$, $b = \sqrt{2}$. To get a look of the regular iteration; we turn to Sheldon Levenstein's Fatou.gp \cite{levenstein_2011}. A graph of the regular iteration is given in Figure \ref{fig:reg_iter_root_two}.

We've also provided a graph of $F_\lambda$ that is very close to the regular iteration. The value of the period of the regular iteration is $-2 \pi i/\ln\ln(2) = 17.143148i$. The value of the period we will display in the next figure, is $16i$. This is the shrinking of the fundamental domain; which causes some fractal errors; but retains much of the same picture. This is seen in Figure \ref{fig:approx_reg_iter_root_two}.

\begin{figure}
    \centering
    \includegraphics[scale = 0.3]{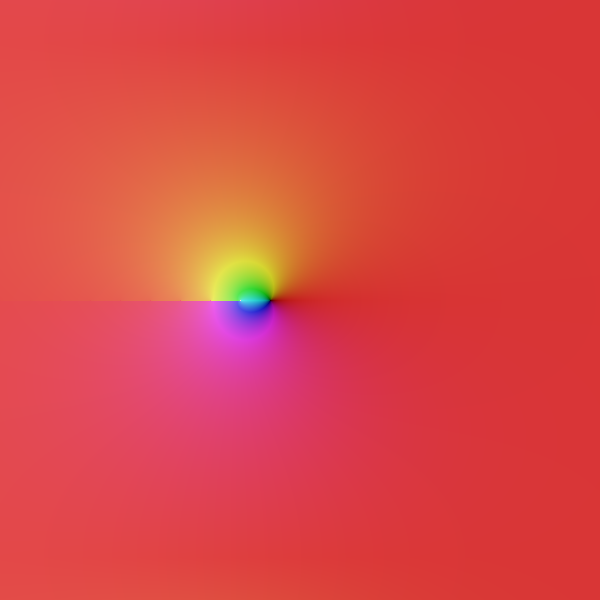}
    \caption{This is a graph of the regular iteration for $b = \sqrt{2}$. This means it's the function $\text{tet}_{\sqrt{2}}(s) = \phi^{-1}(e^{\ln\ln(2) (s-s_0)})$; with $\text{tet}_{\sqrt{2}}(0) = 1$ and $\text{tet}_{\sqrt{2}}(s+1) = \sqrt{2}^{\text{tet}_{\sqrt{2}}(s)}$; while $\text{tet}_{\sqrt{2}}$ is real-valued and $2\pi i/\ln(\ln(2))$ periodic. This function is graphed over the domain $|\Re(s)|,|\Im(s)| < 10$. This was drawn using Levenstein's Fatou.gp \cite{levenstein_2011}; but I tweaked the protocols a bit.}
    \label{fig:reg_iter_root_two}
\end{figure}

\begin{figure}
    \centering
    \includegraphics[scale = 0.35]{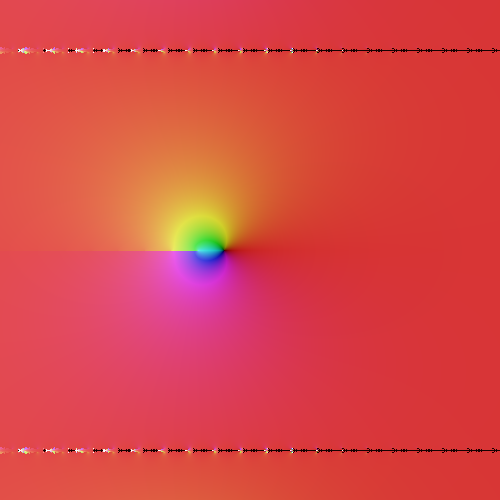}
    \caption{This is a graph of the beta method $F_\lambda(s) = \beta_\lambda(s) + \tau(s)$ for $\mu = \log(2)/2$. We have set $\lambda = 0.392699082$ which is just a bit away from $-\ln\ln(2) = 0.366512920$--the natural border. We have normalized this function, so $F_\lambda(0) = 1$, and is graphed over the domain $|\Re(s)|,|\Im(s)| < 10$}
    \label{fig:approx_reg_iter_root_two}
\end{figure}

The regular iteration for this base, has a period of $-2 \pi i/\ln\ln2$. Then what we are saying is that the width of our fundamental domain cannot exceed the width of this fundamental domain. Where the width is simply the imaginary part of the period (in this case it's purely imaginary). By which, we can think of the regular iteration as the limit of the various $F_\lambda$. 

This is a very difficult concept to cognize; but it is a surprisingly simple argument when you know what you are doing. By, of course, if we relinquish ourselves to $\theta$ mappings; it's a statement on the maximum width of our parallelogram $\mathcal{Q}^*$. Thereso, the manner this statement works is, 
$\Im\left(2\pi i/\lambda\right) = \Im\left(2 \pi i/\ln(\mu \omega)\right)$ serves as a natural boundary to all elliptic $\theta$ mappings.\\

To begin, we provide an alternative representation of the $\theta$ mapping.

\begin{theorem}\label{thmTHTBETA}
The $\theta$-mapping holomorphic on $\mathcal{Q}^*$ can be alternatively represented as:

$$
\theta(s) = \lim_{n\to\infty}\alpha(\beta(s+n))  -s - n\\
$$
\end{theorem}

\begin{proof}
The function $\theta$ is $1$-periodic and initially described as:

$$
\theta(s) = \alpha(F(s))-s\\
$$

O\`{u}, 

\begin{align*}
    \theta(s+1) &=  \alpha(F(s+1))-s-1\\
    &=\alpha\left(e^{\mu F(s)}\right) -s -1\\
    &= \alpha(F(s)) + 1 -s - 1\\
    &= \alpha(F(s)) - s = \theta(s)\\
\end{align*}

And thereby:

$$
\theta(s) = \alpha(F(s+n)) - s - n\\
$$

But $F(s+n) = \beta(s+n) + \tau(s+n)$. The function $\tau(s+n) \to 0$; and does so geometrically like $\mathcal{O}(e^{-\lambda n})$ on compact subsets. Upon, nous voyons,

$$
||\alpha(F(s+n)) - \alpha(\beta(s+n))||_{\mathcal{N}} \le C|e^{-\lambda n}|\\
$$

For any compact neighborhood $\mathcal{N} \subset \mathcal{Q}^*$. The function $\alpha(s)$ grows at worst like $\mathcal{O}(1/(s-\omega))$ as $s \to \omega$ (we reference Milnor for this tidbit \cite{milnor_2000}; but it isn't a difficult result to intuit). Any Abel function grows at worst linearly; this is comparable to the statement $\log(z) = \mathcal{O}(z)$.
\end{proof}

Because of this, we've shown an alternative representation for the $\beta$ tetrations. We summarize in the next theorem.

\begin{theorem}
Let $F_\lambda(s)$ be a tetration induced from $\beta$, for $\Re \lambda > -\log |\mu \omega|$ and $b = e^\mu \in \mathfrak{S}^o$ (the interior of the Shell-Thron region). Let $\text{tet}_b(s)$ be the regular tetration base $b$, and let $\alpha$ be its Abel function. Then, defining:

$$
\theta(s) = \lim_{n\to\infty} \alpha (\beta(s+n)) - s - n\\
$$

We must have:

$$
F_\lambda(s) = \text{tet}_b(s + \theta(s))\\
$$
\end{theorem}

This means we can fully construct $F_\lambda$ with pretty much no mention of $\tau$ or $\rho$; we just have to construct $\theta$ mappings. This begs the question, how unique is our theta mapping? The answer is surprisingly simple; they are equivalent up to the addition of an elliptic function.

Take two $\theta$ and $\widetilde{\theta}$ such that:

\begin{align*}
    \theta(s+1) = \theta(s)\,\,\,&\text{and}\,\,\,\widetilde{\theta}(s+1) = \widetilde{\theta}(s)\\
    \theta(s+ 2\pi i/\lambda) = \theta(s) - 2\pi i/\lambda\,\,\,&\text{and}\,\,\,\widetilde{\theta}(s+ 2\pi i/\lambda) = \widetilde{\theta}(s) - 2\pi i/\lambda\\
\end{align*}

Therefore, if we take the difference of these two functions:

$$
p(s) = \theta(s) - \widetilde{\theta}(s)\\
$$

We must have that:

\begin{align*}
    p(s+1) &= p(s)\\
    p(s+2\pi i/\lambda) &= p(s)\\
\end{align*}

There are two cases here, where $p$ is what we typically call an elliptic function; and the case that $p$ is only holomorphic on the parallelogram $\mathcal{Q}$ with some deleted domains. Both cases are equally interesting. We will leave this discussion for the next section; but we've reduced the uniqueness of our $\theta$ mapping into a question of elliptic functions.

\section{Arguing for the uniqueness of the $\beta$ method}

We won't be using any deep details of elliptic functions excepting some basic properties. We will point to the introductory text by Ahlfors which handles everything we will use in this section \cite{ahlfors_2021}. The two cases we have to handle in this section are when $p$ is your run of the mill elliptic function; and when $p$ is not your run of the mill elliptic function; meaning it has domains deleted from its parallelogram of holomorphy (It's holomorphic on $\mathcal{Q}^* \subset \mathcal{Q}$ as opposed to being meromorphic on $\mathcal{Q}$).

The first instance, when $p$ is an elliptic function in the traditional sense, is very easy to handle. To explain this, we will stay with the base case of $F_\lambda$ induced by the beta method; and ask what happens if we perturb the $\theta$ mapping by an elliptic function. This means that:

\begin{align*}
    \widetilde{\theta}(s) &= \theta(s) + p(s)\\
    \widetilde{F}_\lambda(s) &= \text{tet}_b(s+\widetilde{\theta}(s))\\
    \widetilde{F}_\lambda(s) &= \text{tet}_b(s+\theta(s) + p(s))\\
\end{align*}

Now, without loss of generality, we can assume that $p(s)$ has all its singularities contained within the parallelogram $\mathcal{Q}$ (if not, just shift the fundamental parallelogram). This means that $\widetilde{F}$ will have ``more'' singularities than $F$. But it's difficult to quantize this intuition. 

We want to be able to determine whether a $\theta$ mapping has a elliptic function attached to it, or not; while just looking at $\theta$. The idea is fairly simple; since $p$ is meromorphic, if it has a singularity at $s^*$, then $\widetilde{F}(s^*)$ has an essential singularity. Since it has an essential singularity, it attains every point in the complex plane near it; but then we must have a whole lot of branching; as $\text{tet}_b(s)$ has a big branch cut along $(-\infty,-2]$. So, by direct comparison $\widetilde{F}$ is holomorphic on a smaller domain than $F$! If $\widetilde{F}$ has an essential singularity, it must have a non-constant elliptic function attached to its $\theta$ mapping.

So, if we have a tetration function $F$ induced by $\beta$; any tetration created by perturbing the $\theta$ mapping by an elliptic mapping $p$ must necessarily be holomorphic on a smaller domain than $F$. And, this tetration function must have an essential singularity.\\

Now, if $p$ is not a traditional elliptic mapping, things are a tad more complicated, as we need not have singularities; or $p$ itself may have essential singularities. If it is defined on the whole parallelogram but with essential singularities (it isn't meromorphic), the above argument still applies. Thus, it suffices to consider $p$ defined on the parallelogram with deleted points/domains. 

We are going to assume a result here and save the proof for the next section. We will assume something using Lebesgue's theory of integration. Let $\mathcal{Q}$ be the parallelogram, and let $\mathcal{Q}^*$ be the parallelogram that the beta methods $\theta$ mapping is holomorphic on. Then,

$$
\int_{\mathcal{Q}^*/\mathcal{Q}} \,dA = 0\\
$$

This is equivalent to the statement that:

$$
\int_{\mathcal{B}}\,dA = 0\\
$$

Which, is the statement that the weak Julia set is measure zero under an area mapping. This is where it is absolutely pivotal that we assume that $b = e^\mu \in \mathfrak{S}^o$. This is a result which only holds for the Shell-Thron region. Therefore, we can begin by only considering $\widetilde{\theta}$ holomorphic almost everywhere on $\mathcal{Q}$ (up to branch cuts). So what we are qualifying with adding an elliptic mapping, is moving the branch cuts around.

This is where it becomes very difficult to determine why the beta method is any better than any other tetration method which is equally holomorphic on $\mathcal{Q}$ almost everywhere. We have very little data in ways of identifying the beta method from any other method which produces a periodic tetration of the same period. Thus, enters the key.

The function $\theta$ induced by the beta method, has its singularities exactly at the four corners of $\mathcal{Q}$. If we can qualify the behaviour of these singularities, we have the uniqueness of the beta method in terms of $\theta$ mappings.

The last thing we have to do, before completing this proof; is proving that any tetration function $F$ with a different period than the regular iteration, cannot be holomorphic in the right half-plane. I.e: there must be branch-cuts or singularities. This result will take us away from everything we've done in this paper, and is actually very simple to show with the right tools. We will also put this in coming sections; as it's a tad technical.

To summarize the goal of the next section:\\

If $\widetilde{\theta}(s)$ is $\theta$ upto a traditionally non-constant elliptic function we will have singularities outside of the singularities at the four corners of the parallelogram. If $\widetilde{\theta}(s)$ has branch cuts on the parallelogram and singularities outside of the singularities at the four corners, we will have more singularities than $\theta$. If the only singularities of $\widetilde{\theta}$ are at the four corners of $\mathcal{Q}$; and the singularities are ``no worse'' than the singularities of $\theta$, then $\widetilde{\theta} = \theta$.

\section{Comparing the $\beta$ method to a $\widetilde{\beta}$ method}

We are going to start by putting a wall of equations and relations; and how our functions interact:

\begin{align*}
    F(s) &= \text{tet}_b(s+\theta(s))\\
    \widetilde{F}(s) &= \text{tet}_b(s+\widetilde{\theta}(s))\\
    &\\
    \theta(s) &= \widetilde{\theta}(s) + p(s)\\
    p\,\,&\text{is doubly periodic (almost elliptic)}\\
    p(0)&=0\\
    &\\
    \theta(s) &= \lim_{n\to\infty} \alpha(\beta(s+n)) - s - n\\
    \widetilde{\theta}(s) &= \lim_{n\to\infty} \alpha(\beta(s+n)) - s - n - p(s)\\
    &= \lim_{n\to\infty} \alpha(\widetilde{\beta}(s+n)) - s - n\\
    &\\
    F(s) &= \beta(s) + \tau(s)\\
    \widetilde{F}(s) &= \widetilde{\beta}(s) + \widetilde{\tau}(s)\\
    &\\
    \beta(s) &=\beta(s+2\pi i/\lambda)\\
    \widetilde{\beta}(s)  &= \widetilde{\beta}(s+2\pi i/\lambda)\\
\end{align*}

This gives us access to a $\widetilde{\beta}$ function from the function $\widetilde{\theta}$. If we can create a uniqueness criterion for $\beta$; it will transfer over as a uniqueness condition for $\theta$. The thing we're looking for is; $\beta$ is holomorphic for $\Re(\lambda s) < 1$, and has simple poles at $s=1+(2k+1)\pi i/\lambda$ with residue $\frac{1}{\lambda}e^{\mu\beta(\pi i /\lambda)}$. This defines a unique function by mapping $\Re(\lambda s) < 1.5$ to the unit disk, and treating the zeroes of $\frac{1}{\beta(s)}$; then, any function with the same poles $\widetilde{\beta}(s) = \beta(s)e^{g(s)}$. But then, the residues are not the same; they don't equal $e^{\mu\widetilde{\beta}(\pi i / \lambda)}/\lambda$. This uniquely characterizes the $\beta$ function.

\section{Tying up the loose ends}

There are two results we need to solidify before we can state that the $\theta$ mapping induced by $\beta$ is unique. The first result is that every periodic solution to the tetration equation which is not the period of the regular iteration must have singularities in the right half plane other than the singularities of the regular iteration. Let us assume everything is normalized, and then:

$$
\theta(s) = \alpha(F(s)) - s\\
$$

Where we assume that $F$ has no singularities for large $\Re(s)$. Where additionally, as we've assumed normalization, $\theta(0) = \theta(1) = 0$. The function $\theta(s) = \mathcal{O}(s)$; and therefore applying Ramanujan's master theorem $\theta(s) = 0$.  This implies that $F$ is identically the regular iteration. For our purposes, we need only reference Ramanujan's theorem as, if $f(s) = \mathcal{O}(s)$ and is holomorphic in the right half plane; if $f\big{|}_{\mathbb{N}} = 0$, then $f = 0$. Typically this is done with functions of exponential type, but that's a slight overkill for us. We reference an introductory journal article for this \cite{amdeberhan_2012}.

The second thing we need to clear up is that the weak Julia set $\mathcal{B}$ when $e^\mu = b \in \mathfrak{S}^o$ is measure zero under an area mapping. For this, we detract from an actual proof and reference Devaney \cite{devaney_2021}; and his analysis of the exponential function for $0 < \mu < 1/e$; which extends exactly for everywhere in the Shell-Thron region. This is never exactly stated in Devaney; but the Julia set of $e^{\mu z}$ for the Shell-Thron region is all the preimages of the Devaney hairs which tend to infinity.

For example, in the case of $b = \sqrt{2}$; if we take all the preimages of $(4,\infty)$, we have the Julia set. Since this is a measure zero set; so are a countable union of preimages which are measure zero. So the Julia set is measure zero. Since the weak Julia set is a remapping of the Julia set where $\beta \in \mathcal{J}$; the actual weak Julia set is measure zero under an area mapping.

\section{Collecting Results}

From here, we know that the $\theta$ mapping is unique. Collecting how and why it's unique, is difficult. But, if the reader has been paying attention this chapter, it's pretty straight forward. The $\beta$ function induces a $\theta$ mapping. If the $\theta$ mapping is moved, it induces a different $\beta$ function. The $\beta$ function is unique if we ask that:

\begin{align*}
    \beta &\,\,\text{is holomorphic for}\,\, \Re(\lambda s) < 1\\
    \beta&\,\,\text{has a period}\,\,2\pi i/\lambda\\
    \beta&\,\,\text{has simple poles at}\,\, 1 + (2k+1)\pi i/\lambda\,\,\text{for}\,\,k\in\mathbb{Z}\\
    \beta(s) &= \frac{e^{\mu\beta(\pi i/\lambda)}/\lambda}{s-1-(2k+1)\pi i/\lambda}+h(s)\,\,\,\,h\,\text{locally holomorphic}\\
    \log \beta(s+1) - \beta(s) &= \mathcal{O}(e^{-\lambda s})\\ 
\end{align*}

And the $\theta$ function is unique if we ask that its singularities are at the corners of the parallelogram $\mathcal{Q}$; and that the $\beta$ function induced from $\theta$ is as above. This settles the Shell-Thron region (or at least its interior). The very important complex analytic fact, which is implicit in this solution:

$$
\beta(s) = \sum_{k\in\mathbb{Z}} \left(\frac{e^{\mu\beta(\pi i/\lambda)}/\lambda}{s-1-(2k+1)\pi i/\lambda} - p_k(s)\right) + H(s)\\
$$

For a holomorphic function $H$ for $\Re(s)<2$; and polynomials $p_k = \mathcal{O}(1/k)$. This is largely what insures $\beta$ is unique. We can reduce this; by doing some manipulations. Firstly $s\mapsto \pi i s/\lambda +1$:

$$
\beta(\pi i s/\lambda +1) = \frac{e^{\mu\beta(\pi i/\lambda)}}{\pi i}\sum_{k\in\mathbb{Z}} \left(\frac{1}{s-2k-1}-b_k(s)\right) + B(s)\\
$$

Now, we can map $s\mapsto s+1$ and:

$$
\beta(\pi i (s+1)/\lambda +1) = \frac{e^{\mu\beta(\pi i/\lambda)}}{\pi i}\sum_{k\in\mathbb{Z}} \left(\frac{2}{\frac{s}{2}-k}+\frac{2}{k}\right) + A(s)\\
$$

Using the expansion for $\cot(s)$:

$$
\frac{1}{\pi}\cot(\pi s) = \frac{1}{s} + \sum_{k\in \mathbb{Z}/\{0\}}\frac{1}{s-k}+\frac{1}{k}\\
$$

This reduces exactly to:

$$
\beta(\pi i (s+1)/\lambda +1) = -2ie^{\mu\beta(\pi i/\lambda)} \cot(\pi s/2) + A(s)
$$

Where, notably $A(s+2) = A(s)$. The uniqueness condition we are imploring is that $\beta$ satisfies this equation; and it is the only function to do so if we require $\log \beta(s+1) - \beta(s)$ tends to zero geometrically; which is required for any $\beta$ function to induce a $\theta$ mapping. It's important to recognize that this is a recursive equation. We are requiring the function to satisfy a residue equation with a different value of the function. This is the real heavy lifting of our theorem. We are checking whether $\beta$ induced by $\theta$ has this expansion. If not, it's not the correct $\theta$.

Producing this residue is rather simple too; for any contour $\mathcal{C}$ about the singularity $1+(2k+1)\pi i / \lambda$:

$$
\int_{\mathcal{C}} \beta(s)\,ds = \int_{\mathcal{C}}\frac{e^{\mu\beta(s-1)}}{1+e^{-\lambda(s-1)}}\,ds = \frac{2\pi i}{\lambda} e^{\mu \beta(\pi i /\lambda)}\\
$$

Due to the simple pole of the denominator with residue $1/\lambda$; and the holomorphy of the numerator. If this isn't self-explaining; I don't know what to say. The result then, is this uniquely defines $\beta$.\\

As an author, I have not chosen to solidify these results into a theorem. The result of these discussions are intended as a blue print; and if they are not produced perfectly; they still display the general shape of the $\theta$ mappings of $F_\lambda$. This ends our discussion of $\theta$ mappings for the $\beta$ method on the Shell-Thron region.

\section{Closing with the $b>e^{1/e}$ case}

We are beginning to close this report with a very different subject. This subject is much deeper at home to Levenstein's work on tetration--and the $\theta$ mapping philosophy of the tetration forum \cite{tetration_forum}. Much of this work was epitomized in Levenstein's Fatou.gp program. The mathematical work hasn't been written as of yet; but it is in the process of being written by him and I. There exists an exact dichotomy between this case, and the cases handled for the most of this paper. It's precisely the question, how can we make a holomorphic inverse Abel function on the Julia set?

As we've described over and over again in this paper; the $\beta$ method is suited exactly for Fatou sets; and is at best an asymptotic expansion for Julia sets. The trouble is, for $b > e^{1/e}$; the Julia set is the entire complex plane. This can be found in Devaney \cite{devaney_2021}; but is a well known result; which is largely why it makes sense to just refer to base $b = e$; and call it a day. Although the exponentials $b^z$ are not conjugate similar to $e^z$; when it comes to iteration, the mantra is everything you can do with $e^z$--you can do with $b^z$ when $b>e^{1/e}$.

With that being said, the central philosophy of constructing a tetration from $e^z$ is to use a $1$-periodic $\theta$ mapping. As Levenstein likes to say, the $\theta$ mappings I just talked about are very different from these $\theta$ mappings. These $\theta$ mappings need an application of Schwarz's reflection principle to ever have a hope of being a full tetration.

Whereas; when $b = \sqrt{2}$, for example:

\begin{align*}
\theta(s) &= \sum_{k=-\infty}^\infty c_k e^{-2\pi i k s}\\
\overline{c_k} &= c_{-k}\\
\theta &: \mathbb{R} \to \mathbb{R}\\
\end{align*}

And then, using regular iteration $F(s) = \text{tet}_{\sqrt{2}}(s+\theta(s))$. A very different result happens for base $b=e$. Since there is no fixed point in the orbit of $\exp$ about $0$; we have to choose a different iteration base. This is done primarily by focusing on $L\approx 0.31813 + 1.33723i$--which is the fixed point of $\exp$ with minimal imaginary part. This fixed point is repelling, and consequently there does exist an iteration $\Psi(e^{Ls})$ which satisfies the tetration equation. En plus, $\Psi(0) = L$ and $\Psi'(0) = 1$.

The trouble is, this will never be real valued. The only way it can be real valued is if you can find a real valued path and map it to the real line. This can be done for the upper half plane; with a boundary of the real line. Additionally it will be unique. This produces a very different kind of $\theta$ mapping:

\begin{align*}
    \theta(s) &= \sum_{k=0}^\infty d_k e^{2\pi i k s}\\
    \theta\,\,&\text{is holomorphic for}\,\,\Im(s) > 0\\
    \lim_{\Im(s) \to \infty}\theta(s) &= d_0\\
    \lim_{\Im(s) \to 0}\Psi(e^{L(s+\theta(s))}) &\in \mathbb{R}\,\,\text{when}\,\,\Re(s) > -2\\
\end{align*}

Then, to extend to the lower half plane, we can use Schwarz's reflection principle. Which is equivalent to doing the exact same thing I've just done here, but with the fixed point $\overline{L} \approx 0.31813 - 1.33723i$. Then, thank the heavens, we are still analytic on $\mathbb{R}$; and we are pasting two analytic functions together.

It is well conjectured this is the only real valued tetration we get from $e^z$. Thereby, because these $\theta$ mappings are inherently different from the elliptic-\textit{like} $\theta$ mappings, there is no way to ``correct" the $\beta$ function (which is periodic) into Kneser's method. This underscores the deepness of The Asymptotic Theorem of the third kind \ref{thmASYM3}. We are trying to force a periodic solution to this iteration; which makes an elliptic-\textit{like} $\theta$ mapping; and this is impossible.

This leads to the only way that the $\beta$ method can become holomorphic for $b>e^{1/e}$. We must remove the elliptic aspect; there can be no period. The way the author refers to this, is as a double limit. If, as we iterate $\tau^n$, we additionally let $\lambda \to 0$; then $2 \pi i/\lambda \to \infty$; and the period disappears. Upon which, the only hope in hell we have of making the $\beta$ method work for $b>e^{1/e}$ is to work precisely with the limit:

$$
\lim_{n\to\infty,\,\lambda \to 0} \log^{\circ n} \beta(s+n)\\
$$

It is beyond the scope of this paper, but the author has much of the jigsaw discovered. The hidden trick is to ensure that $\lambda = \mathcal{O}(n^{-\delta})$ for $0 < \delta < 1$. Which should directly produce Kneser's tetration. The difficulty in even approaching this, is that any program calculating $\beta$ for $\Re \lambda < 0.001$ is doomed to fail. Again, to even attempt a numerical solution would require a supercomputer. But the author is confident in this result; despite the inability to empirically justify it.

This is largely why the author has avoided talking about the case $b>e^{1/e}$ in a holomorphic respect. The code is expected to fail; but the result is still prominent. With that statement, we nod towards what is so incredible about Levenstein's Fatou.gp \cite{levenstein_2011}; it uses the bare amount of data points to produce Kneser's tetration. The author feels in no manner that the $\beta$ method is feasible (at least computationally) for these cases. Regardless, the $\beta$ method is still a pretty good approximation. It can be made to work; but it is not as natural, or in anyway as convenient as Levenstein's approach. Why use the language of all these $\beta$'s when we can just find $\theta$ and leave it at that.

\chapter*{Appendix}

\section*{A}

This portion of the appendix is intended to give a rough explanation of the program \texttt{beta.gp} included with this report. All the phase-plot graphs in this paper were produced using this program (excepting one graph using \texttt{fatou.gp}--Figure \ref{fig:approx_reg_iter_root_two}). The program is based off of a recursive philosophy; where much of the entire program is written using recursion. It is, as one would say, a formal program that runs perfectly; but is limited by the confines of the computer which runs it. Where its greatest enemy is memory overflow errors. This is something that is unavoidable, as tetration is ingrained in the study of large numbers (numbers too large for a desktop computer's memory).

This program is programmed using the language PARI; and executed through the shell PARI-GP; or the source can be compiled into a C program, which runs noticeably faster. The author chose PARI for its handling of large numbers, that far surpasses any competitor. This is the language of the tetration forum for the most part, because it allows for values of the order \texttt{exp(1E6)}, if you have the memory capacity. It's an odd fit; as PARI is more at home amongst p-adic analysis; but the ideas go hand in hand. PARI can hold an obscene amount of Taylor series/polynomial data that is perfect for tetration--let alone the size and detail of a number it can hold. 

Built into the program \texttt{beta.gp} is a rudimentary graphing protocol \texttt{MakeGraph}; which accounts for all the phase-plot graphs we have included in this paper. This function protocol was written by Mike3 of the tetration forum; and is the only piece of code not written by Levenstein, or myself, in this program. Most of the code in \texttt{beta.gp} was written by myself; but Levenstein offered a couple of optimization techniques which are invaluable. 

This program allows us to analyse: the $\beta$ function, the inverse Abel from $\beta$, the $\tau$ function, the $\rho$ functions, and a bunch of weak Fatou/weak Julia set protocols. Half heartedly coded in are some normalization protocols, but they aren't needed. Additionally we have added alternative initializations which allow us to analyse the $\mu = 1E100$ or higher cases. The analysis of these cases are out of the scope of this paper; but their construction still exists. The code of \texttt{beta.gp} requires different initializations for such extremes, though (and be prepared to eat memory like the cookie monster).

One is able to grab Taylor series from any function involved; so calculating \texttt{beta(1+z)}, or \texttt{tau(1+z)}, will produce a polynomial (to the desired series precision) of $\beta$, or $\tau$, about $1$. Equally, we are constrained by the digit precision declared at the beginning of the program. This can be used to grab the asymptotic series as well; and works identically in either case.

This program is not without its flaws, but it is quite literally an exact translation of the math within this paper. Where things like $\OmSum$ are translated directly into a recursive protocol; and $\tau^n$ similarly. The major difference between \texttt{beta.gp} and a formal algorithm of the math with in this paper; there is a training wheels protocol, which initiates a breaking sequence. The training wheels protocol is done solely if we are about to hit the capacity of PARI; otherwise this is just a formal program where 90\% of the time it will fail because $\beta$ can grow too large sporadically and we'll have a memory error. So there is a predictive if statement which catches before we'll even get close to having an error; this is the only unnatural part of our code. By unnatural, I mean it's not a perfect recursive program. Nonetheless, it is modeled exactly after a perfect recursive algorithm.

The catch statement is precisely if \texttt{x>exp(1E6)}, quit. This is all we've added as an artificial construction to save face within \texttt{PARI}. 

I have included a large amount of documentation with the program \texttt{beta.gp}. It explains more clearly the exact function calls, how to initialize values, how to understand what each variable/function means. I have included a \texttt{readme.txt}; which runs through the purpose of each function. I have also included a bloated amount of commentary within \texttt{beta.gp} itself.

The program \texttt{beta.gp} is included alongside this paper to demonstrate how everything we've talked about is reproducible. This is the scientific end of this paper; we want everything to be clear, but reproducible. For more detailed use of this program (creating more dynamic types of graph); it is possible to attach this program to \texttt{SageMath}--which has a plethora of graphing protocols not used in this paper, but helpful for visualizations.

\section*{B}

This portion of the appendix is a break down of Sheldon Levenstein's analysis of the case $\mu =1$ and $\lambda =1$. This case is particularly interesting, as the beta function has a base $b= e$ and a period $2 \pi i$. This means if $\tau$ were holomorphic, we'd construct a $2\pi i$ periodic real valued tetration base $e$. As we know, this doesn't happen; $\tau$ is not holomorphic. The author had originally thought he'd proven that $\tau$ was holomorphic. Levenstein, however, provided a large amount of counter evidence, which eventually led towards The Asymptotic Theorem of the third kind \ref{thmASYM3}. Where we miss holomorphy by a hair's breadth. Levenstein employed some genius observations.

These comments begin as such:

\begin{displayquote}
The $\beta$ function appears very well behaved in the complex plane, and matches to precision a Taylor series generated by sampling a unit circle, and also matches the function's iterative definition.   The results are accurate to 112 decimal digits with 240 sample points.

The $\tau$ function does not match its Taylor series generated by sampling around a unit circle with 32 sample points, except at those 32 sample points.  
\end{displayquote}

What Levenstein is saying here, in so many words, Cauchy's Integral Formula is failing. Sampling along the boundary of a disk, is causing, using the generated taylor series of $\tau$ at 32 points, a divergence when we try to use these points to evaluate on the interior of said disk. As we increase sample points; we yet again see further and further divergence; where, if $\tau$ were holomorphic we'd see some semblance of convergence. 

This is essentially the basis of a Cauchy type argument. As we express $\tau$ on the border of the disk (using as many sample points, which are really sample polynomials), they don't reflect the behaviour of $\tau$ within the disk in a holomorphic manner. Where Cauchy's integral formula would guarantee this process to work. As so well it does work with $\beta$ itself.

Levenstein then goes on to single out the points $1+e^{2.3826i}$ and $1+e^{2.3827i}$; where along the arc connecting them $\tau$ has a singularity. He posts a graph describing this phenomena in Figure \ref{fig:Shel_sing}. This throws a wrench in the gears, because further iterations only go to sharpen these spikes, and additionally create more distortions in the nice looking waveform we have. This is the germ of the idea that $\tau$ converges point wise; is absolutely great as an asymptotic (like $\tau^9$ is in this graph); but as we iterate further, we experience more jagged discontinuous behaviour. The resolution of this problem, is the eventual Asymptotic Theorem of the third kind.

\begin{figure}
    \centering
    \includegraphics[scale=0.25]{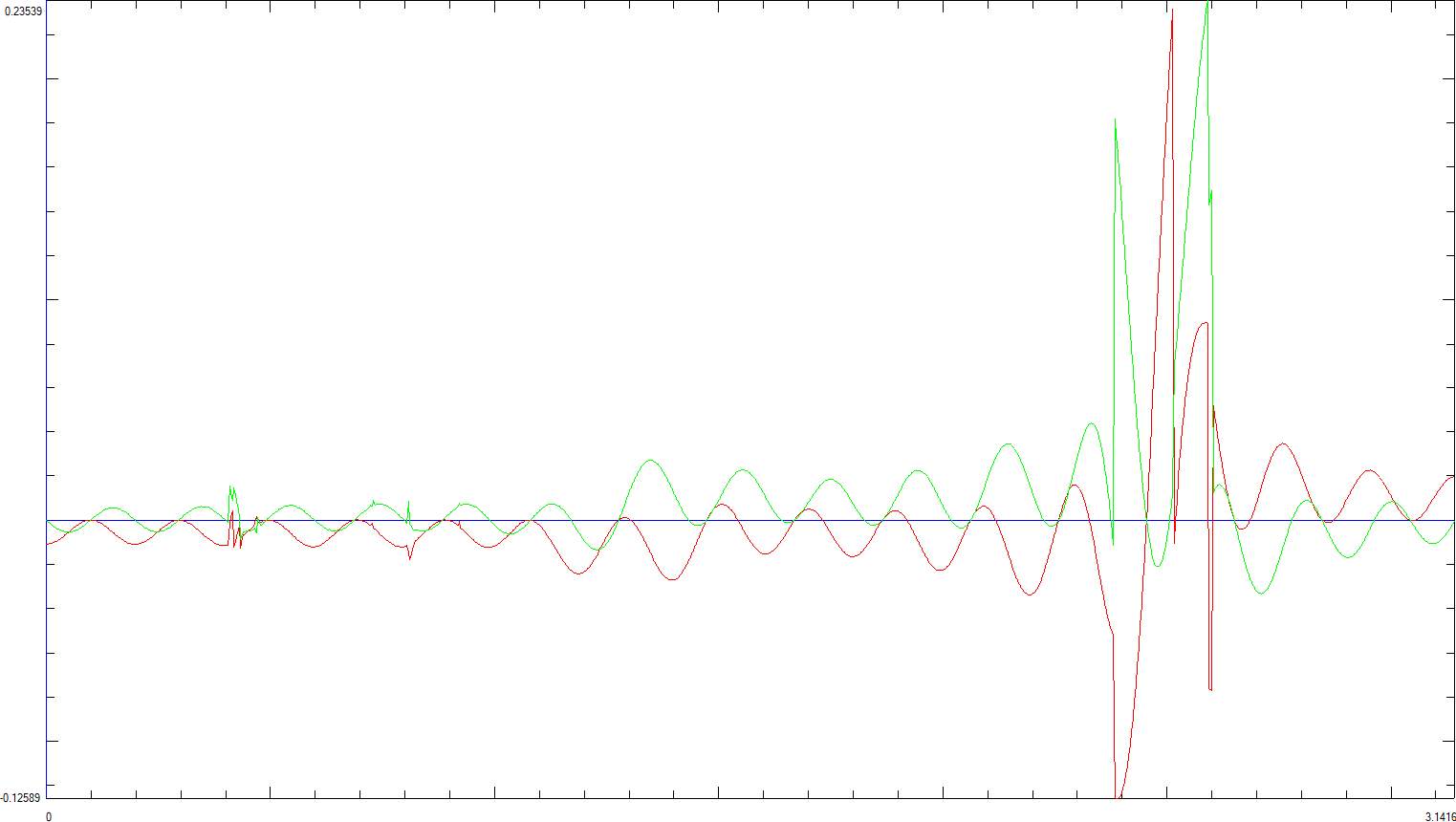}
    \caption{Levenstein's graph displaying the singularity which occurs for $\tau^{n}$ for $n\approx 9$. }
    \label{fig:Shel_sing}
\end{figure}

Levenstein then points his attention to the real line. This begins by introducing the $\rho$ terms; which we dubbed Levenstein's form of the recursion:

\begin{align*}
\tau^n(s) &= \sum_{j=1}^n \rho^j(s)\\
\rho^{n+1}(s) &= \ln\left(1+\frac{\rho^n(s+1)}{\beta(s+1) + \sum_{j=1}^{n-1}\rho^j(s+1)}\right)\\
\end{align*}

The real line becomes incredibly difficult because $\beta(6)$ is about where Pari-gp overflows $\beta(6) \approx 3.49302404\times 10^{48220669901186881}$. This means we are limited when talking about $\rho$. For $0 \le \Re(s) \le 1$ we can only reach $\rho^5$. So it is very difficult to draw a conclusion on an infinite amount of $\rho$ terms from a sample size of $5$. 

The value $\rho^0 = 0$, and we start our chain of equations with $\rho^1 = -\ln(1+e^{-s})$. Levenstein observed, that as we add $\rho^2$, $\rho^3$, $\rho^4$, and $\rho^5$--we converge ridiculously fast pointwise. But something very anomalous happens with the Taylor coefficients. The odd thing being, they happen very very far out. To notice this we can look at $\beta(s) + \rho^1(s)$ and see that it hits zero; and that a clustering of zeroes start to happen near zero. This is displayed in Figure \ref{fig:Shel_sing2}

\begin{figure}
    \centering
    \includegraphics[scale=0.45]{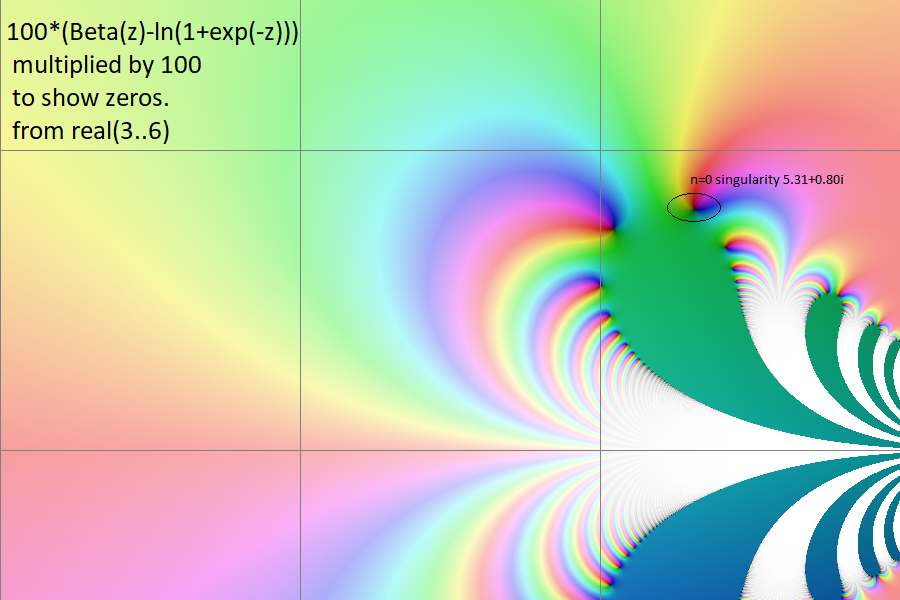}
    \caption{A zoom in of the zeroes which start to appear.}
    \label{fig:Shel_sing2}
\end{figure}

This hints directly that our Taylor series on $\mathbb{R}$ is going to be tricky to prove convergence. And our gut should say that it isn't going to converge at this point. As soon as we hit $\rho^3$ we are going to feel the effect of these zeroes. In Figure \ref{fig:Shel_sing2}, on the border of the white and pale green petals are a countable number of singularities which cluster towards $\mathbb{R}$.

Levenstein, then refers to the Taylor expansions in a neighborhood of the interval $[-1,1]$. Specifically trying to test; using polynomials:

\begin{align*}
    \widetilde{\tau}^5(x) &= \sum_{k=0}^{1000} \tau^{5,(k)}(0) \frac{x^k}{k!}\,\,\text{for}\,\, -1 \le x \le 1\\
    \beta(x) + \widetilde{\tau}^5(x) &\approx \exp(\beta(x-1)+\widetilde{\tau}^5(x-1))\,\,\text{for}\,\,0 \le x \le 1\\
    \beta(x+1) + \widetilde{\tau}^5(x+1) &\approx \exp(\beta(x)+\widetilde{\tau}^5(x))\,\,\text{for}\,\,-1 \le x \le 0\\
\end{align*}

Where, if the Taylor series converges, it must be that these objects are virtually identical with enough terms. It doesn't converge, though. So, if we take the Taylor polynomial of $\tau^5$ about $0$, call it $\widetilde{\tau}^5$, and test $\tau^5(0.5)$ against $\widetilde{\tau}^5(0.5)$; the polynomials begin to disagree rapidly and rabidly as we take further and further terms. Which is to say $\exp(\beta(-0.5) + \widetilde{\tau}(-0.5)) \neq \beta(0.5) +\widetilde{\tau}(0.5)$. Despite the fact, $\exp(\beta(-0.5) + \tau(-0.5)) = \beta(0.5) +\tau(0.5)$. But furthermore, we can map how it gets worse with each iteration.  

\begin{figure}
    \centering
    \includegraphics[scale=0.5]{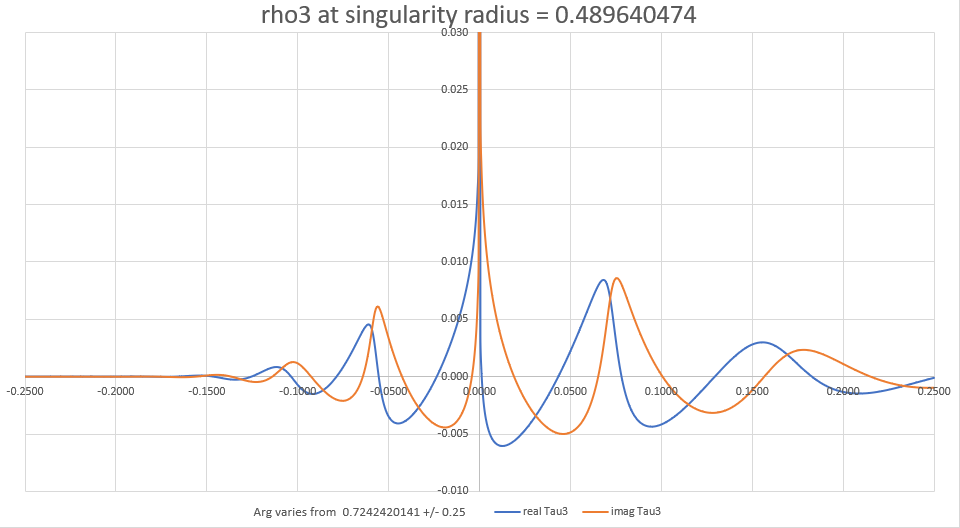}
    \caption{A graph of $\tau^3$, drawn across an arc around a circle of radius $0.489$, about the point $x_0$ such that $\beta(x_0)= 1$. A clear display of singularity-like behaviour arises.}
    \label{fig:Shel_sing3}
\end{figure}

In Figure \ref{fig:Shel_sing3} we begin to see a chaos in the polynomials describing $\tau$ (albeit with $3$ iterations instead of $5$). This means quite clearly, if:

\begin{align*}
    \widetilde{\tau}^n_m(x) &= \sum_{k=0}^m \tau^{n,(k)}(0) \frac{x^k}{k!}\\
    \lim_{m,n\to\infty} \widetilde{\tau}^n_m(x) - \tau(x) &\neq 0\\
\end{align*}

The Asymptotic Theorem of the third kind, ensures that:

$$
\lim_{n\to\infty} \widetilde{\tau}^n_m(x) - \tau_m(x) \,\,\text{converges}\\
$$

It's as we explore further terms of the Taylor polynomials that we start to see more spikes. These spikes cause the series to diverge; despite all the terms converging pointwise. From this, Levenstein made an astute calculation as to where we start to see spikes in the derivatives of $\tau$; they are far off in the $50,000$ terms range. Levenstein came up with a very ingenious way of isolating only far out terms without calculating the lower terms.

In Figure \ref{fig:Shel_sing4}; we can see how very large terms of the Taylor series start to grow faster and faster. First with $\rho^3$, growing in an unbounded way; but $\rho_3$ itself is holomorphic, so this is okay. The function $\rho^4$ is just as holomorphic; but we can see the further out Taylor coefficients start growing faster than $\rho^3$--it's going to be holomorphic on a smaller domain. As we add $\rho^5$, we experience the same phenomena--the further out series coefficients start growing unreasonably. When we add up all the terms $\tau = \sum_{j=1}^\infty \rho^j$, we will see this phenomena causes the far out terms to be unreasonably large, and the power series will not converge.

\begin{figure}
    \centering
    \includegraphics[scale=0.5]{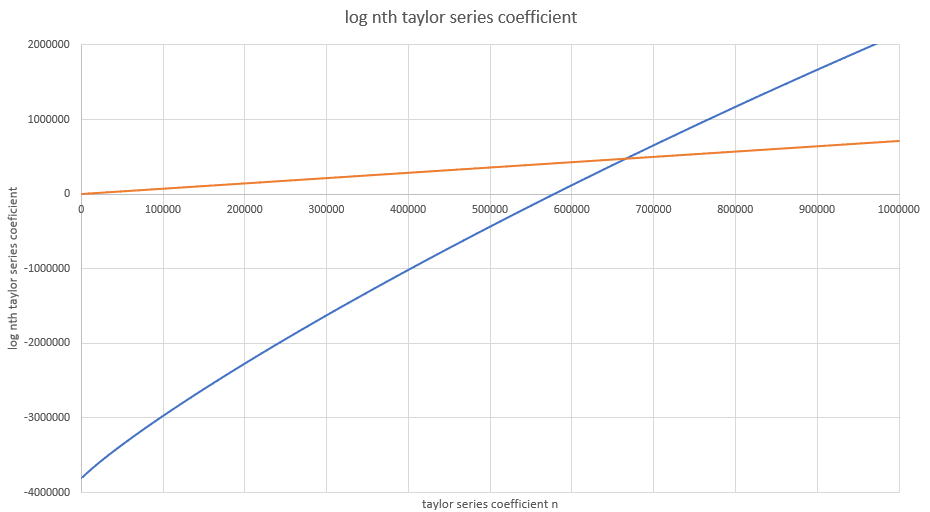}
    \caption{A discrete graph of $\ln \frac{d^k}{dx^k}\rho_3(0)$ (red) and $\ln \frac{d^k}{dx^k}\rho_4(0)$ (blue) from $0 \le k \le 1000000$}
    \label{fig:Shel_sing4}
\end{figure}

Levenstein then does the unthinkable, and produces an unreasonably correct approximation. Let $\ell\rho^1(s) = \ln(\ln(1+e^{-s})) = \ln\left( - \rho^1(s)\right)$; described by the recursion:

\begin{align*}
\ell\rho^n(s) &= \ln\left(-\ln(1-e^{\ell\rho^{n-1}(s+1)-F_{n-1}(s)})\right)\\
F_n(s) &= \beta(s) + \tau^n(s) = \beta(s) + \sum_{j=1}^n \rho^n(s)\\
\end{align*}

This causes us to get:

\begin{align*}
    \ell\rho^n(s) &\approx \ell\rho^{n-1}(s+1) - F_{n-1}(s)\\ 
    \ell\rho^n(s) &\approx \ln \ln \left(1+e^{-s-n}\right) - \sum_{j=1}^n F_{n-1}(s+n-j)\\
\end{align*}

This serves as an awesome approximation, and helps us pull out a faster formula for the further out series terms. From here, we can see that higher $\rho$ terms cause the higher terms of the series to blow up; and not only that--helps us estimate their eventual growth. All in all, this shows that $\tau$ cannot be analytic anywhere on $\mathbb{R}$. This is not shown, per se, as we've only analyzed $\rho^1,\rho^2,\rho^3,\rho^4,\rho^5$; but it provides a lot of solid evidence. 

As to this, I abstracted much of this to derive Levenstein's Theorem \ref{thmLVN}, and using this knowledge, abstracted it to The Generalized Levenstein Theorem \ref{thmGenLev}. We named these theorems for Sheldon Levenstein as they are truly his work. He just didn't know how to compile the results into a proof. Again, as he comes from an engineering background, he'd prefer to hold the calculator in his hand; and can make the calculator. But the proof may escape his train of thought. For that, we call these results as his--especially because he found a hole in my original idea that this can't happen. This entire results hinges on the fact that $||\frac{1}{\beta(s+n)}||_{\mathcal{N}} \not \to 0$ for any compact set $\mathcal{N}\subset \mathbb{C}$, despite the fact $\frac{1}{\beta(s+n)} \to 0$ for each element of $\mathcal{N}$. 

\bibliographystyle{plain}
\bibliography{main}

\end{document}